\documentclass[aos,preprint,reqno]{imsart}
\setattribute{journal}{name}{}
\RequirePackage[OT1]{fontenc}
\RequirePackage{amsthm,amsmath}
\usepackage{mathdots}
\RequirePackage[colorlinks,citecolor=blue,urlcolor=blue]{hyperref}
\allowdisplaybreaks[4]
\usepackage{amssymb}
\usepackage{anysize}
\usepackage{hyperref}
\usepackage{extarrows}
\usepackage{multirow}
\usepackage{enumerate,bm}
\usepackage{graphicx,psfrag,epsfig, subfigure, mathrsfs}
\usepackage{ulem}

\pubyear{\today}
\startlocaldefs
\numberwithin{equation}{section}
\newtheorem{thm}{Theorem}[section]
\newtheorem{lem}[thm]{Lemma}
\newtheorem{pro}[thm]{Proposition}
\newtheorem{cor}[thm]{Corollary}
\newtheorem{defi}[thm]{Definition}
\newtheorem{rem}[thm]{Remark}

\newtheorem{ass}[thm]{Assumption}

\newcommand{\be}{\begin{equation}}
\newcommand{\ee}{\end{equation}}
\newcommand{\bea}{\begin{eqnarray*}}
\newcommand{\eea}{\end{eqnarray*}}
\newcommand{\ntr}{\mathrm{tr}\,}

\makeatletter

\newcommand{\Rmnum}[1]{\expandafter\@slowromancap\romannumeral #1@}
\makeatother
\makeatletter % `@' now normal "letter"
\@addtoreset{equation}{section}
\makeatother  % `@' is restored as "non-letter"

\endlocaldefs

\newcommand{\E}{\mathbb{E}}

\renewcommand{\P}{\mathbb{P}}

\newcommand{\De}{\Delta}
\newcommand{\la}{\lambda}

\newcommand{\de}{\delta}

\newcommand{\ga}{\gamma}

\newcommand{\cH}{{\mathcal H}}

\newcommand{\scM}{{\mathscr M}}

\newcommand{\A}{{\bf A}}

\newcommand{\B}{{\bf B}}

\newcommand{\U}{{\bf U}}

\renewcommand{\(}{\left(}
\renewcommand{\)}{\right)}

\newcommand{\lb}{\label}
\newcommand{\no}{\nonumber\\}

\begin{document}

\begin{frontmatter}
\title{Canonical correlation coefficients of high-dimensional  Gaussian vectors: finite rank case}
\runtitle{High-dimensional CCA}

\begin{aug}
\author{\fnms{Zhigang} \snm{Bao}\thanksref{t1}\ead[label=e1]{mazgbao@ust.hk}},
\author{\fnms{Jiang} \snm{Hu}\thanksref{t4}\ead[label=e4]{huj156@nenu.edu.cn}},
\author{\fnms{Guangming} \snm{Pan}\thanksref{t2}\ead[label=e2]{gmpan@ntu.edu.sg}}
\and
\author{\fnms{Wang} \snm{Zhou}\thanksref{t3}
\ead[label=e3]{stazw@nus.edu.sg}
\ead[label=u1,url]{http://www.sta.nus.edu.sg/~stazw/}}

\thankstext{t1}{Z.G. Bao was  supported by a startup fund from HKUST}
\thankstext{t4}{J. Hu  was partially supported by CNSF 11301063}
\thankstext{t2}{G.M. Pan was partially supported by the Ministry of Education, Singapore, under grant \# ARC 14/11}
\thankstext{t3}{ W. Zhou was partially supported by the Ministry of Education, Singapore, under grant MOE2015-T2-2-039(R-155-000-171-112).}
\runauthor{Z.G. Bao et al.}

\affiliation{Hong Kong University of Science and Technology\\ Northeast Normal
University \\ Nanyang Technological University \\and National University of
 Singapore}

\address{Department of Mathematics,  \\ Hong Kong University of Science and Technology, \\ Hong Kong\\
\printead{e1}\\
\phantom{E-mail:\ }}

\address{KLASMOE and School of Mathematics $\And$ Statistics\\ Northeast Normal
University \\P. R. China 130024\\
\printead{e4}\\
\phantom{E-mail:\ }}

\address{Division of Mathematical Sciences, \\ Nanyang Technological University, \\ Singapore 637371\\
\printead{e2}\\
\phantom{E-mail:\ }}

\address{Department of Statistics and Applied Probability,\\ National University of
 Singapore,\\ Singapore 117546\\
\printead{e3}\\
\printead{u1}}
\end{aug}

\begin{abstract}
Consider a Gaussian vector $\mathbf{z}=(\mathbf{x}',\mathbf{y}')'$, consisting of two sub-vectors $\mathbf{x}$ and $\mathbf{y}$ with dimensions $p$ and $q$ respectively. With $n$ independent observations of $\mathbf{z}$, we study the correlation between $\mathbf{x}$ and $\mathbf{y}$, from the perspective of the Canonical Correlation  Analysis. We investigate the  high-dimensional case: both $p$ and $q$ are proportional to the sample size $n$. Denote by $\Sigma_{uv}$ the population  cross-covariance matrix of  random vectors $\mathbf{u}$ and $\mathbf{v}$, and denote by $S_{uv}$ the sample counterpart. The canonical correlation coefficients between $\mathbf{x}$ and $\mathbf{y}$ are known as the square roots of the nonzero eigenvalues of the canonical correlation matrix $\Sigma_{xx}^{-1}\Sigma_{xy}\Sigma_{yy}^{-1}\Sigma_{yx}$. In this paper, we focus on the case that $\Sigma_{xy}$ is of finite rank $k$, i.e. there are $k$ nonzero canonical correlation coefficients, whose squares are denoted by $r_1\geq\cdots\geq r_k>0$. We study the sample counterparts of $r_i,i=1,\ldots,k$, i.e. the largest $k$ eigenvalues of the sample canonical correlation matrix $ S _{xx}^{-1} S _{xy} S _{yy}^{-1} S _{yx}$, denoted by $\lambda_1\geq\cdots\geq \lambda_k$.  We show that there exists a threshold $r_c\in(0,1)$, such that for each $i\in\{1,\ldots,k\}$, when $r_i\leq r_c$, $\lambda_i$ converges almost surely to the right edge of the limiting spectral distribution of the sample canonical correlation matrix, denoted by $d_{+}$. When $r_i>r_c$, $\lambda_i$ possesses an almost sure limit in $(d_{+},1]$, from which we can recover $r_i$'s in turn,  thus provide an estimate of the latter in the high-dimensional scenario. We also obtain the limiting distribution of  $\lambda_i$'s under appropriate normalization. Specifically, $\lambda_i$ possesses Gaussian type fluctuation if $r_i>r_c$, and follows Tracy-Widom distribution if $r_i<r_c$.  Some applications of our results are also discussed.
\end{abstract}

\begin{keyword}[class=MSC]
\kwd{62H20, 60B20,60F99}
\end{keyword}

\begin{keyword}
\kwd{Canonical correlation analysis}
\kwd{ Random Matrices}
\kwd{ MANOVA ensemble}
\kwd{ High-dimensional data}
\kwd{ finite rank perturbation}
\kwd{ largest eigenvalues}
\end{keyword}

\end{frontmatter}
\section{Introduction}
In multivariate analysis,  the most general and favorable method to investigate the relationship between two  random vectors $\mathbf{x}$ and $\mathbf{y}$, is the Canonical Correlation Analysis (CCA), which was raised in the seminal work of Hotelling \cite{Hotelling1936}. CCA aims at finding two sets of basis vectors,  such that the correlations between the projections of the variables  $\mathbf{x}$ and $\mathbf{y}$ onto these basis vectors are mutually maximized, namely,
 seeking vectors $\mathbf{a}=\mathbf{a}_1$ and $\mathbf{b}=\mathbf{b}_1$ to maximize the correlation coefficient
\begin{eqnarray}
\rho\equiv\rho(\mathbf{a},\mathbf{b}):=\frac{\text{Cov}(\mathbf{a}'\mathbf{x},\mathbf{b}'\mathbf{y})}{\sqrt{\text{Var}(\mathbf{a}'\mathbf{x})}\cdot\sqrt{\text{Var}(\mathbf{b}'\mathbf{y})}}. \label{17050701}
\end{eqnarray}
Conventionally,  $\rho_1:=\rho(\mathbf{a}_1,\mathbf{b}_1)$ is called the {\it{first canonical correlation coefficient}}. Having obtained the first $m$ canonical correlation coefficients $\rho_i,i=1,\ldots,m$ and the corresponding vector pairs $(\mathbf{a}_i,\mathbf{b}_i),i=1\ldots,m$, one can proceed to seek vectors $(\mathbf{a}_{m+1},\mathbf{b}_{m+1})$ maximizing $\rho$ subject to the constraint that $(\mathbf{a}_{m+1}'\mathbf{x},\mathbf{b}_{m+1}'\mathbf{y})$ is uncorrelated with $(\mathbf{a}'_i\mathbf{x},\mathbf{b}_{i}'\mathbf{y})$ for all $i=1,\ldots,m$. Analogously, we call $\rho_i$ the {\it{$i$th canonical correlation coefficient}} if it is nonzero. Denoting by $\Sigma_{uv}$ the population cross-covariance matrix of arbitrary two random vectors $\mathbf{u}$ and $\mathbf{v}$, it is well known that
$r_i:=\rho_i^2$ is the $i$th largest eigenvalue of the  {\it{ (population) canonical correlation matrix}} \[\Sigma_{xx}^{-1}\Sigma_{xy}\Sigma_{yy}^{-1}\Sigma_{yx}.\] Let  $\mathbf{z}_i=(\mathbf{x}_i',\mathbf{y}_i')', i=1,\ldots,n$ be  $n$ independent observations of the vector $\mathbf{z}:=(\mathbf{x}',\mathbf{y}')'\sim N(\boldsymbol{\mu},\Sigma)$ with mean vector $\boldsymbol{\mu}$ and covariance matrix
\begin{eqnarray*}
\Sigma=\left(\begin{array}{cc}
\Sigma_{xx} &\Sigma_{xy}\\
\Sigma_{yx} &\Sigma_{yy}
\end{array}\right),
\end{eqnarray*}
We can study the canonical correlation coefficients via their sample counterparts. To be specific, we employ the notation $S_{uv}$ to represent the sample cross-covariance matrix for arbitrary two random vectors $\mathbf{u}$ and $\mathbf{v}$, where the implicit sample size of $(\mathbf{u}',\mathbf{v}')'$ is assumed to be $n$, henceforth. Then the square of the $i$th sample canonical correlation coefficient is defined as the $i$th largest eigenvalue of  the {\it{sample canonical correlation matrix}} (CCA matrix in short) \[ S _{xx}^{-1} S _{xy} S _{yy}^{-1} S _{yx},\]
denoted by $\lambda_i$ in the sequel.

Let $p$ and $q$ be the dimensions of the sub-vectors $\mathbf{x}$ and $\mathbf{y}$, respectively. In the classical low-dimensional setting, i.e., both $p$ and $q$ are fixed but $n$ is large, one can safely use $\lambda_i$ to estimate $r_i$, considering the convergence of the sample cross-covariance matrices towards their population counterparts. However, nowadays, due to the increasing demand in the analysis of high-dimensional data springing up in various fields such as genomics, signal processing, microarray, finance and proteomics, putting forward a theory on high-dimensional CCA is much needed.    So far, there are only a handful of works  devoted to this topic.
 Fujikoshi in \cite{Fujikoshi1} derived
 the asymptotic distributions of the
  canonical correlation coefficients when $q$ is fixed while $p$ is proportional to $n$.
    Oda et al. in \cite{Fujikoshi2} considered the problem of   testing  for redundancy
in high dimensional canonical correlation analysis. Recently, with certain sparsity assumption,  the theoretical  results and potential applications of high-dimensional sparse CCA have been discussed in \cite{GaoM2014S,GaoM2015M}.   In the null case, i.e., $\mathbf{x}$ and $\mathbf{y}$ are independent, the Tracy-Widom law for the largest canonical correlation coefficients has been studied in \cite{Johnstone2008, HanP16T, HPY}, when $p, q, n$ are proportional.    Recently,  in \cite{JO15} Johnstone and Onatski derived the asymptotics of the likelihood ratio processes of  CCA corresponding to the null hypothesis of no spikes and the alternative of a single spike.

In this paper, we will work with the following high-dimensional setting.
\begin{ass}[On the dimensions] \label{ass.070301} We assume that
$p:=p(n)$, $q:=q(n)$, and
\begin{eqnarray*}
p/n= c_1\to y_1\in(0,1),\quad q/n= c_2\to y_2\in (0,1),\quad  \text{as}\quad n\to \infty,  \quad \text{s.t.} \quad y_1+y_2\in (0,1).
\end{eqnarray*}
Without loss of generality, we always work with the additional assumption
\begin{eqnarray*}
p> q, \quad \text{thus} \quad c_1> c_2.
\end{eqnarray*}
\end{ass}
 Observe that here $y_1$ and $y_2$ are asymptotic parameters, while $c_1$ and $c_2$ are non-asymptotic parameters. Hence, in general,  a $c_1, c_2$-dependent random variable $X(c_1,c_2)$ can not serve as a limiting target of a random sequence $X_n(c_1,c_2), n\geq 1$.  Nevertheless, to ease the presentation, from time to time, we still write $X_n(c_1, c_2)\to X(c_1,c_2)$,  if $X_n(c_1,c_2)-X(c_1,c_2)\to 0$, where the convergence could be in distribution,  in probability or a.s., etc.

Let $\bar{\mathbf{x}}$ and $\bar{\mathbf{y}}$ be the sample means of $n$ samples $\{\mathbf{x}_i\}_{i=1}^n$ and $\{\mathbf{y}_i\}_{i=1}^n$ respectively, and use the notation $\mathring{\mathbf{x}}_i:=\mathbf{x}_i-\bar{\mathbf{x}}$ and $\mathring{\mathbf{y}}_i:=\mathbf{y}_i-\bar{\mathbf{y}}$ for $i=1,\ldots,n$.  We can then write
\begin{eqnarray*}
 S _{ab}=\frac{1}{n-1}\sum_{i=1}^n\mathring{\mathbf{a}}_i\mathring{\mathbf{b}}_i',\qquad \mathbf{a}, \mathbf{b}=\mathbf{x} \text{  or   }\mathbf{y}
\end{eqnarray*}
It is well known that there exist $n-1$ i.i.d. Gaussian vectors $\tilde{\mathbf{z}}_i=(\tilde{\mathbf{x}}_i', \tilde{\mathbf{y}}_i')'
\sim N(\mathbf{0},\Sigma),
$
such that
\begin{eqnarray*}
 S _{ab}=\frac{1}{n-1}\sum_{i=1}^{n-1}\tilde{\mathbf{a}}_i\tilde{\mathbf{b}}'_i,\qquad \mathbf{a}, \mathbf{b}=\mathbf{x} \text{  or   }\mathbf{y}.
\end{eqnarray*}
For simplicity, we recycle the notation $\mathbf{x}_i$ and $\mathbf{y}_i$ to replace $\tilde{\mathbf{x}}_i$ and $\tilde{\mathbf{y}}_i$,  and work with $n$ instead of $n-1$, noticing that such a replacement on sample size is harmless to Assumption \ref{ass.070301}. Hence, we can and do assume that $\mathbf{z}$ is centered in the sequel and denote
\begin{align*}
S_{ab}=\frac{1}{n} \sum_{i=1}^n \mathbf{a}_i\mathbf{b}_i',\qquad \mathbf{a}, \mathbf{b}=\mathbf{x}\;  \text{or} \; \mathbf{y}.
\end{align*}
 Furthermore, for brevity, we introduce the notation
\begin{align}
 C_{xy}:= S _{xx}^{-1} S _{xy} S _{yy}^{-1} S _{yx}, \label{17031001}
\end{align}
for the CCA matrix, and $ C_{yx}$ can be analogously defined via switching the roles of $x$ and $y$ in (\ref{17031001}). Notice that  $ C_{xy}$ and $ C_{yx}$ possess the same non-zero eigenvalues.

By our assumption $p>q$, there are at most $q$ non-zero canonical correlations, either population ones or sample ones. An elementary fact is that $\lambda_{i},r_i\in[0,1]$ for all $i=1,\ldots, q$. Note that $\lambda_i,i=1,\ldots,q$ are also eigenvalues of the $q \times q$ matrix $ C_{yx}$, whose empirical spectral distribution (ESD) will be denoted by
\begin{align}
F_n(x):=\frac{1}{q}\sum_{i=1}^q\mathbf{1}_{\{\lambda_i\leq x\}}. \label{17032801}
\end{align}
We use the following notations for the data matrices
\begin{eqnarray*}
\mathscr{X}:=\mathscr{X}_n=(\mathbf{x}_1,\ldots,\mathbf{x}_n),\quad \mathscr{Y}:=\mathscr{Y}_n=(\mathbf{y}_1,\ldots, \mathbf{y}_n).
\end{eqnarray*}
 Let $\text{vec}(\mathscr{X})=(\mathbf{x}_1',\ldots,\mathbf{x}_n')'$ be the vectorization of the matrix $\mathscr{X}$, and define $\text{vec}(\mathscr{Y})$ be the analogously. We see that $\text{vec}(\mathscr{X})\sim N(\mathbf{0},  I_n\otimes \Sigma_{xx})$ and $\text{vec}(\mathscr{Y})\sim N(\mathbf{0},  I_n\otimes \Sigma_{yy})$.
Our aim, in this work, is to study the asymptotic behavior of a few   largest sample canonical correlation coefficients $\sqrt{\lambda}_i$'s, and try to get the information about the population ones $\rho_i=\sqrt{r_i}$ from the sample ones. We will focus on the case of finite rank, i.e.,   there is some fixed nonnegative integer $k$, such that
\begin{eqnarray*}
r_1\geq\ldots\geq r_k\geq r_{k+1}=\ldots=r_q=0.
\end{eqnarray*}
Specifically, we make the following assumption throughout the work.
\begin{ass}[On the rank of the population matrix] \label{ass.070302} We assume that $\mathrm{rank}(\Sigma_{xy})\leq k$ for some fixed positive integer $k$. Furthermore, setting $r_0=1$, we denote by $k_0$ the nonnegative integer satisfying
\begin{eqnarray}1=r_0\geq  \ldots \geq r_{k_0}> r_c\geq r_{k_{0}+1}\geq \ldots r_k\geq  r_{k+1}=0, \label{062806}
\end{eqnarray}
where
\begin{eqnarray}
r_c\equiv r_c(c_1,c_2):=\sqrt{\frac{c_1c_2}{(1-c_1)(1-c_2)}}. \label{062805}
\end{eqnarray}
\end{ass}
In Section 1.2 we will state our main results. Before that, we introduce in Section 1.1 some known results in the {\it{null case}}, i.e. $k=0$, which will be the starting point of our discussion.
\subsection{The null case: MANOVA ensemble}\label{section. null monova} At first, we introduce some known results on the limiting behavior of $\{\lambda_i\}_{i=1}^q$ in the null case,
i.e. $\mathbf{x}$ and $\mathbf{y}$ are independent, or else,  $r_i=0$ for all $i=1,\ldots,q$. It is elementary to see that the canonical correlation coefficients are invariant under the block diagonal transformation $(\mathbf{x}_i,\mathbf{y}_i)\to(\A\mathbf{x}_i,\B\mathbf{y}_i)$, for any $p\times p$ matrix $\A$ and $q\times q$ matrix $\B$, as long as both of them are nonsingular. Hence, without loss of generality,  in this section, we tentatively assume that $\Sigma_{xx}= I_p$ and $\Sigma_{yy}= I_q$. Under our high-dimensional setting, i.e. Assumption \ref{ass.070301}, it is known that $\lambda_i$'s do not converge to $0$ even in the null case, instead, they typically spread out over an interval contained in $[0,1]$. Specifically, we have the following theorem on $F_n(x)$ (c.f. (\ref{17032801})), which is essentially due to Wachter \cite{Wachter1980}.
\begin{thm} \label{thm.070301}When $\mathbf{x}$ and $\mathbf{y}$ are independent Gaussian and Assumption \ref{ass.070301} holds, almost surely, $F_n$ converges weakly to a deterministic probability distribution $F(x)$ with density
\begin{align}
f(x)=\frac{1}{2\pi c_2}\frac{\sqrt{(d_{+}-x)(x-d_{-})}}{x(1-x)}\mathbf{1}(d_{-}\leq x\leq d_{+}), \label{17040310}
\end{align}
where

	\begin{align}
	d_{\pm}=\big(\sqrt{c_1(1-c_2)}\pm \sqrt{c_2(1-c_1)}\big)^2 \label{17031301}
	\end{align}

 \end{thm}
 \begin{rem} In the null case, the convergence of the ESD of the CCA matrix actually holds under a more general distribution assumption, see \cite{YP2012}.
 \end{rem}
 Conventionally, we call $F(x)$ in Theorem \ref{thm.070301} the limiting spectral distribution (LSD) of $ C_{yx}$.
 One might note that $F(x)$ is also the LSD of the so-called MANOVA ensemble with appropriately chosen parameters, which is widely studied in the Random Matrix Theory (RMT). Actually, when $\mathbf{x}$ and $\mathbf{y}$ are Gaussian and independent, the CCA matrix $ C_{yx}$ is exactly a MANOVA matrix. To see this, we note that  $ P_x:=\mathscr{X}'(\mathscr{X}\mathscr{X}')^{-1}\mathscr{X}$ is a projection matrix independent of $ Y$. Hence, we can write
 \begin{eqnarray*}
 C_{yx}=(\mathscr{Y}( I- P_x)\mathscr{Y}'+\mathscr{Y} P_x\mathscr{Y}')^{-1}\mathscr{Y} P_x\mathscr{Y}'.
 \end{eqnarray*}
Using Cochran's theorem, we see that $\mathscr{Y}( I- P_y)\mathscr{Y}'$ and $\mathscr{Y} P_x\mathscr{Y}'$ are independent Wishart,
 \begin{eqnarray*}
\mathscr{Y}( I- P_x)\mathscr{Y}'\sim \mathcal{W}_q( I_q, n-p),\quad \mathscr{Y} P_x\mathscr{Y}'\sim\mathcal{W}_q( I_q, p).
 \end{eqnarray*}
 Hereafter, we use the notation $\mathcal{W}_m(\cdot, \cdot)$ to denote the Wishart distribution $\text{Wishart}_m(\cdot, \cdot)$ for short.
Consequently, $\lambda_i,i=1,\ldots,q$ are known to possess the following joint density function,
 \begin{eqnarray*}
 p_n(\lambda_1,\ldots,\lambda_q)=C_n \prod_{i<j}^q|\lambda_i-\lambda_j|\prod_{i=1}^q(1-\lambda_i)^{(n-p-q-1)/2}\lambda_i^{(p-q-1)/2}\mathbf{1}(\lambda_{i}\in [0,1]),
  \end{eqnarray*}
 where $C_n$ is the normalizing constant, see Muirhead \cite{Muirhead1982}, page 112, for instance. Or else, one can refer to \cite{Johnstone2008}, for more related discussions.  In the context of RMT, the point process possessing the above joint density is also called Jacobi ensemble.

 Throughout the paper, we will say that an $n$-dependent event $A\equiv A(n)$ holds with {\it{overwhelming probability}}, if  for any given positive number $\ell$, there exists a constant $C_\ell$ such that
 \begin{eqnarray*}
 \mathbb{P}(A)\geq 1-C_\ell n^{-\ell}.
 \end{eqnarray*}
  Especially, for any fixed integer $K\geq 0$, we have $\cap_{i=1}^{n^K} A_i$ holds with overwhelming probability if $A_i$ holds with overwhelming probability individually with the common $C_\ell$'s.
  The next  known result concerns the  convergence of the largest eigenvalues.
 \begin{thm} \label{thm.071501} When $\mathbf{x}$ and $\mathbf{y}$ are independent and Assumption \ref{ass.070301} holds, we have
 \begin{eqnarray}
 \lambda_i-d_{+}\stackrel{\text{a.s.}} \longrightarrow 0, \quad  \label{071501}
 \end{eqnarray}
 for any fixed positive integer $i$. Moreover, for any small constant $\varepsilon>0$,
 \begin{eqnarray}
\lambda_1\leq d_{+}+\varepsilon \label{071502}
 \end{eqnarray}
 holds with overwhelming probability.
 \end{thm}
 \begin{rem} The estimate (\ref{071502}),  is actually  implied by some existing results in the literature directly. For example, one can refer to the small deviation estimate of the largest eigenvalue of the Jacobi ensemble in \cite{Katz2012} . Moreover, (\ref{071501}) is a direct consequence of (\ref{071502}) and Theorem \ref{thm.070301}.
 \end{rem}
 \begin{rem} As we mentioned in Introduction , on the fluctuation level,  $\lambda_1$  possesses a Type 1 Tracy-Widom limit after appropriate normalization. Such a result has been  established recently in \cite{Johnstone2008,HanP16T, HPY}.
 \end{rem}
\subsection{Finite rank case} We now turn to the case we are interested in: the finite rank case. To wit, Assumption \ref{ass.070302} holds. It will be clear that the CCA matrix in such a finite rank case can be viewed as a finite rank perturbation of that in the null case. Consequently, the global behavior, especially the LSD, turns out to coincide with the null case. However, finite rank perturbation may significantly alter the behavior of the extreme eigenvalues,  when the perturbation is strong enough. Similar problems have been studied widely for various random matrix models, not trying to be comprehensive, we refer to the spiked sample covariance matrices \cite{Johnstone2001, BS2006, Paul2007, BY2008, BBP2005, FP2009}, the deformed Wigner matrices \cite{CDF2009, peche2006, FP2007, CDF2012, KY2013}, the deformed unitarily invariant matrices \cite{BBCF2012, Kargin2014}, and  some other deformed models \cite{BN2011, BGM2011, WY}. In this work, for our CCA matrix $ C_{xy}$,  we study the limits  and the fluctuations of its largest eigenvalues, i.e. squares of the largest sample canonical correlation coefficients, under Assumption \ref{ass.070302}.
Our main results are the following three theorems. Let
\begin{align}
\gamma_i:=r_i(1-c_1+c_1r_i^{-1})(1-c_2+c_2r_i^{-1}). \label{17050801}
\end{align}
Recall $r_c$ and $d_+$ defined in (\ref{062805}) and (\ref{17031301}), respectively.  It is easy to check that $\ga_i\geq d_+$ if $r_i\geq r_c$.

\begin{thm}[Limits] \label{thm.061901} Under Assumptions \ref{ass.070301} and \ref{ass.070302}, the squares of the largest canonical correlation coefficients exhibit the following convergence as $n\to\infty$.
\begin{itemize}
\item[(i):] (Outliers) For $1\leq i\leq k_0$, we have
\begin{eqnarray*}
\lambda_i-\ga_i\stackrel{a.s.}\longrightarrow 0.
\end{eqnarray*}
\item[(ii):] (Sticking eigenvalues)  For  each fixed $i\geq k_0+1$, we have
\begin{eqnarray*}
\lambda_i-d_{+} \stackrel{a.s.}\longrightarrow 0 .
\end{eqnarray*}
\end{itemize}
\end{thm}

The next two results are on fluctuations of $\lambda_i$'s. We need the following definition.
\begin{defi} \label{def.well separated} For two (possibly) $n$-dependent numbers $a(n),b(n)\in \mathbb{C}$, we say $a(n)$ is well separated from $b(n)$, if there exists a small positive constant $\varepsilon$ such that $|a(n)-b(n)|\geq \varepsilon$ for sufficiently large $n$.
\end{defi}

The next theorem is on the fluctuations of the outliers.
\begin{thm}[Fluctuations of the outliers] \label{thm.fluctuation}  Suppose that   Assumptions \ref{ass.070301} and \ref{ass.070302}  hold.
 Let $l_0$ be the cardinality of the set  $\Gamma:=\{r_1, \ldots, r_{k_0}\}$ (not counting multiplicity), and  denote by $r_1=r_{(1)}>\cdots>r_{(l_0)}=r_{k_0}$ the $l_0$ different values in $\Gamma$.
Set $n_0=0$ and denote by  $n_l$ the multiplicity of  $r_{(l)}$ in $\Gamma$ for $1\leq l\leq l_0$. Let  $J_l=[\sum_{i=0}^{l-1}n_i+1,\sum_{i=0}^{l}n_i]\cap \mathbb{Z}$ for $1\leq l\leq l_0$.   If $r_{(\ell)}$ is well separated from $r_c, 1 $, $r_{(\ell-1)}$ and $r_{(\ell+1)}$,  the $n_l$-dimensional random vector
	\begin{align*}
	\Big\{\frac{\sqrt{n}(\lambda_{j}-\gamma_j)}{\xi(r_j)},~j\in J_l\Big\}
	\end{align*}
	converges weakly to the distribution of the ordered eigenvalues of an $n_l$-dimensional symmetric Gaussian random matrix $G=(g_{ij})$ with  independent (up to symmetry) entries $g_{ij}\sim N(0, 1+\delta_{ij})$.  Here
	\begin{align}
	&\xi^2(r_j):=\frac{(1-r_j)^2\Big(2(1-c_1)(1-c_2) r_j+c_1+c_2-2c_1c_2\Big)\Big((1-c_1)(1-c_2) r_j^2-c_1c_2\Big)}{r_j^2}. \label{17040780}
	\end{align}
\end{thm}

\begin{rem} \label{rem.separation1} The assumption that $r_{\ell}$ is away from $r_c, 1 $, $r_{(\ell-1)}$ and $r_{(\ell+1)}$ by a constant distance $\varepsilon>0$ (well separated) is not necessary. It is possible to reduce $\varepsilon$ to some $n$-dependent distance $n^{-\alpha}$ for sufficiently small $\alpha$. But we do not pursue this direction here.
\end{rem}

Our last theorem  is on the fluctuation of the sticking eigenvalues.
\begin{thm}[Fluctuations of the sticking eigenvalues] \label{thm. fluctuation of sticking evs} Suppose that   Assumptions \ref{ass.070301} and  \ref{ass.070302} hold. In addition, we assume that $r_c$ is well separated from $r_{k_0}$ and $r_{k_0+1}$. There exists a CCA matrix in the null case with the same parameters $p,q,n$, whose nonzero eigenvalues are denoted by $\mathring{\lambda}_1>\mathring{\lambda}_2\ldots>\mathring{\lambda}_q$, such that for any fixed positive integer $\mathfrak{m}$ and any small constant $\varepsilon>0$, we have
\begin{align}
\max_{1\leq i\leq \mathfrak{m}}|\lambda_{k_{0}+i}-\mathring{\lambda}_i| \leq n^{-1+\varepsilon} \label{17052821}
\end{align}
in probability. This implies
\begin{align}
\frac{n^{\frac{2}{3}}(\lambda_{k_0+1}-d_+)}{\xi_{tw}} \Longrightarrow F_1, \label{17052822}
\end{align}
where  $F_1$ is the Type 1 Tracy-Widom distribution and
\begin{align*}
\xi^3_{tw}=\frac{ d_+^2(1- d_+)^2}{\sqrt{c_1c_2(1-c_1)(1-c_2)}}.
\end{align*}
\end{thm}

\begin{rem}\label{rem.17052825} From (\ref{17052821}), it is easy to conclude (\ref{17052822}), by using the Tracy-Widom limit of $\mathring{\lambda}_1$ derived in \cite{Johnstone2008,  HanP16T, HPY}. Observe that in these references, the Tracy Widom law is stated for the logit transform of  $\mathring{\lambda}_1$, i.e., $\log (\mathring{\lambda}_{1}/(1-\mathring{\lambda}_1))$. Using a Taylor expansion, it is elementary to check that that $n^{\frac{2}{3}}(\mathring{\lambda}_{1}-d_+)/\xi_{tw} \Longrightarrow F_1$ from the Tracy-Widom law for the logit transform of $\mathring{\lambda}_1$ (see Theorem 1 and Section 2.1.1.in \cite{Johnstone2008} for more details).
\end{rem}

\begin{rem} Similarly to Remark \ref{rem.separation1},  the assumption that $r_c$ is away from $r_{k_0}$ and $r_{k_0+1}$ by a constant distance can be weakened. But we do not pursue this direction here.
\end{rem}

To illustrate the result in Theorems \ref{thm.061901}, \ref{thm.fluctuation} and \ref{thm. fluctuation of sticking evs}, we did some numerical simulations.
The different limiting behavior of $\lambda_i$ in (i) and (ii) of Theorem \ref{thm.061901} can be observed in Figure \ref{fig1}.
The fluctuation for $\lambda_i$ in Theorem \ref{thm.fluctuation} and Theorem \ref{thm. fluctuation of sticking evs}  can be seen from Figures \ref{fig2} and \ref{fig1}.
\begin{figure}[h]
\begin{center}
\includegraphics[width=16cm]{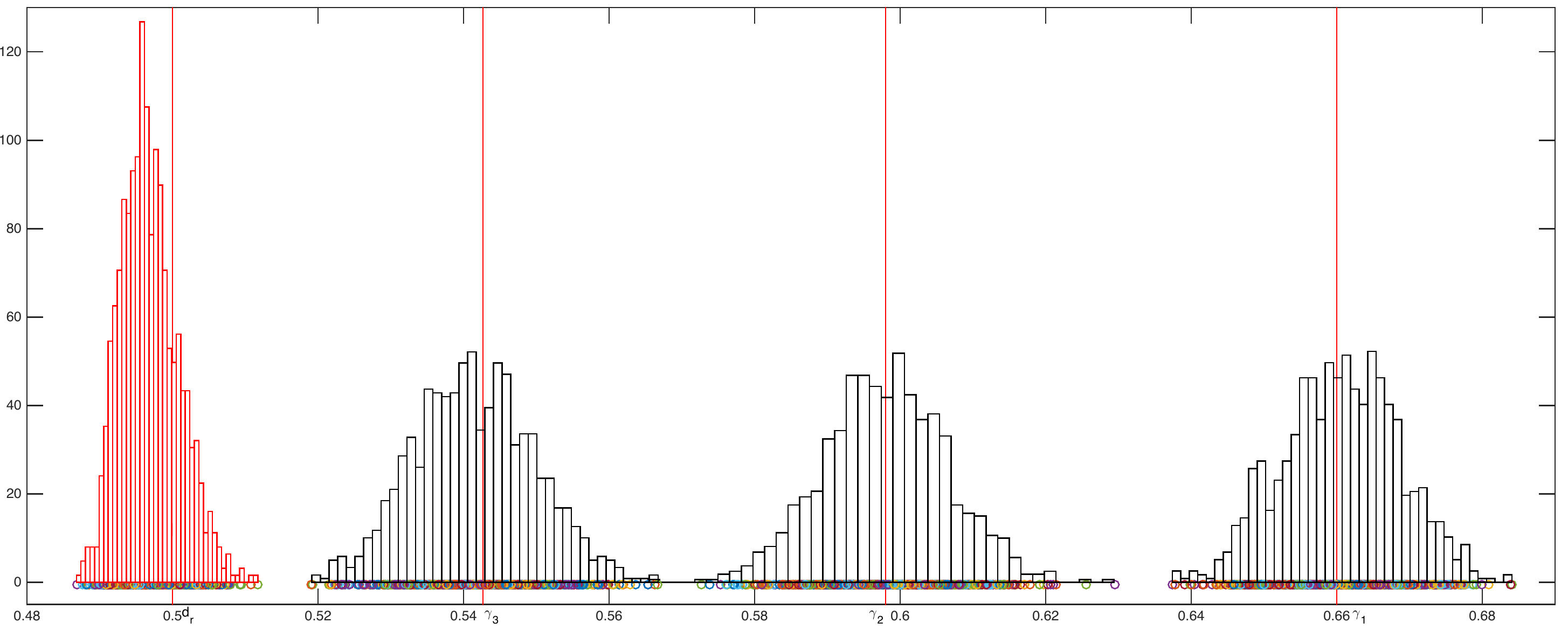}
\end{center}
\caption{ We chose a Gaussian vector $\mathbf{z}=(\mathbf{x}',\mathbf{y}')'$ with $p=500$ and $q=1000$. The sample size is $n=5000$. Hence, $c_1=0.1$, $c_2=0.2$. Then $r_c\approx 0.17$ and $d_{+}=0.5$. We chose $k=4$ and $(r_1, r_2, r_3, r_4)=(0.5, 0.4, 0.3, 0.16)$. Then $\gamma_1\approx 0.66$, $\gamma_2\approx 0.6$, $\gamma_3\approx 0.54$ and $\gamma_4\approx 0.50$ . The above simulation result is based on 1000
replications. The abscises of the vertical segments represent $d_{+}$ and $\{\gamma_1,\gamma_2,\gamma_3\}$. The four histograms represent  the distributions of the largest  four  eigenvalues of the CCA matrix respectively.} \label{fig1}
\end{figure}

\begin{figure}[h]
	\begin{center}
		\includegraphics[width=16cm]{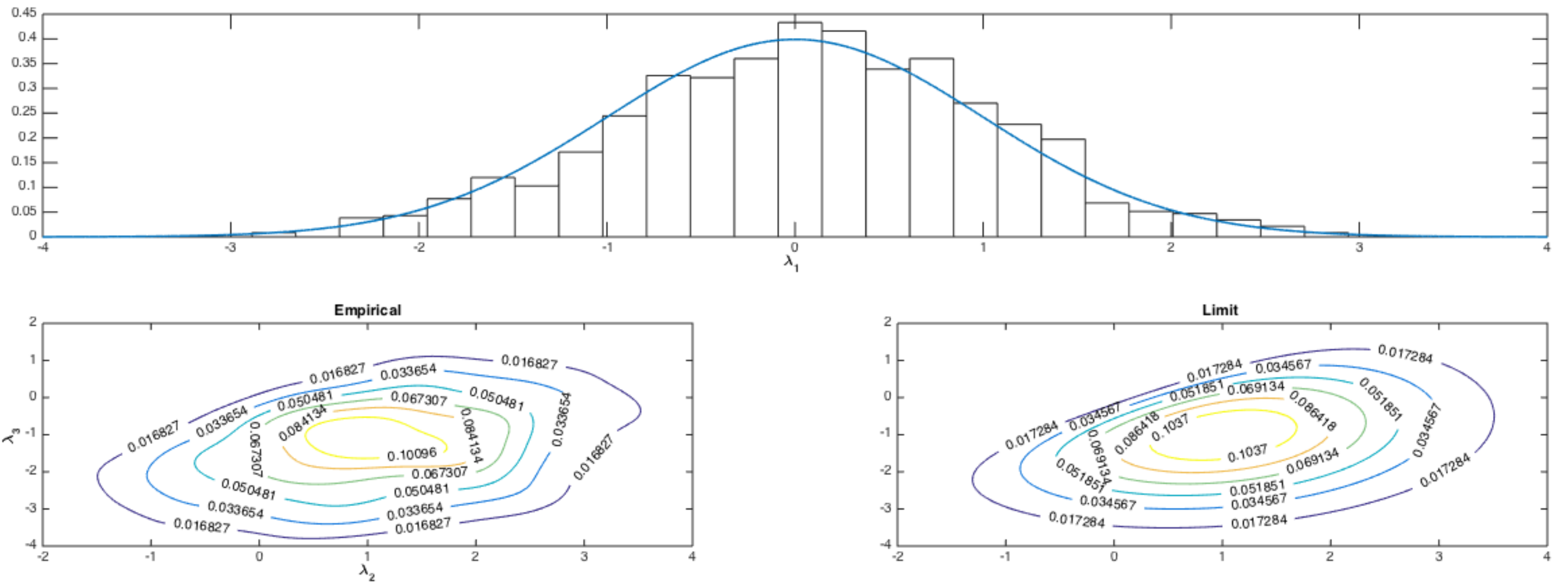}
	\end{center}
	\caption{We chose a Gaussian vector $\mathbf{z}=(\mathbf{x}',\mathbf{y}')'$ with $p=500$ and $q=1000$. The sample size is $n=5000$. Hence, $c_1=0.1$, $c_2=0.2$. Then $r_c\approx 0.17$ and $d_{+}=0.5$. We chose $k=4$ and $(r_1, r_2, r_3, r_4)=(0.5, 0.4, 0.4, 0.16)$. Then $\gamma_1\approx 0.66$ and $\gamma_2=\gamma_3\approx 0.6$. This result is based on 1000
		replications. The upper panel shows that the  frequency histogram of $\lambda_1$ compared to its Gaussian limits (dashed lines).  The lower panels show that the contour plots of empirical joint density function of $\lambda_2$ and
		 $\lambda_3$ (left plot, after centralisation and scaling) and contour plots of their limits (right plot). The empirical  joint density function is displayed by using the two-dimensional kernel density estimates.} \label{fig2}
\end{figure}

\subsection{Organization and notations}
Our paper is organized as follows. In Section \ref{s.app}, we display some applications of our results. We introduce in Section \ref{s. preliminary} some necessary preliminaries. In Section \ref{s. determinant eq}, we will reformulate the sample canonical correlation matrix in the finite rank case as a perturbation of that in the null case, thereby obtaining a determinant equation for the largest eigenvalues. By solving the limiting counterpart of the determinant equation, we can get the limits of the largest eigenvalues in Section \ref{s. proof of main theorem}, i.e. prove Theorem \ref{thm.061901}. In Section \ref{section.fluctuation} we derive the fluctuations of the outliers (Theorem \ref{thm.fluctuation}), and in Section \ref{s. fluctuation of sticking}, we study the fluctuations of the sticking eigenvalues (Theorem \ref{thm. fluctuation of sticking evs}).  Sections \ref{sec5} and \ref{sec6} are then devoted to proving some  technical lemmas used in previous sections.

Throughout the paper,  the notation $C$ represents some generic constants whose values which may vary from line to line. The notation $\mathbf{0}_{k\times \ell}$ is used to denote the $k$ by $\ell$ null matrix, which will be abbreviated to $\mathbf{0}_{k}$ if $k=\ell$, and from time to time we also use the abbreviation $\mathbf{0}$ if the dimension $k\times \ell$ is clear from the context. For any matrix $\A$, its $(i,j)$th entry will be written as $\A_{ij}$. When $\A$ is square, we denote by  $\text{Spec}(\A)$ its spectrum. For a function $f:\mathbb{C}\to\mathbb{C}$ and an Hermitian matrix $\A$ with spectral decomposition $\U_\A\bm\Lambda_\A\U^*_\A$, we define $f(\A)$ as usual, in the sense of functional calculus, to wit, $f(\A)=\U_\A f(\bm\Lambda_\A)\U^*_\A$, where $f(\bm\Lambda_\A)$ is the diagonal matrix obtained via maping the eigenvalues of $\A$ to their images under $f$.
We will conventionally use the notations $\|\A\|$ and $\|A\|_{\text{HS}}$ to represent the operator norm and Hilbert-Schmidt norm of a matrix $\A$, respectively,  while $\|\mathbf{b}\|$ stands for the Euclidean norm of a vector $\mathbf{b}$. Throughout this paper, we use $o_p(1)$ to denote  a scalar negligible (in probability) or a fixed-dimensional random matrix with negligible (in probability) entries. And the notations $o_{\text{a.s}}(1)$,  $O_p(1)$ and $O_{\text{a.s.}}(1)$ are used in a similar way.

For brevity, we further introduce the notations
\begin{align}
\varpi:=c_1c_2,\qquad \vartheta:=(1-c_1)(1-c_2).  \label{17050702}
\end{align}

\section{Applications} \label{s.app} In this section, we discuss three applications of our  results Theorems \ref{thm.061901}, \ref{thm.fluctuation} and  \ref{thm. fluctuation of sticking evs} in
hypothesis testing, estimation of the number and the values of the canonical correlation coefficients  (CCC for short).
At the end, we present an experiment on a real rain forest data.
\subsection{Application 1: Power of testing for the independence of two high dimensional normal vectors} Testing for the independence of random vectors is a very classical
problem. For two normal random vectors $\mathbf{x}$ and $\mathbf{y}$ with dimensions $p$ and $q$ respectively, we consider the test
\begin{align*}
H_0: \Sigma_{xy}=\mathbf{0}_{p\times q}\quad\mbox{v.s.} \quad H_1: \mbox{not } H_0.
\end{align*}
Currently,  for high  dimensional cases, there are  three kinds of widely discussed test procedures: (i)  Corrected likelihood ratio tests (see \cite{JiangB13T,JiangY13C}); (ii) Trace tests (see \cite{JiangB13T,YangP15I, HyodoS15T,YamadaH17T,BaoH17T}); (iii) Largest eigenvalue tests (see \cite{Johnstone2008,Johnstone09A,HanP16T}).
It has been shown by numerical results  in \cite{BaoH17T,HanP16T} that if $ \Sigma_{\bf{x}\bf{y}} $ is sparse,  the corrected likelihood ratio tests and trace tests fail and the largest eigenvalue tests works well.  In the following, we propose a statistic based on the CCC of $\mathbf{x}$ and $\mathbf{y}$, and show that it is powerful against finite rank case. It is well known that testing for the independence of two high dimensional normal vectors
equals to testing  their first canonical correlation coefficient being zero or not, i.e.,
\begin{align*}
H_0: r_1=0\quad\mbox{v.s.} \quad H_1: r_1>0.
\end{align*}
Therefore, a natural test statistic is the first eigenvalue  $\lambda_1$ of the sample canonical correlation matrix $C_{xy}$.  Then according to Section 2.1.1.in \cite{Johnstone2008}  and Theorem 2.1 in \cite{HanP16T}, under the null hypothesis and Assumption \ref{ass.070301} we  have
\begin{align}\label{teststa1}
\frac{n^{2/3}(\lambda_1-d_+)}{\xi_{tw}}\Longrightarrow F_1.
\end{align}
 Therefore, we reject $H_0$ if \begin{align}\label{test}
\lambda_1>n^{-2/3}q_\alpha\xi_{tw}+ d_+,
\end{align} where $q_\alpha$ is the $1- \alpha$ quantile  of Tracy-Widom distribution $F_1$. If  the sample is under the Assumption \ref{ass.070302} and $r_1>r_c$, then this test will be able to detect the alternative hypothesis with a power tending to one as the dimension tends to infinity.
\begin{thm}[Power function]\label{thm:power}
	Suppose that the assumptions in  Theorems \ref{thm.fluctuation} and \ref{thm. fluctuation of sticking evs}  hold with $r_1>r_c$. Then as  $n\to\infty$, the power function of the test procedure (\ref{test})
	\begin{align*}
	{Power}=\mathbb{P}\(\frac{\sqrt{n}(\lambda_1-\ga_1)}{\xi(r_1)}>\frac{n^{-1/6}q_\alpha\xi_{tw}}{\xi(r_1)}+\frac{\sqrt{n}( d_+-\ga_1)}{\xi(r_1)}\)\to1.
	\end{align*}
\end{thm}
\begin{proof}
	Under the conditions in  Theorems \ref{thm.fluctuation} and \ref{thm. fluctuation of sticking evs}, for any $\alpha $, we have  $ \frac{n^{-1/6}q_\alpha\xi_{tw}}{\xi(r_1)}\to0$. From the conclusion of   Theorem \ref{thm.061901} and $r_1>r_c$, we  have $\lim_{n\to\infty} ( d_+-\ga_1)<0$, which implies $\lim_{n\to\infty}\frac{\sqrt{n}( d_+-\ga_1)}{\xi(r_1)}\to-\infty$. In addition, by applying Theorem \ref{thm.fluctuation}, $\frac{\sqrt{n}(\lambda_1-\ga_1)}{\xi(r_1)}$ converges weakly to the distribution of the largest eigenvalue of a symmetric Gaussian random matrix. Therefore,  combining the above results, we then complete the proof of Theorem \ref{thm:power}.
\end{proof}

\subsection{Application 2: Estimating  the number  of the population  CCC}

 As a fundamental problems in CCA, the estimation of   the number  and the values of the population  CCC are widely investigated.
In this subsection,  we first apply our results to   determine  the number  of the outliers of  high dimensional population  CCC (not counting multiplicity).
Actually, due to the threshold in our results, we cannot detect the population  CCC which are smaller than $\sqrt{r_c}$. To the best of our knowledge, there is no effective method in general to successfully detect the population spikes below the threshold when the dimension is large,  even for the simpler spike PCA problem. The reason  can be find in the  recent work  \cite{BaiF17H} for instance.

We propose our estimator of  the number of the outliers of population  CCC $k_0$ by the number of eigenvalues of  the sample canonical correlation matrix which are larger than $d_+$:
\begin{align}\label{estk0}
\hat k_0:=\max \{i:\lambda_i\geq d_++\epsilon_n\},
\end{align}
where $\epsilon_n$ is a sequence of positive number only depending on $n$ and satisfying $\epsilon_n \sqrt{n}\to0$ and $\epsilon_n {n}^{2/3}\to\infty$.
Then the estimator $\hat k_0$ is weakly consistent according to the following theorem.
\begin{thm}[Weak consistency of  $\hat k_0$] Suppose that the assumptions in   Theorems \ref{thm.fluctuation} and \ref{thm. fluctuation of sticking evs} hold.  As  $n\to\infty$, the  estimator $\hat k_0$ in \eqref{estk0} is  weakly consistent for  the  number of the outliers of population  CCC $k_0$, that is
	\begin{align}\label{esteq0}
	\P (\hat k_0=k_0)\to1.
	\end{align}
\end{thm}
\begin{proof}
	According to the definition of $\hat k_0$ in \eqref{estk0}, we have that
	\begin{align}\label{esteq1}
	\P (\hat k_0=k_0)=&\P\big(\lambda_{k_0}\geq d_++\epsilon_n, \lambda_{k_0+1}< d_++\epsilon_n\big)\nonumber\\
		\geq&1- \P(\lambda_{k_0}< d_++\epsilon_n)-\P(\lambda_{k_0+1}\geq  d_++\epsilon_n).
	\end{align}

	By Theorem \ref{thm.fluctuation} and $\epsilon_n \sqrt{n}\to0$, we obtain that
	\begin{align}\label{esteq2}
	\P(\lambda_{k_0}< d_++\epsilon_n)=\P\big(\sqrt{n}(\lambda_{k_0}-\ga_{k_0})< \sqrt{n}(d_+-\ga_{k_0}+\epsilon_n)\big)\to0.
	\end{align}
	Analogously, from Theorem \ref{thm. fluctuation of sticking evs} and $\epsilon_n n^{2/3}\to\infty$, we also have
	\begin{align*}
	\P(\lambda_{k_0+1}\geq d_++\epsilon_n)=\P\big(n^{2/3}(\lambda_{k_0+1}-d_+)\geq n^{2/3}\epsilon_n\big)\to 0,
	\end{align*}
	which together with \eqref{esteq1} and \eqref{esteq2} implies \eqref{esteq0}.
	Therefore,  we complete the proof.
\end{proof}

\begin{rem}
	Although in theory,  any sequence $\epsilon_n$ which  satisfies the conditions $\epsilon_n \sqrt{n}\to0$ and $\epsilon_n {n}^{2/3}\to\infty$ is applicable, we have to choose one in practice. Thus,  it is worth to notice that the  smaller   $\epsilon_n$ one chooses, the easier one overestimates $k_0$ and vice versa.     In the simulation, based on the  idea in  the paper \cite{PASSEMIERY12D,WY},  we choose   $\epsilon_n$  to be $\log\log(n)/n^{2/3}$, which makes our  estimator conservative.
	
\end{rem}
In the following, we  report a short simulation result to illustrate the performance of our estimator. For  comparison, we also show the result of other three estimators  in \cite{FujikoshiS16H} by  using  the following model selection criteria: AIC,  BIC and  $C_p$  respectively, which  are
\begin{align*}
\hat k_A=\mathop{\arg\min}_{j\in\{0,1,\dots,q\}}AIC_j, \quad  \hat k_B=\mathop{\arg\min}_{j\in\{0,1,\dots,q\}}BIC_j, \quad  \hat k_C=\mathop{\arg\min}_{j\in\{0,1,\dots,q\}}CP_j.
\end{align*}
Here for $j\in\{1,\dots,q\}$,
\begin{align*}
AIC_j&=-n\log\Big[\prod\limits_{i=j+1}^{q}(1-\la_i)\Big]-2(p-j)(q-j),\\
BIC_j&=-n\log\Big[\prod\limits_{i=j+1}^{q}(1-\la_i)\Big]-\log(n)(p-j)(q-j),\\
CP_j&=n\sum_{i=j+1}^{q}\frac{\la_i}{1-\la_i} -2(p-j)(q-j)
\end{align*}
and $AIC_0=BIC_0=CP_0=0$.

 According to the results in  Table \ref{tab1} and Table \ref{tab2}, we can find that the performance  of  our estimator is much better than AIC,  BIC and $C_p$ estimators when the dimensions are big. If the dimensions are small  and the sample size is big enough, then besides our  estimator, the AIC  and $C_p$ estimators perform well, but BIC estimator performs badly. That coincides with the conclusion in \cite{FujikoshiS16H}, which showed if $q$ is fixed and $p/n\to c_1>0$ but small,  AIC and $C_p$ estimators are weakly consistent but BIC estimator is not.
\begin{table}[htbp]
	\renewcommand\arraystretch{1.5}
	\begin{tabular}{ccccccccccccc}
		\hline
		\multirow{2}{2em}&\multicolumn{4}{c}{$p$=10,  $r_c=0.0071$} &\multicolumn{4}{c}{$p$=60, $r_c=0.0445$} &\multicolumn{4}{c}{$p$=110, $r_c=0.0854$} \\ \cline{2-13}
		&$\hat k_0$&$\hat k_A$&$\hat k_B$&$\hat k_C$&$\hat k_0$ &$\hat k_A$&$\hat k_B$&$\hat k_C$&$\hat k_0$&$\hat k_A$&$\hat k_B$&$\hat k_C$ \\ \hline
		$\leq2$    &0&0&0&0&0&0& 375 &0& 0&0&1000 &0 \\ \hline
		$3$     &0&0&0&0&90& 0 &625&0& 95&0&0 &0 \\ \hline	
		$4$    & {\bf 1000}&{\bf 937}&{\bf 1000}&{\bf 935}&{\bf1000}& {\bf846} &\bf0&{\bf463}& \bf905&\bf488&\bf0 &\bf0   \\ \hline	
		$5$     &0&63&0&65& 0 &151&0& 465&0&458 &0&26 \\ \hline	
		$\geq 6$  &0&0&0&0&0& 3&0&72& 0&54&0 &974 \\ \hline	
		\hline
		&\multicolumn{4}{c}{$p$=160,  $r_c=0.1305$} &\multicolumn{4}{c}{$p$=210, $r_c=0.1809$} &\multicolumn{4}{c}{$p$=260, $r_c=0.2379$} \\ \cline{2-13}
		
		$\leq2$    &0&0&1000&0&0& 0 &1000&0& 61&0&1000 &0 \\ \hline
		$3$     &714&0&0&0&27& 0 &0&0& \bf939&\bf0&\bf0 &\bf0 \\ \hline	
		$4$    & {\bf 286}&{\bf 62}&{\bf 0}&{\bf 0}&{\bf793}& {\bf0} &\bf0&{\bf0}& 0&0&0 &0  \\ \hline	
		$5$     &0&349&0&0& 0 &1&0& 0&0&0 &0&0 \\ \hline	
		$\geq 6$  &0&589&0&1000&0& 999&0&1000& 0&1000&0 &1000 \\ \hline	
	\end{tabular}
	\caption{ We fix $(r_1, r_2, r_3, r_4)=(0.8, 0.6, 0.4, 0.2)$, $r_5=\cdots r _q=0$, $p/q=2$, $n=1000$, and let $\{p,q\}$ vary. The results are obtained  based on 1000 replications with centered normal distributions. The tuning parameter $\epsilon_n$ is chosen to be $\log\log(n)/n^{2/3}$. We present the counting numbers  of AIC, BIC, $C_p$ estimators proposed in \cite{FujikoshiS16H} and our estimator which are equal to the values in the first column.  The correct numbers for the estimators are marked {\bf bold}. Notice that as $r_c$ changes along with $\{p,q\}$,  the true value of $k_0$ are different.  }\label{tab1}
\end{table}
\begin{table}[htbp]
	\renewcommand\arraystretch{1.5}
	\begin{tabular}{ccccccccccccc}
		\hline
		\multirow{2}{2em}&\multicolumn{4}{c}{$p$=10,  $q=5$, $n=100$} &\multicolumn{4}{c}{$p$=60, $q=30$, $n=600$} &\multicolumn{4}{c}{$p$=110, $q=55$, $n=1100$} \\ \cline{2-13}
		&$\hat k_0$&$\hat k_A$&$\hat k_B$&$\hat k_C$&$\hat k_0$ &$\hat k_A$&$\hat k_B$&$\hat k_C$&$\hat k_0$&$\hat k_A$&$\hat k_B$&$\hat k_C$ \\ \hline
		$\leq2$    &44&0&88&0&0&0& 1000 &0& 0&0&1000 &0 \\ \hline
		$3$     &877&65&732&39&186& 2 &0&0& 22&0&0 &0 \\ \hline	
		$4$    & {\bf 79}&{\bf 871}&{\bf 180}&{\bf 876}&{\bf814}& {\bf744} &\bf0&{\bf118}& \bf978&\bf559&\bf0 &\bf2   \\ \hline	
		$5$     &0&64&0&85& 0 &238&0& 525&0&400 &0&51 \\ \hline	
		$\geq 6$  &0&0&0&0&0& 16&0&357& 0&41&0 &947 \\ \hline	
		\hline
		&\multicolumn{4}{c}{$p$=160,  $q=80$, $n=1600$} &\multicolumn{4}{c}{$p$=210, $q=105$, $n=2100$} &\multicolumn{4}{c}{$p$=260, $q=130$, $n=2600$} \\ \cline{2-13}
		
		$\leq2$    &0&0&1000&0&0& 0 &1000&0& 0&0&1000 &0 \\ \hline
		$3$     &6&0&0&0&3& 0 &0&0& 0&0&0 &0 \\ \hline	
		$4$    & {\bf 994}&{\bf 426}&{\bf 0}&{\bf 0}&{\bf1000}& {\bf289} &\bf0&{\bf0}& \bf1000&\bf188&\bf0 &\bf0  \\ \hline	
		$5$     &0&482&0&0& 0 &547&0& 0&0&546 &0&0 \\ \hline	
		$\geq 6$  &0&92&0&1000&0& 164&0&1000& 0&266&0 &1000 \\ \hline	
	\end{tabular}
	\caption{ We fix $(r_1, r_2, r_3, r_4)=(0.8, 0.6, 0.4, 0.2)$, $r_5=\cdots r _q=0$,  $c_1=0.1$, $c_2=0.05$,  $r_c=0.077$, $k_0=4$, and let $\{p,q,n\}$ change. The results are obtained  based on 1000 replications with centered normal distributions. The tuning parameter $\epsilon_n$ is chosen to be $\log\log(n)/n^{2/3}$.  We present the counting numbers  of AIC, BIC, $C_p$ estimators proposed in \cite{FujikoshiS16H} and our estimator which are equal to the values in the first column.  The correct numbers are marked {\bf bold}. }\label{tab2}
\end{table}
\subsection{Application 3: Estimating  the values  of the population  CCC}
As a measure to detect the  correlation level between two random vectors,  the  population   CCC are also important. Traditionally,  one uses the sample CCC to estimate the population ones directly, since under the traditional assumption that $p,q$ are fixed and  $n\to\infty$,  the sample CCC tend to the  population ones almost surely (see Chapter 12 in \cite{Anderson03I}).  However, it has been  noticed  that sample CCC suffer from inflation, which occurs when the number of observations is not sufficient (see  \cite{RencherP80I,DutilleulP08M,ZhengJ14I}). But the good news is that now we can  easily explain the reason of  the   inflation properties  of the sample CCC by our Theorem \ref{thm.061901}.
In this subsection, we will give an approach to estimate the  population  CCC which are bigger than $\sqrt{r_c}$.

Solving the equation \eqref{17050801} via  replacing $\ga_i$ by $\la_i$, we have two solutions
\begin{align}\label{eqri}
 \hat r_i:=\frac{2c_1c_2-c_1-c_2+\la_i\pm\sqrt{(\la_i-d_-)(\la_i-d_+)}}{2(c_1c_2-c_1-c_2+1)}.
\end{align}
 Notice that the product of two solutions equals to $c_1c_2/[(1-c_1)(1-c_2)]=r_c^2$. Therefore,  according to the fact that  $\la_i>d_+$ and $\hat r_i\to r_i>r_c$ when  $n\to\infty$,  we should choose a plus sign in (\ref{eqri}). Hence,  the estimator of  $r_i$ is chosen to be
 \begin{align}\label{esthatt}
\hat r_i:=\phi(\la_i)=\frac{2c_1c_2-c_1-c_2+\la_i+\sqrt{(\la_i-d_-)(\la_i-d_+)}}{2(c_1c_2-c_1-c_2+1)},
\end{align}
Then according to Theorem \ref{thm.061901}, we have the following theorem.
\begin{thm}
Under  the same conditions of Theorem \ref{thm.061901}, for any $1\leq i\leq k_0$, we have almost surely
\begin{align*}
\hat r_i\to r_i.
\end{align*}
\end{thm}

Next, we display a short simulation result  in Table \ref{tab3_1} and Table \ref{tab3_2}, which are the sample  means  and  standard deviations  (s.d.) of the estimators $\hat r_i$ . In  the simulation, we first estimate $k_0$ by \eqref{estk0} with  the tuning parameter $\epsilon_n=\log\log(n)/n^{2/3}$, and then for $i\leq \hat k_0$, we   obtain the estimators $\hat r_i$ by \eqref{esthatt}.
 \begin{table}[htbp]
	\renewcommand\arraystretch{1.5}
	\begin{tabular}{ccccccccc}
		\hline
		\multirow{2}{2em}&\multicolumn{2}{c}{$\hat r_1$}&\multicolumn{2}{c}{$\hat r_2$} &\multicolumn{2}{c}{$\hat r_3$}&\multicolumn{2}{c}{$\hat r_4$}  \\ \cline{2-9}
		&mean&s.d.&mean&s.d.&mean&s.d.&mean&s.d.\\ \hline
		$p$=10,  $q=5$, $n=100$ &0.798&0.036&0.588&0.066&0.362&0.072&0.241&0.033\\ \hline
		$p$=60,  $q=30$, $n=600$     &0.800&0.016&0.598&0.027&0.392&0.035&0.190&0.029\\ \hline
		$p$=110,  $q=55$, $n=1100$& 0.799&0.012&0.599&0.020&0.397&0.025&0.189&0.026\\ \hline
		$p$=160,  $q=80$, $n=1600$&0.800&0.010&0.598&0.017&0.397&0.020&0.191&0.023\\ \hline
		$p$=210,  $q=105$, $n=2100$ & 0.800&0.008&0.600&0.014&0.397&0.019&0.194&0.020\\ \hline
		$p$=260,  $q=130$, $n=2600$ & 0.800&0.008&0.599&0.013&0.399&0.017&0.194&0.018\\ \hline
		
	\end{tabular}
	\caption{ We fix $(r_1, r_2, r_3, r_4)=(0.8, 0.6, 0.4, 0.2)$, $r_5=\cdots r _q=0$,  $c_1=0.1$, $c_2=0.05$,  $r_c=0.077$, $k_0=4$, and let $\{p,q,n\}$ vary. The results are obtained  based on 1000 replications with centered normal distributions. The tuning parameter $\epsilon_n$ is chosen to be $\log\log(n)/n^{2/3}$.  }\label{tab3_1}
\end{table}
\begin{table}[htbp]
	\renewcommand\arraystretch{1.5}
	\begin{tabular}{ccccccccc}
		\hline
		\multirow{2}{2em}&\multicolumn{2}{c}{$\hat r_1$}&\multicolumn{2}{c}{$\hat r_2$} &\multicolumn{2}{c}{$\hat r_3$}&\multicolumn{2}{c}{$\hat r_4$}  \\ \cline{2-9}
		&mean&s.d.&mean&s.d.&mean&s.d.&mean&s.d.\\ \hline
		$p$=10,  $q=5$, $n=100$ &
		0.800&0.037&0.594&0.061&0.409&0.066&0.288&0.051\\ \hline
		$p$=60,  $q=30$, $n=600$     &
		0.799&0.015&0.599&0.027&0.419&0.029&0.356&0.031\\ \hline
		$p$=110,  $q=55$, $n=1100$&
		0.799&0.011&0.599&0.020&0.417&0.021&0.370&0.021\\ \hline
		$p$=160,  $q=80$, $n=1600$&
		0.800&0.010&0.600&0.015&0.415&0.018&0.376&0.018\\ \hline
		$p$=210,  $q=105$, $n=2100$ &
		0.800&0.008&0.600&0.014&0.413&0.015&0.381&0.016\\ \hline
		$p$=260,  $q=130$, $n=2600$ &
		0.800&0.007&0.600&0.014&0.412&0.014&0.383&0.015\\ \hline
		
	\end{tabular}
	\caption{ We fix $(r_1, r_2, r_3, r_4)=(0.8, 0.6, 0.4, 0.4)$, $r_5=\cdots r _q=0$,  $c_1=0.1$, $c_2=0.05$,  $r_c=0.077$, $k_0=4$, and let $\{p,q,n\}$ vary. The results are obtained  based on 1000 replications with centered normal distributions. The tuning parameter $\epsilon_n$ is chosen to be $\log\log(n)/n^{2/3}$.  }\label{tab3_2}
\end{table}
According to the results we find that our estimator is excellent, especially when the population CCC are not equal to each other.  If the multiplicity of  some CCC is bigger than one,  as the results shown in Table \ref{tab3_2}, there should be  some $\hat r_i$'s  close to each other. In this case, although $r_i=r_{i+1}$ for certain multiple population CCC, the estimator $\hat{r}_i$ may differ from $\hat{r}_{i+1}$ by a certain amount since our $\lambda_i$'s are ordered. Suppose now we can determine $r_i=r_{i+1}$ from the information in the CCA matrix, we may then use the average of $\hat{r}_i$ and $\hat{r}_{i+1}$ to get more precise estimate of both $r_i$ and $r_{i+1}$.
Hence,  determining  if there are multiple $r_i$'s would be important for the estimate of $r_i$ as well,   although sometimes directly using $\hat r_i$ is sufficient. In the following,  we propose a statistic to test the hypothesis:
\begin{align}\label{apptest2}
H_0: r_{j_0-1}>r_{j_0}=r_{j_0+1}=\dots= r_{j_0+j_1-1}>r_{j_0+j_1}\quad\mbox{v.s.}\quad H_1: \mbox{not}~ H_0,
\end{align}
where $j_0\geq1$, $j_0+j_1-1\leq k_0$ and $r_0=1$. If this test is not rejected, we then estimate $r_{j_0}=r_{j_0+1}=\dots= r_{j_0+j_1-1}$ by $j_1^{-1}\sum_{i=j_0}^{j_0+j_1-1}\hat r _i $.
Under the null hypothesis in (\ref{apptest2}), using Theorem \ref{thm.fluctuation},  we know that  the $j_1$-dimensional random vector
\begin{align*}
\Big\{\frac{\sqrt{n}(\lambda_{j}-\gamma_j)}{\xi(r_{j_0})},~j\in \{j_0,\dots,j_1-1\}\Big\}
\end{align*}
converges weakly to the distribution of the ordered eigenvalues of  $j_1$-dimensional symmetric Gaussian random  matrix $G$.  Thus we can naturally use  the testing statistics $$T_n:=\frac{\sqrt{n}(\la_{j_0}-\la_{j_0+j_1-1})}{\xi({\hat r_{j_0}})},$$
where $ \hat{r}_{j_0}=j_1^{-1}\sum_{i=j_0}^{j_0+j_1-1}\hat r _i $, to test the hypothesis \eqref{apptest2}. We have the following theorem.
\begin{thm}
	Suppose that the assumptions in  Theorem \ref{thm.fluctuation} hold. Under the null hypothesis of  \eqref{apptest2}, we have for any $x\in \mathbb{R}$,
	\begin{align*}
	|\P(T_n\leq x)-\P(\la^G_{1}-\la^G_{j_1}\leq x)|\to 0,
	\end{align*}
	where $\la^G_1$ and  $\la^G_{j_1}$ are  the largest and smallest  eigeivalues of $j_1$-dimensional symmetric Gaussian random matrix $G$ (defined in Theorem \ref{thm.fluctuation}) respectively. In addition,  if the alternative hypothesis $H_1$ is: there exists some $j_2\in[0,j_1-2]\cap \mathbb{Z}$ such that $ r_{j_0+j_2}$ and $r_{j_0+j_2+1}$ are well separated (c.f. Definition \ref{def.well separated}), we have the power function
	\begin{align*}
	Power=\P(T_n>q_{\alpha}(j_1))\to 1,
	\end{align*}
	where $q_{\alpha}(j_1)$ is the $1-\alpha$  quantile
	of the distribution $\la^G_{1}-\la^G_{j_1}$.
\end{thm}
This theorem can be easily obtained by using Skorohod strong representation  theorem and Theorem \ref{thm.fluctuation}, thus we omit the proof here. For convenient application, 	under the null hypothesis the 95\% empirical quantile
of the distribution $\la^G_{1}-\la^G_{j_1}$ is shown in Table \ref{tab4} with some different $j_1$ by $10^8$ bootstrap replicates.
\begin{table}[htbp]
	\begin{tabular}{ccccccccccc}
		\hline
		&$j_1=2$&$j_1=3$&$j_1=4$&$j_1=5$&$j_1=6$&$j_1=7$&$j_1=8$&$j_1=9$&$j_1=10$&$\cdots$ \\ \hline
		$q_{5\%}(j_1)$    &3.462&4.593&5.459&6.191&6.838&7.424&7.964 &8.468&8.942&$\cdots$  \\ \hline
		
	\end{tabular}
	\caption{ The 5\% empirical quantile
		of the distribution $\la^G_{1}-\la^G_{j_1}$ with 4 digits. }\label{tab4}
\end{table}

Finally, we    give a  procedure of  estimating the population CCC. The algorithm is shown in Table \ref{tab5}.
\begin{table}[htbp]
	\begin{tabular}{|l|}
		\hline
		\\
		Step 1: $\{\la_1\geq\cdots\geq\la_q\}=$ Eigenvalue($C_{yx}$);\\
		Step 2: Test whether $r_1=0$. \\
	\qquad ~~~	If not reject, then continue;\\
		Step 3: $\hat k_0:=\max \{i:\lambda_i\geq d_++\epsilon_n\}$;\\
		Step 4: $\hat r_i=\phi(\la_i )$, $i=1,\dots,\hat k_0$;\\
		Step 5:  If there exists some $ \{j_0,\dots,j_0+j_1-1\}$,\\
		\qquad ~~~ such that $\frac{\sqrt{n}(\la_{j_0}-\la_{j_0+j_1-1})}{\xi({\hat r_{j_0}})}\leq q_{5\%}(j_1)$,\\
			\qquad ~~~  then reset $\hat r_{j_0}=\dots=\hat  r_{j_0+j_1-1}=j_1^{-1}\sum_{i=j_0}^{j_0+j_1-1}\hat r _i $.\\
			Step 6: $\hat\rho_i=\sqrt{\hat r_i}$.
		\\\hline
		
	\end{tabular}
	\caption{ Algorithm for estimating the population CCC. }\label{tab5}
\end{table}

\subsection{The  limestone grassland community data}
To illustrate the application of canonical correlation, we apply our result   to  a
 limestone grassland community data which  can  be  obtained  from  Table A-2 of  \cite{Gittins85C}. This data  records eight species ({\bf x}, $p=8$) and   six  soil variables ({\bf y}, $q=6$)  from  forty-five ($n=44$) 10m $\times$ 10m stands
  in a limestone grassland community
 in Anglesey, North Wales. This experiment concerns relationships between the abundances of several plant
 species and associated soil characteristics.
More explanations can be found in Chapter 7 of  \cite{Gittins85C}. Notice that although the dimensions $p,q$ and the sample size $n$ are not so ``large", from the simulations (some results are not reported in this manuscript), we can already observe the high-dimensional effect when the dimensions are bigger than 5 and sample sizes  are
 bigger than 20, i.e.,  the classical estimates are already not reliable.  The eigenvalues of $C_{yx}$ are as follows:
\begin{align*}
\la_1=0.829,\quad \la_2=0.520,\quad \la_3=0.359,\quad \la_4=0.107,\quad \la_5=0.094,\quad \la_6=0.038.
\end{align*}
%Notice that the result we obtained here is different from that in  \cite{Gittins85C}.
In addition, we have $d_+=0.533$ and $\xi^3_{tw}=0.468$.  Then we use
\eqref{teststa1} to test the hypothesis $r_1=0$, and obtain the p-value to be $3.71\times10^{-5}$. Thus we have strong evidence to  reject the null hypothesis,  which suggests the existence of a  linear relationship between {\bf x} and {\bf y}. Since there  is only one eigenvalue of $C_{yx}$ which is bigger that $d_+$,  we can only determine the existence of the  first CCC  $\hat\rho_1=\sqrt
{\hat r_1} = 0.864$,  which is estimated by \eqref{esthatt}.  For the rest CCC, we have no enough evidence to ensure their existence.

\section{Preliminaries} \label{s. preliminary}
In this section, we introduce some basic notions and known technical results escorting our proofs and calculations in the subsequent sections.

For any given probability distribution $\sigma(\lambda)$, its Stieltjes transform is known as
\begin{eqnarray*}
s_{\sigma}(z):=\int\frac{1}{\lambda-z}d\sigma(\lambda),\quad z\in \mathbb{C}^+:=\{\omega\in\mathbb{C}: \Im \omega>0\}.
\end{eqnarray*}
From the definition, we can immediately get the fact that $\Im s_{\sigma}(z)>0$ for $z\in\mathbb{C}^+$. Actually,
the definition of $s_{\sigma}(z)$ can be extended to the domain $\mathbb{C}\setminus\text{supp}(\sigma)$,
by setting $s_\sigma(\bar{z})=\overline{s_\sigma(z)}$, where $\text{supp}(\sigma)$ represents the support of $\sigma(\lambda)$. Then $s_\sigma(z)$ is holomorphic on $\mathbb{C}\setminus \text{supp}(\sigma)$.

In the sequel, we will also need the Stieltjes transform of the MANOVA matrices $A(A+B)^{-1}$ and $\mathcal{A}(\mathcal{A}+\mathcal{B})^{-1}$, where
\begin{align*}
A \sim \mathcal{W}_p( I_{p}, q),\quad  B\sim\mathcal{W}_p( I_{p}, n-q), \quad \mathcal{A} \sim \mathcal{W}_q( I_{q}, p), \quad \mathcal{B}\sim\mathcal{W}_q( I_{q}, n-p)
\end{align*}
From \cite{BaiH15C}, we know that the Stieltjes transform of $A(A+B)^{-1}$  converges (a.s.) to
\begin{align}
\check{s}(z):= \frac{z-c_1-c_2-\sqrt{(z-d_-)(z-d_+)}}{2c_1z(z-1)}-\frac{1}{z}, \label{17040350}
\end{align}
and  the Stieltjes transform of $\mathcal{A}(\mathcal{A}+\mathcal{B})^{-1}$ converges to
\begin{align}
\tilde{s}(z):= \frac{z-c_1-c_2-\sqrt{(z-d_-)(z-d_+)}}{2c_2z(z-1)}-\frac{1}{z}.  \label{17040351}
\end{align}
Note that $\tilde{s}(z)$ is the Stieltjes transform of the distribution given in (\ref{17040310}), according to the discussions in Section \ref{section. null monova}.

Throughout the paper, we will often use the well known large deviation result of the extreme eigenvalues of Wishart matrices. Assume that $ S \sim \mathscr{W}_{n_1}( I_{n_1}, n_2)$ for some $n_1:=n_1(n)$ and $n_2:=n_2(n)$ satisfying $n_1/n\to a_1\in(0,1)$, $n_2/n\to a_2\in(0,1)$  as $n$ tends to infinity, and $a:=a_1/a_2\in (0,1)$. Denoting $\lambda_1(n_2^{-1} S )$ and $\lambda_{n_1}(n_2^{-1} S )$ the largest and the smallest eigenvalues of $n_2^{-1} S $ respectively, it is well known that for any given positive number $\varepsilon>0$,
\begin{eqnarray}\label{eq4.10}
(1+\sqrt{a})^2+\varepsilon\geq \lambda_1(n_2^{-1} S )\geq \lambda_{n_1}(n_2^{-1} S )\geq (1-\sqrt{a})^2-\varepsilon \label{0726100}
\end{eqnarray}
holds with overwhelming probability ( see Theorem 2.13 of \cite{DS})

\section{Determinant equation} \label{s. determinant eq} In this section, we derive a determinant equation for the eigenvalues of $C_{xy}$ which are not in the spectrum of the null case. We will see that the canonical correlation matrix in the finite rank case can be viewed as a finite rank perturbation of that in the null case.

As mentioned above, the canonical correlation coefficients are invariant under the block diagonal transformation $(\mathbf{x}_i,\mathbf{y}_i)\to(\A\mathbf{x}_i,\B\mathbf{y}_i)$, for any $p\times p$ matrix $\A$ and $q\times q$ matrix $\B$, as long as both of them are nonsingular. Hence, to study $\lambda_i$, we can start with the following setting
\begin{eqnarray*}
\left(
\begin{array}{cc}
\mathscr{X}\\
\mathscr{Y}
\end{array}
\right)=\Sigma^{1/2}
V,\quad \Sigma=\left(
\begin{array}{ccc}
 I_p & T\\
T' &  I_q
\end{array}
\right),
\end{eqnarray*}
where  $V$ is a $(p+q)\times n$ matrix with i.i.d. $N(0,1)$ entries, and \[T=\text{diag}(\sqrt{r_1},\ldots, \sqrt{r_k})\oplus \mathbf{0}_{(p-k)\times (q-k)},\] e.g. see Muirhead \cite{Muirhead1982}, page 530, formula (7).  Apparently, there exists a Gaussian matrix $\mathscr{W}$ independent of $\mathscr{Y}$ such that
\begin{align}
\left(
\begin{array}{cc}
\mathscr{X}\\
\mathscr{Y}
\end{array}
\right)=\left(
\begin{array}{cc}
\mathscr{W}+T\mathscr{Y}\\
\mathscr{Y}
\end{array}
\right). \label{052203}
\end{align}
Here
\begin{align*}
\text{vec}(\mathscr{W})\sim N\big(\mathbf{0},  I_n\otimes  (I_p-TT')\big).
\end{align*}
Hence, $\mathscr{W}$ is a $p\times n$ matrix with independent entries. More specifically, $w_{ij}\sim N(0, 1-r_i)$.
According to the above definitions, we have
\begin{eqnarray*}
\Sigma_{xx}= I_p,\quad \Sigma_{yy}= I_q,\quad \Sigma_{xy}= T,\quad \Sigma_{yx}= T'.
\end{eqnarray*}

For brevity, from now on, we use symbols $ W, X, Y$ to denote the scaled data matrices, i.e.,
\begin{align}
 W:=\frac{1}{\sqrt{n}} \mathscr{W},\qquad  X:=\frac{1}{\sqrt{n}} \mathscr{X},\qquad  Y:=\frac{1}{\sqrt{n}} \mathscr{Y}.  \label{17040101}
\end{align}
Correspondingly, we introduce the notations
\begin{align}
 S _{ww}= W W',\quad  S _{wy}= W Y',\quad  S _{yw}= Y W' ,\quad
 S _{yy}= Y Y'. \label{17052401}
\end{align}
In light of (\ref{052203}), we have
\begin{eqnarray}
& S _{xx}= S _{ww}+ T S _{yw}+ S _{wy} T'+ T S _{yy} T',&\nonumber\\
& S _{xy}= S _{wy}+ T S _{yy},\quad  S _{yx}= S _{yw}+ S _{yy} T'.& \label{061809}
\end{eqnarray}
Hence, due to the assumption that $\text{rank}( T)=k$, $ C_{xy}$ can be regarded as a finite rank perturbation of $ C_{wy}$.
Note that, with probability $1$, the eigenvalues of $ C_{xy}$  are the solutions for $\lambda$ of the characteristic equation
\begin{eqnarray}
\det(D):=\det( S _{xy} S _{yy}^{-1} S _{yx}-\lambda  S _{xx})=0. \label{072501}
\end{eqnarray}
In light of (\ref{061809}), it is equivalent to
\begin{eqnarray*}
\det( S _{wy} S _{yy}^{-1} S _{yw}-\lambda  S _{ww}+(1-\lambda)( T S _{yw}+ S _{wy} T'+ T S _{yy} T'))=0.
\end{eqnarray*}
Denote by
\begin{align}
 P_{y}= Y'( Y Y')^{-1} Y . \label{projection y}
\end{align}
 Then both $ P_{y}$ and $ I_n- P_{y}$ are projections with  $\text{rank}( P_{y})=q $ and $\text{rank}( I_n- P_{y})=n-q$ almost surely.
 Further, we decompose the matrix $ W$ and $ T$ as follows,
\begin{align}
 W=\binom{ W_1}{ W_2} ,\qquad  T=\binom{ T_1}{\mathbf{0}}, \label{16121201}
\end{align}
where $ W_1$ (resp. $ T_1$) is a $ k\times n$ (resp.  $k\times q$) matrix composed by  the first $k$ rows of $ W$ (resp.  $ T$).
Correspondingly, we can introduce the notations like $ S _{y w_2}=YW_2'$, $ S _{w_2w_2}=W_2W_2'$, analogously to (\ref{17052401}). Applying the decompositions in (\ref{16121201}), we can rewrite   (\ref{072501}) in the following form
\begin{align}
\det(D)=\det\left(\begin{array}{cc}{D}_{11}& {D}_{12}\\ {D}_{21} & {D}_{22}\end{array}\right)=0 \label{17051402}
\end{align}
where
\begin{align}
&{D}_{11}\equiv{D}_{11}(\la):= W_1 P_{y} W_1' -\lambda  W_1 W_1'+(1-\lambda)( T_1 Y W_1'+ W_1 Y' T_1'+ T_1 Y Y' T_1'),\nonumber\\
&{D}_{12}\equiv {D}_{12}(\la):= W_1 P_{y} W_2' -\lambda  W_1 W_2'+(1-\lambda) T_1 Y W_2', \nonumber\\
&{D}_{21}\equiv {D}_{21}(\la):= W_2 P_{y} W_1' -\lambda  W_2 W_1'+(1-\lambda) W_2 Y' T_1' \label{16121228}
\end{align}
and
\begin{align}
{D}_{22}\equiv {D}_{22}(\la):= W_2 P_{y} W_2' -\lambda  W_2 W_2'=S_{w_2w_2} (C_{w_2y}-\lambda), \label{16121221}
\end{align}
where $ C_{w_2y}$ is defined analogously to (\ref{17031001}).
In case   ${D}_{22}$ is invertible, we have the equation
 $\det({ D })=\det({D}_{22})\det({D}_{11}-{D}_{12}{D}_{22}^{-1}{D}_{21})=0$, which implies \begin{align}\lb{3.9.1}
\det({D}_{11}(\lambda)-{D}_{12}(\lambda){D}_{22}^{-1}(\lambda) {D}_{21}(\lambda))=0.
\end{align}
Apparently, the sufficient condition of   ${D}_{22}(\lambda)$ being  invertible is that $\la$ is not an eigenvalue of $ C_{w_2y}$. Notice that since $k$ is fixed and $W_2$ is independent of $Y$, we can apply the results in Theorems \ref{thm.070301} and  \ref{thm.071501} to the matrix $ C_{w_2y}$ as well.
Hence,  if we want to investigate the eigenvalues which are not in $\text{Spec} (C_{w_2y})$, it suffices to solve \eqref{3.9.1} and find the  properties of the limits of its solutions.

For brevity, we introduce the notation
\begin{eqnarray}
 M_n(z):={D}_{11}(z)-{D}_{12}(z){D}_{22}^{-1}(z){D}_{21}(z) \label{071801}
\end{eqnarray}
which is a well defined $k\times k$ matrix-valued function for all $z\in \mathbb{C}\setminus \text{Spec}( C_{w_2y})$.
using (\ref{16121228}), (\ref{16121221}) and the fact $YP_y=Y$, we can write
\begin{align}
M_n(z)= \mathcal{Z}_1 \mathscr{A}(z) \mathcal{Z}_1', \label{17050302}
\end{align}
where
\begin{align}
\mathcal{Z}_1=W_1+T_1Y, \qquad
 \mathscr{A}(z)=( P_{y}-z)-( P_{y}-z)  W_2' \big( W_2( P_{y}-z)  W_2'\big)^{-1} W_2( P_{y}-z).  \label{17050301}
\end{align}
Hereafter, we often use $z$ to represent  $zI$ in short when the dimension of the identity matrix is clear.

\section{Limits} \label{s. proof of main theorem}In this section, we provide the proof of Theorem \ref{thm.061901} based on several lemmas, whose proofs will be postponed. Our discussion consists of two parts, aimed at (i) and (ii)  in Theorem \ref{thm.061901}, respectively:
(1) For the outliers, we locate them by deriving the limits of the solutions to the equation (\ref{3.9.1}); (2) For the eigenvalues sticking to $d_{+}$, we simply use Cauchy's interlacing property to get the conclusion.
\begin{itemize}
\item {\it{The outliers}}
\end{itemize}
To locate the outliers, we start with the equation (\ref{3.9.1}), i.e., $\det M_n(z)=0$.
Intuitively, if $ M_n(z)$ is {\it{close}} to some deterministic matrix-valued function $M(z)$ in certain sense, it is reasonable to expect that the solutions of (\ref{3.9.1}) are close to those of the equation $\det [M(z)]=0$. Such an implication can be explicitly formulated in the {\it {location lemma}} given later, see Lemma \ref{lem.062902}. To state the result, we introduce more notations. Set for  any positive constant $\delta$ the domain
\begin{align}
\mathcal{D}\equiv\mathcal{D}(\delta):=\big\{z\in \mathbb{C}: d_{+}+\delta< \Re z\leq 2, |\Im z|\leq 1\big\}. \label{16121240}
\end{align}
 Define the functions $s(z)$, $m_i(z): \mathcal{D}\to \mathbb{C}$ as follows:
\begin{align}
    &s(z):= \frac{z-c_1-c_2-\sqrt{(z-d_{-})(z-d_{+})}}{2 (z-1)}. \label{17040110}\\
  & m_i(z):=\big(c_2-(1+c_1)z-(1-z) s(z)\big)(1-r_i)+(1-z)\big(1-\frac{1-z}{c_2} s(z)\big)r_i, \label{17033101}
\end{align}
where
the  square root  is specified as the one with positive imaginary part in case $z\in \mathbb{C}^+$, and $s(\bar{z})=\overline{s(z)}$.  It is elementary to see that $s(z)$ and $m_i(z)$'s are all holomorphic on $\mathbb{C}\setminus [d_-, d_+]$ and thus also on $\mathcal{D}$. In addition, from the definition (\ref{17040110}), $s(z)$ satisfies the following equation
\begin{align}
(z-1)s^2(z)+(c_1+c_2-z) s(z)-c_1c_2=0. \label{17040301}
\end{align}
Recall $\check{s}(z)$ and $\tilde{s}(z)$ from (\ref{17040350}) and (\ref{17040351}). We see that
\begin{align}
\check{s}(z)=\frac{1}{c_1z}s(z)-\frac{1}{z},\qquad \tilde{s}(z)=\frac{1}{c_2z}s(z)-\frac{1}{z}. \label{17040615}
\end{align}

Further, we  define the diagonal matrix
\begin{eqnarray*}
M(z):=\text{diag}\big(m_1(z),\dots,m_k(z)\big).
\end{eqnarray*}
Recall $\delta$ in the definition of $\mathcal{D}$ in (\ref{16121240}).
We introduce the following event
\begin{align}
& \Xi_1\equiv \Xi_1(n,\delta):=\Big{\{} \| C_{w_2y}\|\leq d_{+}+\frac{\delta}{2}\Big{\}}, \label{17040361}
\end{align}
Note that, Theorem \ref{thm.071501} tells us that $ \Xi_1$ holds with overwhelming probability. Hence, we have $\|(C_{w_2y}-z)^{-1}\|=O(1)$ for all $z\in \mathcal{D}$ with overwhelming probability.  This also implies that, in $ \Xi_1$, $ M_n(z)$ is  holomorphic on $\mathcal{D}$ almost surely.

Our main technical task for the limits part is the following lemma.
\begin{lem} \label{lem.062802}
For any given $\delta>0$, and any sufficiently small $\varepsilon>0$,
\begin{align*}
\sup_{z\in \mathcal{D}}\sup_{i,j=1,\ldots,k}|( M_n)_{ij}(z)- M_{ij}(z)|\leq n^{-\varepsilon}
\end{align*}
holds almost surely.
\end{lem}

The proof of Lemma \ref{lem.062802}  will be postponed  to Section \ref{sec5}. The following  location lemma is a direct consequence of Lemma \ref{lem.062802}. At first, we remark here, it will be clear that the solutions of the equation $\det (M(z)) =0$ can only be real.
\begin{lem}[The location lemma] \label{lem.062902} For any given and sufficiently small $\delta>0$,  let $z_1>\cdots>z_{\kappa}$ be the solutions in $(d_{+}+\delta,1)$ of the equation $\det (M(z))=0$, with multiplicities $n_1,\ldots,n_{\kappa}$ respectively. Then for any fixed $\eta>0$ and each $i\in\{1,\ldots,\kappa\}$, almost surely, there exists $z_{n,i,1}>\cdots>z_{n,i,\kappa_i}$ with multiplicities $m_{i,1},\ldots,m_{i,\kappa_i}$ respectively,  satisfying $\sum_{j=1}^{\kappa_i}m_{i,j}=n_i$, such that
\begin{eqnarray}
\sup_{j=1\ldots, \kappa_i}|z_{n,i, j}- z_i|\leq \eta,  \label{062807}
\end{eqnarray}
and $\cup_{i=1}^{s}\{z_{n,i,j},j=1,\ldots, \kappa_i\}$ is the collection of all solutions (multiplicities counted) of the equation $\det  M_n(z)=0$ in $(d_{+}+\delta,1)$.
\end{lem}

\begin{proof}[Proof of Lemma \ref{lem.062902} with Lemma \ref{lem.062802} granted] At first, as mentioned above, in $ \Xi_1$, $ M_n(z)$ is holomorphic on $\mathcal{D}$. Moreover, according to (\ref{0726100}), both $ Y$ and $ W$ are bounded by some constant in operator norm with overwhelming probability. In addition, we know that $ \Xi_1$ holds with overwhelming probability. Hence, according to the definition of $ M_n(z)$ in (\ref{071801}), one sees that $\| M_n(z)\|$ is bounded uniformly on $\mathcal{D}$ with overwhelming probability. Therefore, the entries of $ M_n(z)$ are all bounded in magnitude with overwhelming probability as well. In addition, it is clear that $M(z)$ is holomorphic and its entries are also bounded in magnitude on $\mathcal{D}$.
Hence, Lemma \ref{lem.062802} implies that almost surely,
\begin{eqnarray*}
\sup_{z\in \mathcal{D}}\big|\det ( M_n(z))-\det (M(z))\big|\leq n^{-\varepsilon},
\end{eqnarray*}
taking into account the fact that the determinant is a multivariate polynomial of the matrix entries. It is obvious that $\det [ M_n(z)] $  only has real roots since the equation (\ref{072501}) does. Then by Rouche's theorem, we can immediately get (\ref{062807}).
\end{proof}
 Now, with the aid of Lemma \ref{lem.062902}, we prove (i) of Theorem \ref{thm.061901}.
\begin{proof}[Proof of (i) of Theorem \ref{thm.061901}] According to Lemma \ref{lem.062902}, it suffices to solve $\det (M(z)) =0$ in $(d_{+}+\delta, 1]$ to get $z_i$, for sufficiently small $\delta>0$. By the definition of $M(z)$, we shall solve the equation
 \begin{align}\lb{071601}
m_i(z)=0.
\end{align}

It  will be clear that there is a unique simple solution for the above equation. We denote it by $\gamma_i$ in the sequel.

First, if $r_i=1$, we get from the definition (\ref{17033101}) that
\begin{align*}
m_i(\ga_i)=(1-\ga_i)\Big(1-\frac{1-\ga_i}{c_2} s(\ga_i)\Big)=0.
\end{align*}
It is elementary to check that the only solution is $\ga_i=1$.  If $r_i<1$, we introduce the shorthand notation
\begin{align}
t_i=\frac{r_i}{1-r_i}. \label{170529100}
\end{align}

According to the definition in (\ref{17033101}), we see that (\ref{071601}) is equivalent to
\begin{eqnarray}\lb{0623002}
&&\sqrt {\ga_i^2+2(2c_1c_2-c_1-c_2)\ga_i+(c_1-c_2)^2} \left( c_2-{t_i}\ga_i+
{t_i} \right) \no
&=&t_i\ga_i^2+(t_i(c_2-c_1-1)+c_2-2c_1c_2)\ga_i+(c_1-c_2)t_i+c_1c_2-c_2^2.
\end{eqnarray}
Notice that if $\ga_i\in (d_+, 1]$, we have
\begin{align*}
\ga_i^2+2(2c_1c_2-c_1-c_2)\ga_i+(c_1-c_2)^2=(\ga_i-d_{-})(\ga_i-d_{+})> 0
\end{align*}
and $ c_2-{t_i}\ga_i+{t_i} > 0.$
Thus
under the restriction that
\begin{eqnarray}
t_i\ga_i^2+(t_i(c_2-c_1-1)+c_2-2c_1c_2)\ga_i+(c_1-c_2)t_i+c_1c_2-c_2^2\geq  0,\label{062901}
\end{eqnarray}
solving (\ref{0623002}), we get
\begin{eqnarray}
\gamma_i=\frac{(1+t_i^{-1}c_1)(1+t^{-1}_ic_2)}{1+t^{-1}_i}=r_i(1-c_1+c_1r_i^{-1})(1-c_2+c_2r_i^{-1}). \label{060901}
\end{eqnarray}
Applying (\ref{170529100}) and  \eqref{060901}, it is  elementary to check that (\ref{062901}) is equivalent to
\begin{align}
r_i\geq r_c= \sqrt{\frac{c_1c_2}{(1-c_1)(1-c_2)}}. \label{17053001}
\end{align}
Note that
\begin{eqnarray}
\gamma_i&=&c_1(1-c_2)+c_2(1-c_1)+r_i(1-c_1)(1-c_2)+r_i^{-1}c_1c_2\nonumber\\
&&\geq c_1(1-c_2)+c_2(1-c_1)+2\sqrt{c_1c_2(1-c_1)(1-c_2)}=d_{+} \label{0607002}
\end{eqnarray}
Moreover, equality holds in the second step of (\ref{0607002})  only if $r_i=r_c$.
 Hence, (\ref{071601}) has solution in $(d_+,\infty)$ only if $r_i> r_c$, with the solution $\gamma_i$ given by (\ref{060901}).
Now, what remains is to check $\gamma_i < 1$. It again follows  from (\ref{17053001}) by elementary calculation.
 Then by (\ref{062807}) in Lemma \ref{lem.062802}, we get that
\begin{eqnarray*}
\lambda_i-\gamma_i\stackrel{\text{a.s.}} \longrightarrow 0, \quad \text{for} \quad i=1,\ldots, k_0.
\end{eqnarray*}
Hence, we conclude the proof of (i) of Theorem  \ref{thm.061901}.
\end{proof}

Now, we  proceed to the  proof of (ii) .
\begin{itemize}
\item {\it{The Sticking eigenvalues}}
\end{itemize}
\begin{proof}[Proof of (\romannumeral2) of Theorem \ref{thm.061901}]
At first, according to the proof of (i) of Theorem \ref{thm.061901}, we see that with overwhelming probability, there are exactly $k_0$ largest eigenvalues  (multiplicities counted) of $ C_{xy}$ in the interval $(d_{+}+\delta,1]$, for any sufficiently small $\delta>0$. Hence, we see that $\limsup_{n\to\infty}\lambda_i\leq d_{+}+\delta$ for all $i\geq k_0+1$ almost surely. Moreover, by the Cauchy's interlacing property  and (\ref{071501}), we always have $\liminf_{n\to\infty}\lambda_i\geq d_{+}-\delta$ almost surely for any fixed $i$. Since $\delta$ can be arbitrarily small, we get the conclusion that $\lambda_i-d_{+}\stackrel{a.s.}\longrightarrow 0$ for any fixed $i\geq k_0+1$.  Hence, we complete the proof.
\end{proof}

\section{Fluctuations of the outliers} \label{section.fluctuation}
 For the fluctuations of the outliers, we first derive a CLT for the entries of $ M_n$, which together with \eqref{3.9.1}  will imply a CLT for the outliers.  Recall from Lemma \ref{lem.062902} the notations $n_1,\ldots, n_\kappa$ as the multiplicities of the solutions of $\det [M(z)] =0$. Set
$$J_l=\Big\{\sum_{i=1}^{l-1}n_i+1,\dots,\sum_{i=1}^{l}n_i\Big\}.$$
We derive the central limit theorem for the $n_l$-packed sample eigenvalues $
\{\sqrt{n}(\lambda_{j}-\gamma_j),~j\in J_l\}.
$
Note that $\ga_j$'s are all the same for $j\in J_l$, and they are $l$-th largest  solution of equation $\det [M(z)] =0$  with  multiplicity $n_l$.
Actually, from  Lemma \ref{lem.062802} and Lemma \ref{lem.062902}, for $j\in J_l$, it would not be difficult to conclude that for any sufficiently small $\varepsilon>0$,
\begin{align}
 \sup_{1\leq \alpha, \beta \leq k}\big|( M_n)_{\alpha\beta}(\lambda_{j})- M_{\alpha\beta}(\gamma_j)\big|\leq n^{-\varepsilon}\label{16121202}
\end{align}
almost surely. However,
for Theorem \ref{thm.fluctuation},  we need   to take a step further to  consider the fluctuation of  the difference in (\ref{16121202}). We split the task into three steps, and the results are collected in Lemmas \ref{lem4.3}-\ref{lem4.5}.  The first lemma is to expand $ M_n(\lambda_{j})$ around $ M_n(\gamma_{j})$.   Recall the notations $\vartheta, \varpi$ defined in (\ref{17050702}).
{\begin{lem}\lb{lem4.3}
Suppose that the assumptions in Theorem \ref{thm.fluctuation} hold. Especially,  $j\in \{1, \ldots, k_0\}$ such that $r_j$ is well separated from $r_c$ and $1$. As $n\to \infty$, we have
\begin{align*}
\sqrt{n}[ M_n(\lambda_{j})- M_n(\gamma_{j})]=-\sqrt{n}(\lambda_{j}-\gamma_{j})\big( \Delta+o_{\text{p}}(1)\big),
\end{align*}
where $\De$ is a diagonal matrix defined as
\begin{align*}\De=\frac{(1-c_1)\vartheta r_j^2+c_1^2c_2}{\vartheta r_j^2-\varpi} \big(I_k-T_1 T_1'\big)
+\frac{(1-c_1)\vartheta r_j^3+c_1(3c_1c_2-c_1-2c_2+1)r_j-2c_1^2c_2 }{r_j(\vartheta r_j^2-\varpi)} T_1 T_1',
\end{align*}
and $o_{\text{p}}(1)$ represents a generic $k\times k$ matrix with negligible (in probability) entries.
\end{lem}

 Recall $\mathscr{A}(z)$ defined in (\ref{17050301}). We further introduce the following matrices
 \begin{align}
  \mathscr{M}_1&\equiv \mathscr{M}_1(z):= W_1\mathscr{A}(z) W_1', \nonumber\\
  \mathscr{M}_2&\equiv \mathscr{M}_2(z):=T_1Y\mathscr{A}(z) W_1', \nonumber\\
 \mathscr{M}_3&\equiv \mathscr{M}_3(z):=T_1Y\mathscr{A}(z) Y'T_1'.
 \label{16123001}
 \end{align}
 According to the definitions in  (\ref{071801}), (\ref{16121228}) and (\ref{16121221}),  it is easy to check that
 \begin{align}
  M_n(z)=\scM_1+\scM_2+\scM_2'+\scM_3 . \label{17031505}
 \end{align}

  We denote the SVD of $ Y$ by
\begin{align}
 Y=U_{y} \Lambda_{y} V_{y}. \label{17032720}
\end{align}
It is well known that $U_{y}$ and $V_{y}$ are Haar distributed, and $U_{y}, \Lambda_{y}, V_{y}$ are mutually independent (see (3.9) of \cite{ER05} for instance).
Then we see that
\begin{align}
 P_{y}=V_{y}'  \Lambda_{y}' \big( \Lambda_{y}\Lambda_{y}'\big)^{-1} \Lambda_{y}  V_{y}. \label{17040601}
\end{align}
is independent of $U_{y}$.  With the above notations, we introduce
\begin{align}
\mathscr{B}\equiv \mathscr{B}(z):=\Lambda_{y} V_{y}\mathscr{A}(z),\qquad  \mathscr{D}\equiv \mathscr{D}(z):= \Lambda_{y} V_{y}\mathscr{A}(z)V_{y}' \Lambda_{y} '. \label{17032301}
\end{align}
Therefore, we can write
\begin{align}
 \mathscr{M}_1= W_1\mathscr{A} W_1', \qquad \mathscr{M}_2&= T_1 U_{y}\mathscr{B} W_1', \qquad \mathscr{M}_3= T_1 U_{y}\mathscr{D} U_{y}' T_1'. \label{17040420}
 \end{align}

Set
\begin{align*}
\mathcal{M}(\ga_j):=\frac{1}{n}\ntr \mathscr{A}(\ga_j) \cdot \big(I_k-T_1T_1'\big)+\frac{1}{q} \ntr \mathscr{D}(\ga_j) \cdot  T_1 T_1'
\end{align*}

 The next lemma depicts the fluctuation of $ M_n(\gamma_{j})$ around $\mathcal{M}(\ga_j)$.
\begin{lem}\lb{lem4.4}
Suppose that the assumptions in Theorem \ref{thm.fluctuation} hold. For any $j\in \{1, \ldots, k_0\}$ such that $r_j$ is well separated from $r_c$ and $1$, we have that
\begin{align*}
\sqrt{n}\Big( M_n(\gamma_{j})-\mathcal{M}(\ga_j)\Big)
\end{align*}
converges weakly to a ${k\times k}$ symmetric  Gaussian matrix $ \mathscr{R} $ with independent (up to symmetry) mean zero entries and the variance of $\mathscr{R} _{\alpha\beta}$ is
\begin{align}
&\frac{1+\de_{\alpha\beta}}{r_j^2(\vartheta r_j^2-\varpi )}\Big(\big((1-c_1)r_j+c_1\big)^2\big(\vartheta r_j-\varpi\big)^2\big((1-c_2)r_j^2+c_2\big)  (1-r_\alpha)(1-r_\beta)\nonumber\\
&\qquad+(1-c_2)\big((1-c_1)r_j+c_1\big)^2(1-r_j)^2\big(\vartheta r_j-\varpi\big)^2(r_{\alpha}(1-r_\beta)+r_{\beta}(1-r_\alpha)) \nonumber\\
&\qquad+\big((1-c_1)r_j+c_1\big)^2\big(\vartheta r_j-\varpi \big)^2\big((1-c_2)r_j^2+2c_1(1-c_2)r_j+c_1(1-2c_2)\big) r_\alpha r_\beta  \Big). \label{17040788}
\end{align}
\end{lem}

Finally, we claim that the difference between $\mathcal{M}(\ga_j)$ and $ M(\gamma_{j})$ is of order $o_{\text{p}}(\frac{1}{\sqrt{n}})$.
\begin{lem}\lb{lem4.5}
Suppose that the assumptions in Theorem \ref{thm.fluctuation} hold. For any $j\in \{1, \ldots, k_0\}$ such that $r_j$ is well separated from $r_c$ and $1$, we have that
\begin{align*}
\sqrt{n}\big(\mathcal{M}(\ga_j)- M(\gamma_{j})\big)=o_{\text{p}}(1),
\end{align*}
where $o_{\text{p}}(1)$ represents a generic $k\times k$ matrix with negligible (in probability) entries.
\end{lem}
}
The proofs of Lemmas \ref{lem4.3}-\ref{lem4.5}  will be postponed to   Section \ref{sec6}.   With the aid of these lemmas, we can now prove Theorem \ref{thm.fluctuation}.

  \begin{proof}[Proof of Theorem  \ref{thm.fluctuation}]
  Note that  $\det  M_n(\la_{j})=0$,  and $ M(\ga_{j})$ is a diagonal matrix with diagonal
elements $m_i(\ga_j)$, $i=1,\dots,k$ (c.f. (\ref{17033101})).  In addition, for any $l=1,\dots,l_0$, $m_i(\ga_j)$ is zero if $i\in J_l$ and nonzero otherwise.
Therefore by Lemmas \ref{lem4.3}-\ref{lem4.5}, the definition of $ T_1$  and   Skorohod strong representation  theorem, we conclude that there exists a random matrix $\mathscr{R}$ defined in Lemma \ref{lem4.4}, such that  for any $j, \alpha, \beta\in J_l $,
\begin{align}\lb{lim11}
\sqrt{n}\Big(( M_n)_{\alpha\beta}(\la_j)- M_{\alpha\beta }(\gamma_{j})\Big)= -\de_{\alpha\beta}\De_{jj}\sqrt{n}(\lambda_{j}-\gamma_{j})(1+o_{\text{p}}(1))
+\mathscr{R} _{\alpha\beta}+o_{\text{p}}(1).
\end{align}
Let $\mathscr{N}^{(l)}$ be a $k\times
k$ diagnoal matrix with $\mathscr{N}^{(l)}_{jj}=n^{1/4}$ if $j\in J_l$ and $\mathscr{N}^{(l)}_{jj}=1$ otherwise. Apparently we have  $\det [\mathscr{N}^{(l)}]=n^{n_l/4}\neq 0$.
Combining Theorem \ref{thm.061901} with  Lemmas \ref{lem4.3}-\ref{lem4.5}, we conclude that
\begin{align*}
\mathscr{N}^{(l)} M_n(\la_j)\mathscr{N}^{(l)}= M(\gamma_{j})+\big(\sqrt{n}( M_n)_{\alpha\beta}(\la_j)\mathbf{1}(j,\alpha, \beta\in J_l)\big)_{k\times k}+o_p(1).
\end{align*}
Notice that $ M(\gamma_{j})$ is  a diagonal matrix with $J_l$ diagonal
elements being zeros and
\begin{align*}
\big(\sqrt{n}( M_n)_{\alpha\beta}(\la_j)\mathbf{1}(j,\alpha,\beta\in J_l)\big)_{k\times k}
\end{align*}
 is a null matrix except the $J_l\times J_l$ block.
According to  $\det (\mathscr{N}^{(l)} M_n(\la_j)\mathscr{N}^{(l)})=0$ and \eqref{lim11}, we obtain that
 \begin{align}
\det\Big(-\lim_{n\to \infty}\Big(\sqrt{n}(\lambda_{j}-\gamma_{j})( I_{n_l}+o_{\text{p}}(1))
+\mathscr{R}(J_l\times J_l)/\De(j,j)\Big)\Big)=0, \label{17033111}
 \end{align}
 where $\mathscr{R}(J_l\times J_l)$ represents the $J_l\times J_l$ block of $\mathscr{R}$.
From (\ref{17033111}), we see that  $\lim_{n\to\infty}\sqrt{n}(\lambda_{j}-\gamma_{j})$ is the eigenvalue of $\mathscr{R}(J_l\times J_l)/\De_{jj}$.   Since the eigenvalues of $ C_{yx}$ are simple almost surely, we see that the $n_l$ random variables $\{\sqrt{n}(\lambda_{j}-\gamma_j),~j\in J_l\}$ converge in probability to the set of eigenvalues of the  matrix $\mathscr{R}(J_l\times J_l)/\Delta_{jj}$. The variance in (\ref{17040780}) can be obtained from  the definition of $\Delta$ in Lemma \ref{lem4.3}, and also  (\ref{17040788}).

 Therefore, we complete the proof of Theorem \ref{thm.fluctuation}.

\end{proof}

\section{Fluctuations of the sticking eigenvalues} \label{s. fluctuation of sticking}

In this section, we prove Theorem \ref{thm. fluctuation of sticking evs}. When $k=1$, the main idea is to compare the eigenvalues of $C_{xy}$ with those of $C_{w_2y}$ directly.   Then we use the Tracy-Widom law for the largest eigenvalues of $C_{w_2y}$ to conclude the fluctuation of the  sticking eigenvalues of $C_{xy}$.  For more general $k$, we use the argument for the case $k=1$ recursively. Our proof strategy is inspired by the work  \cite{BGM2011}.

For the comparison, throughout this section,  we denote by $\widetilde{\lambda}_1\geq  \widetilde{\lambda}_2\geq \cdots\geq \widetilde{\lambda}_q$ the ordered
eigenvalues of $ C_{y w_2}$,  while use $\lambda_1\geq \lambda_2\geq \cdots \geq \lambda_q$ to denote the ordered eigenvalues of $ C_{yx}$ as before.

Our aim is to show the following proposition.
\begin{pro}\label{pro. comparison of evs} Suppose that the assumptions in Theorem \ref{thm. fluctuation of sticking evs} holds. Then for any small positive constant $\varepsilon$ and any fixed  integer $i\geq 1$, the following holds in probability
\begin{align}
|\lambda_{k_0+i}-\widetilde{\lambda}_i|\leq n^{-1+\varepsilon}.  \label{17052820}
\end{align}
\end{pro}

With the aid of Proposition \ref{pro. comparison of evs}, we can prove Theorem \ref{thm. fluctuation of sticking evs} below.
\begin{proof}[Proof of Theorem \ref{thm. fluctuation of sticking evs}] Let
\begin{align}
\hat{X}:=W+T\hat{Y},
\end{align}
where $\hat{Y}$ is a i.i.d. copy of $Y$ and $\hat{Y}$ is independent of $Y$. Denote by $C_{y\hat{x}}$ via replacing $X$ by $\hat{X}$ in $C_{yx}$. Then $C_{yx}$ is a CCA matrix in the null case. We first  compare the eigenvalues of $C_{y\hat{x}}$, denoted by $\mathring{\lambda}_i$'s, with the eigenvalues of $C_{yw_2}$, i.e., $\widetilde{\lambda}_i$.  Applying  Proposition \ref{pro. comparison of evs} to the case $r_1=\cdots=r_k=0$, we can get $\max_{1\leq i\leq \mathfrak{m}}|\mathring{\lambda}_i-\widetilde{\lambda}_i|\leq n^{-1+\varepsilon}$ in probability. This together with (\ref{17052820}) implies (\ref{17052821}). This together with Remark \ref{rem.17052825} completes the proof of Theorem \ref{thm. fluctuation of sticking evs}.
\end{proof}

To prove Proposition \ref{pro. comparison of evs}.  We will need the following  lemmas.
\begin{lem} \label{lem. monotonicity} For any $i=0,\ldots, q$, and $j=1,\ldots, k$, the diagonal entry $( M_n)_{jj}(z)$ is a decreasing function in $z\in(\widetilde{\lambda}_{i+1},\widetilde{\lambda}_i )$, where $\widetilde{\lambda}_0=\infty$ and  $\widetilde{\lambda}_{q+1}=-\infty$.
\end{lem}
\begin{proof} Recall   (\ref{17050302}) and (\ref{17050301}).
It is elementary to check that
\begin{align*}
\frac{\partial \mathscr{A}(z)}{\partial z}= - \mathcal{Q}(z)(\mathcal{Q}(z))^{T},
\end{align*}
where
\begin{align*}
\mathcal{Q}(z):= I+( P_{y}-z)  W_2' \big( W_2( P_{y}-z)  W_2'\big)^{-1} W_2
\end{align*}
Note that, as long as $z\not\in \text{Spec}(  C_{w_2y})$, $\frac{\partial \mathscr{A}(z)}{\partial z}$ is a well-defined negative semi-definite matrix.  Hence,
\begin{align*}
\frac{\partial }{\partial z} ( M_n)_{jj}(z)= \Big( \mathcal{Z}_1 \frac{\partial \mathscr{A}(z)}{\partial z}\mathcal{Z}'_1\Big)_{jj}\leq 0
\end{align*}
in the interval $(\widetilde{\lambda}_{i+1}, \widetilde{\lambda}_i)$ for any $i=0, \ldots, q$, which means that $( M_n)_{jj}(z)$ is decreasing.
\end{proof}

Let $\mathfrak{m}\geq 1$ be any fixed  integer. For  any small constants $\varepsilon, \delta>0$, we introduce the following random domain
\begin{align}
\Omega\equiv \Omega(\mathfrak{m}, \varepsilon, \delta):=  (\widetilde{\lambda}_1+n^{-1+\varepsilon}, d_++\delta)\cup  \Big\{\cup_{i=1}^{\mathfrak{m}} (\widetilde{\lambda}_{i+1}+n^{-1+\varepsilon}, \widetilde{\lambda}_i-n^{-1+\varepsilon})\Big\} \label{17052410}
\end{align}

We also need the following lemma.
\begin{lem} \label{lem.17042902}  With the above notations,   we have the following:

(i): If $i\neq j$, for any small $\varepsilon, \delta>0$, there exists some positive constant $c>0$, such that
\begin{align*}
\sup_{z\in \Omega} |( M_n)_{ij}(z)|\leq n^{-c}
\end{align*}
holds with overwhelming probability.

(ii):  For  any small $\varepsilon>0$ and any small $\varepsilon'>0$, there exists a sufficiently small $\delta$, such that
\begin{align*}
\sup_{z\in \Omega} |( M_n)_{ii}(z)-m_i(d_+)|\leq \varepsilon'
\end{align*}
holds with overwhelming probability.
\end{lem}
The proof of Lemma \ref{lem.17042902}  is postponed to Appendix B.

In the sequel, we first show the proof of Proposition \ref{pro. comparison of evs} in the case of $k=1$.  At the end of this section, we will extend the discussion to general $k$ case.   Note that when $k=1$, We have $k_0=1$ or $k_0=0$, which means $r_1>r_c$ or $r_1<r_c$, respectively. We restate Proposition \ref{pro. comparison of evs} for $k=1$ in the following proposition.
\begin{pro} \label{pro.17042901} Suppose that $k=1$. We have the following

(i): If $r_1>r_c$, then for any fixed integer $m\geq 1$, and any  $\varepsilon>0$, we have
\begin{align*}
|\lambda_{m+1}- \widetilde{\lambda}_m |\leq n^{-1+\varepsilon}
\end{align*}
in probability.

(ii): If $r_1< r_c$,  then for any fixed integer $m\geq 1$, and any  $\varepsilon>0$, we have
\begin{align*}
|\lambda_{m}- \widetilde{\lambda}_m |\leq n^{-1+\varepsilon}
\end{align*}
in probability.
\end{pro}

With Lemma \ref{lem. monotonicity} and Lemma \ref{lem.17042902}, we  prove Proposition \ref{pro.17042901}.

\begin{proof}[Proof of Proposition \ref{pro.17042901}] Since  $k=1$, we have $\det  M_n= M_n(z)= M_{11}(z)$. Applying Lemma \ref{lem.17042902} with $k=1$, we see that  for any small $\varepsilon'$, there is a sufficiently small $\delta>0$ such that
\begin{align}
\sup_{z\in \Omega} |\det  M_n (z)-m_1(d_+)|\leq \varepsilon' \label{17042905}
\end{align}
with overwhelming probability.  Since we assume that $r_1$  is well separated from $r_c$ , it is easy to check from the definition (\ref{17033101}) that $|m_1(d_+)|\geq c$ for some positive constant $c$.  Hence, choosing $\varepsilon'$ in (\ref{17042905}) sufficiently small,  we see that $|\det  M_n(z)|\neq 0$ uniformly on $\Omega$, with overwhelming probability. That means,   with overwhelming probability, there is no eigenvalue of $ C_{xy}$ in $\Omega$.

Let $\mathbf{w}_i, i=1, \ldots, p$ be the rows of $W$. Note that, in case of $k=1$, we have $\mathbf{w}_1= W_1$. Moreover, $ P_{w}$ is the projection onto the subspace spanned by $\{\mathbf{w}_1+ T_1  Y, \mathbf{w}_2, \cdots, \mathbf{w}_p\}$, and $ P_{w_2}$ is the projection onto the subspace spanned by $\{\mathbf{w}_2, \cdots, \mathbf{w}_p\}$.  Then, by Cauchy interlacing, we know that  $\text{Spec}( P_{w}  P_{y} P_{w})$ and  $ \text{Spec}( P_{w_2}  P_{y} P_{w_2})$  are interlacing. This implies that
\begin{align}
\lambda_1\geq \widetilde{\lambda}_1\geq \lambda_2\geq \cdots \geq \widetilde{\lambda}_{q-1}\geq \lambda_q\geq \widetilde{\lambda}_{q},  \label{17042906}
\end{align}
since the nonzero eigenvalues of  $ C_{wy}$ (resp.   $ C_{w_2y}$) are the same as those of $ P_{w}  P_{y} P_{w}$ (resp. $ P_{w_2}  P_{y} P_{w_2}$).

In case (i): $r_1>r_c$, from Theorem \ref{thm.061901} we  know that  $\lambda_1\to \gamma_1>d_++\delta$ almost surely, for any sufficiently small $\delta>0$.
Since $\lambda_i$'s are solutions to $\det  M_n(z)=0$, and $m_1(d_+)\neq 0$, we see from (\ref{17042905}) and (\ref{17042906}) that
\begin{align*}
 \lambda_i\in [\widetilde{\lambda}_{i+1}, \widetilde{\lambda}_{i+1}+n^{-1+\varepsilon} ]\cup [ \widetilde{\lambda}_{i}-n^{-1+\varepsilon}, \widetilde{\lambda}_{i}], \qquad i=2, \ldots, \mathfrak{m}.
\end{align*}
in probability.

In case (ii): $r_1<r_c$, from Theorem \ref{thm.061901} we know that $\lambda_1\to d_+$ in probability. Hence, we have
\begin{align*}
\lambda_1\in [\widetilde{\lambda}_1, \widetilde{\lambda}_1+n^{-1+\varepsilon}],\qquad  \lambda_i\in [\widetilde{\lambda}_{i+1}, \widetilde{\lambda}_{i+1}+n^{-1+\varepsilon} ]\cup [ \widetilde{\lambda}_{i}-n^{-1+\varepsilon}, \widetilde{\lambda}_{i}], \qquad i=2, \ldots, \mathfrak{m}.
\end{align*}
Therefore, to prove Proposition \ref{pro.17042901}, it suffices to check that for all $i=2, \cdots, \mathfrak{m}$, $ \lambda_i$ is in $[\widetilde{\lambda}_{i+1}, \widetilde{\lambda}_{i+1}+n^{-1+\varepsilon} ]$ or $[ \widetilde{\lambda}_{i}-n^{-1+\varepsilon}, \widetilde{\lambda}_{i}]$. Note that $\det  M_n(z)=( M_n)_{11}(z)$ is decreasing in the interval $(\widetilde{\lambda}_{i+1}, \widetilde{\lambda}_{i})$, according to Lemma \ref{lem. monotonicity}. Furthermore, since $ s(d_+)=\frac{d_+-c_1-c_2}{2(d_+-1)}$, it is elementary to check
\begin{align*}
m_j(d_+)>0, \quad \text{if}\quad r_j>r_c, \qquad m_j(d_+)<0, \quad \text{if}\quad r_j<r_c.
\end{align*}
Therefore, if $r_1>r_c$, according to  (\ref{17042905}), we have  $\det  M_n(z)>0$ on  $(\widetilde{\lambda}_{i+1}+n^{-1+\varepsilon},  \widetilde{\lambda}_{i}-n^{-1+\varepsilon})$. By the monotonicity of $\det  M_n(z)$, we see that $\det  M_n(z)=0$ can only be achieved if $z\in [ \widetilde{\lambda}_{i}-n^{-1+\varepsilon}, \widetilde{\lambda}_{i}]$. Hence, in case of $r_1>r_c$, we have $ \lambda_i\in [ \widetilde{\lambda}_{i}-n^{-1+\varepsilon}, \widetilde{\lambda}_{i}]$  for all $i=2, \ldots, \mathfrak{m}$. In contrast, when $r_1<r_c$, we have $\lambda_i\in [\widetilde{\lambda}_{i+1}, \widetilde{\lambda}_{i+1}+n^{-1+\varepsilon} ]$ for all $i=2, \ldots, \mathfrak{m}.$

This concludes the proof of Proposition \ref{pro.17042901}.

\end{proof}

Now, we prove  Proposition  \ref{pro. comparison of evs}.
\begin{proof}[Proof of Proposition  \ref{pro. comparison of evs}] For general $k$, we need  more notations.  For $1\leq a\leq k$, let
\begin{align*}
W_{1a}:=\left(\begin{array}{cc}
\mathbf{w}_{1}\\
\vdots \\
\mathbf{w}_a
\end{array}
\right), \qquad T_{1a}:=\left(\begin{array}{cccccc}
t_1 & &   &  \\
 &\ddots &  &\mathbf{0}   \\
 & & t_a   &
\end{array}\right)
\end{align*}
be the $a\times n$ matrices composed of the first  $a$ rows of $W$ and $T$,  respectively, and we write
\begin{align*}
W=\binom{W_{1a}}{W_{2a}}, \qquad T=\binom{T_{1a}}{T_{2a}}.
\end{align*}
We take the above as the definitions of $W_{2a}$ and $T_{2a}$.  Further, we denote
\begin{align*}
X_{a}:= W_{2a}+T_{2a} Y.
\end{align*}
Denote the eigenvalues of $C_{x_ay}$ by
\begin{align}
\lambda_{1,a}\geq \lambda_{2,a}\geq \cdots \label{17052470}
\end{align}
Especially, with the above notations, we have $\widetilde{\lambda}_i=\lambda_{i,k}$ and $\lambda_i=\lambda_{i,0}$.
Our aim is compare $\lambda_{i,a-1}$'s with $\lambda_{i,a}$'s. First of all, by Cauchy interlacing, we have
\begin{align}
\lambda_{i-1,a}\geq \lambda_{i,a-1}\geq  \lambda_{i,a}. \label{17051401}
\end{align}
Further, if $\lambda_{i,a-1}\not\in \text{Spec}(C_{x_ay})$, then $\lambda_{i,a-1}$ is a solution to the equation
\begin{align}
\mathbf{z}_{a-1} \mathcal{P}_a \mathbf{z}_{a-1}'=0, \label{17051301}
\end{align}
where
\begin{align*}
\mathbf{z}_{a-1}:= \mathbf{w}_{a-1}+t_{a-1} \mathbf{y}_{a-1},
\end{align*}
and
\begin{align*}
\mathcal{P}_a(z):= (P_y-z)-(P_y-z) X_{a}' (X_{a}(P_y-z)X_{a}')^{-1}X_{a} (P_y-z).
\end{align*}

We start with the case $a=k-1$ and compare $\lambda_{i,a-1}$'s with $\lambda_{i,a}$'s. This is equivalent to discuss the case $k=2$.  Hence, we can set $k=2$ and $a=1$. We recall the decomposition of $D$ in (\ref{17051402}). By Schur complement, we have
\begin{align}
D^{-1}=\left(
\begin{array}{ccc}
D_{11} & D_{12}\\
D_{21}  & D_{22}
\end{array}
\right)^{-1}= \left( \begin{array}{ccc}
(M_n)^{-1} & *\\
* & *
\end{array}
\right), \label{17051404}
 \end{align}
where  $*$'s represent other blocks, whose explicit expressions are irrelevant to us. Observe that $M_n$ (c.f. (\ref{071801}))
is $2\times 2$ by the assumption $k=2$. On the other hand, using a different decomposition,  we can also get by Schur complement that
\begin{align}
D^{-1}=\left( \begin{array}{cccccc}
\mathbf{z}_1(P_y-z)\mathbf{z}_1' & \mathbf{z}_1(P_y-z)X_2\\
X_2'(P_y-z)\mathbf{z}_1'  & X_{2}(P_y-z)X_{2}'
\end{array}
\right)^{-1}=  \left( \begin{array}{ccc}
(\mathbf{z}_{1} \mathcal{P}_2 \mathbf{z}_{1}')^{-1} & *\\
* & *
\end{array}
\right).  \label{17051405}
\end{align}
From  (\ref{17051404}) and  (\ref{17051405}), we obtain
\begin{align}
\mathbf{z}_{1} \mathcal{P}_2 \mathbf{z}_{1}'= \frac{1}{ (M_n^{-1})_{11}}.  \label{17051420}
\end{align}

Then, we first consider the eigenvalues of $C_{x_1y}$, which is a perturbation of $C_{w_2y}$. Recall the notations introduced in (\ref{17052470}).  Using Proposition \ref{pro.17042901}, we see that the following statements hold in probability:

(i): If $r_2>r_c$, we have for any fixed integer $m\geq 0$
\begin{align}
|\lambda_{m+1,1}- \widetilde{\lambda}_m |\leq n^{-1+\varepsilon}. \label{17051410}
\end{align}

(ii): If $r_2<r_c$, we have for any fixed integer $m\geq 0$,
\begin{align}
|\lambda_{m,1}- \widetilde{\lambda}_m |\leq n^{-1+\varepsilon}.  \label{17051436}
\end{align}

According to the assumption $r_1\geq r_2$, it suffices to consider the following  three cases:
\begin{align*}
\text{(i)}':\; r_1\geq r_2>r_c,\quad  \text{(ii)}':\; r_1>r_c>r_2,\quad  \text{(iii)}': \; r_c>r_1>r_2.
 \end{align*}
 First, we consider case (i)'.   Applying Theorem \ref{thm.061901}, we see that
 \begin{align*}
 \lambda_{1,0}\to \gamma_1, \quad \lambda_{2,0}\to \gamma_2, \quad \lambda_{m,0}\to d_+
\end{align*}
almost surely. Recall Cauchy interlacing formula (\ref{17051401}). We see that $\lambda_{m,1} \geq \lambda_{m+1,0}\geq \lambda_{m+1,1}$.  We first show
\begin{align}
\lambda_{m+1,0}\not\in  (\widetilde{\lambda}_{m}+n^{-1+\varepsilon}, \widetilde{\lambda}_{m-1}-n^{-1+\varepsilon}) \label{17051415}
\end{align}
in probability.
Then according to (\ref{17051410}), we have
 \begin{align}
\lambda_{m+1,0}\not\in  (\lambda_{m+1,1}+2n^{-1+\varepsilon}, \lambda_{m,1}-2n^{-1+\varepsilon})
\label{17051416}
\end{align}
in probability.

To see (\ref{17051415}), we use  (\ref{17051420}), which together with Lemma \ref{lem.17042902} implies
\begin{align*}
\sup_{z\in \Omega} |\mathbf{z}_{1} \mathcal{P}_2(z) \mathbf{z}_{1}'-m_1(d_+)|\leq \varepsilon'.
\end{align*}
Hence, (\ref{17051415}) and (\ref{17051416}) follow.  Further, similarly to Lemma \ref{lem. monotonicity}, we can show that $\mathbf{z}_{1} \mathcal{P}_2(z) \mathbf{z}_{1}'$ is decreasing in $(\lambda_{m,1}, \lambda_{m+1,1})$. Then, by the sign of $m_1(d_+)$ and the fact that $\lambda_{i,0}$ is the solution to the equation $\mathbf{z}_{1} \mathcal{P}_2(z) \mathbf{z}_{1}'=0$, we see that
\begin{align}
|\lambda_{m+1,0}-\lambda_{m,1}|\leq 2n^{-1+\varepsilon} \label{17051431}
\end{align}
in probability. This together with (\ref{17051410}) further implies that
\begin{align}
|\lambda_{m+1,0}-\widetilde{\lambda}_{m-1}|\leq 3n^{-1+\varepsilon}. \label{17051430}
\end{align}
Then we get the conclusion for case (i)', by slightly modifying $\varepsilon$.  For case (ii)', it suffices to go through the above discussion, with $(\widetilde{\lambda}_{m-1}, \widetilde{\lambda}_m)$ replaced by $(\widetilde{\lambda}_{m}, \widetilde{\lambda}_{m+1})$ in  (\ref{17051415}) and (\ref{17051430}).  For case (iii)', we first need replace $(\widetilde{\lambda}_{m-1}, \widetilde{\lambda}_m)$ by $(\widetilde{\lambda}_{m}, \widetilde{\lambda}_{m+1})$ in  (\ref{17051415}) and (\ref{17051430}). Then, instead of (\ref{17051431}), we  have
\begin{align*}
|\lambda_{m+1,0}-\lambda_{m+1,1}|\leq 2n^{-1+\varepsilon}
\end{align*}
since the sign of $m_1(d_+)$ changed. This together with (\ref{17051436}) further implies that
\begin{align}
|\lambda_{m+1,0}-\widetilde{\lambda}_{m+1}|\leq 3n^{-1+\varepsilon}. \label{17051430}
\end{align}

The above discussion on the case $k=2$ simply gives the result for $\lambda_{i, k-1}$  for general $k$. Then, we apply the above argument again to compare $\lambda_{i,k-3}$'s with $\lambda_{i, k-2}$'s. This is equivalent to consider the eigenvalues
of $C_{xy}$ with $k=3$. Using the argument recursively, we can finally get the result for $\lambda_i$'s for general $k$. This concludes the proof.

\end{proof}

\section{Proof of Lemma \ref{lem.062802}}\lb{sec5}
To ease the presentation, we introduce the following notations
 \begin{align}
 &\Phi\equiv\Phi(z):=D_{22}^{-1}(z)=( W_2 (P_{y}-z) W_2')^{-1},\nonumber\\
 &\Upsilon\equiv\Upsilon(z):=(  Y  (P_{w_2}-z)  Y')^{-1}, \label{16121220}
 \end{align}
Observe that $\text{rank}( P_{w_2})=p-k$ and $\text{rank}( P_y )=q$ almost surely.
For any $m\times \ell$ rectangle matrix $A$, using SVD, it is easy to check  the identity
\begin{align}
A(A'A-zI_\ell)^{-1}A'=AA'(AA'-z I_m)^{-1}=I_m+z (AA'-z I_m)^{-1}.  \label{17052502}
\end{align}
Moreover, we can write
 \begin{align}
 \Phi(z)= S _{w_2w_2}^{-\frac{1}{2}}(\widetilde{C}_{w_2y}-z)^{-1}  S _{w_2w_2}^{-\frac{1}{2}},\label{17050130}
 \end{align}
 where we introduced the symmetric matrix
  \begin{align}
\widetilde{C}_{w_2y} := S _{w_2w_2}^{-\frac{1}{2}} S _{w_2y} S _{yy}^{-1} S _{yw_2} S _{w_2w_2}^{-\frac{1}{2}}. \label{17052501}
 \end{align}
Using (\ref{17050130}) and (\ref{17052502}), it is not difficult to check
\begin{align}
	 S _{yy}^{-1/2} S _{yw_2}\Phi  S _{w_2y} S _{yy}^{-1/2}
	=  I_q+z S _{yy}^{1/2}\Upsilon  S _{yy}^{1/2}, \label{17031935}
\end{align}
from which we immediately get
\begin{eqnarray}\label{070201}
 S _{yw_2}\Phi  S _{w_2y}= S _{yy}+z  S _{yy}\Upsilon  S _{yy}.
\end{eqnarray}
Further, we introduce the following  notations for Wishart matrices
\begin{align}
	&E:=  W_2 P_y  W_2', \quad  H:=  W_2( I- P_y ) W_2',\nonumber\\
	&\mathcal{E}:= Y P_{w_2} Y',\quad \mathcal{H}:= Y( I- P_{w_2}) Y'. \label{16121501}
\end{align}
By Cochran's Theorem, we know that $E$ and $ H$ are independent, so are $\mathcal{E}$ and $\mathcal{H}$. In addition, in light of (\ref{17040101}), we have
\begin{align}
	&nE\sim \mathcal{W}_{p-k}( I_{p-k}, q),\quad n H\sim\mathcal{W}_{p-k}( I_{p-k}, n-q),\nonumber\\
	& n\mathcal{E}\sim \mathcal{W}_q(  I_q, p-k),\quad n\mathcal{H}\sim \mathcal{W}_q( I_q, n-p+k). \label{17031930}
\end{align}
Observe that
\begin{align}
 & S _{w_2w_2}=E+ H, \qquad  S _{yy}=\mathcal{E}+\mathcal{H},\nonumber\\
& C_{w_2y}=\big( E+ H\big)^{-1} E, \qquad    C_{y w_2}=\big(\mathcal{E}+\mathcal{H}\big)^{-1}\mathcal{E},  \label{pres1}
\end{align}
where $ C_{w_2y}$ and $ C_{y w_2}$ are defined analogously to (\ref{17031001}).
From (\ref{pres1}), we can also write
\begin{align}
\Phi(z)=((1-z)E-z  H)^{-1}\, \qquad \Upsilon(z)=\big((1-z)\mathcal{E}-z \mathcal{H}\big)^{-1}.\label{061911}
\end{align}
Therefore, we have
\begin{eqnarray}
 S _{yy}\Upsilon  S _{yy}&=&(\mathcal{E}+\mathcal{H})\big((1-z)\mathcal{E}-z \mathcal{H}\big)^{-1}(\mathcal{E}+\mathcal{H})\nonumber\\ \nonumber\\
&=&(1-z)^{-1}\mathcal{E}-z^{-1}\mathcal{H}+(z(1-z))^{-1}\big((1-z)\mathcal{H}^{-1}-z \mathcal{E}^{-1}\big)^{-1} \label{061912}
\end{eqnarray}
Following from (\ref{070201}) and (\ref{061912}) we obtain
\begin{eqnarray} \label{pres2}
	(1-z) S _{yw_2}\Phi  S _{w_2y}=\mathcal{E}+\big((1-z)\mathcal{H}^{-1}-z \mathcal{E}^{-1}\big)^{-1}.
\end{eqnarray}
In the sequel, we use the notation
\begin{align}
	\Psi\equiv \Psi_n(z):=\big((1-z)\mathcal{H}^{-1}-z \mathcal{E}^{-1}\big)^{-1}. \label{16122950}
\end{align}
Then, we can write
\begin{align}
(1-z) S _{yw_2}{\Phi} S _{w_2y}=\mathcal{E}+\Psi. \label{17031801}
\end{align}

  Recall the SVD in (\ref{17032720}). Note that $ W_1, W_2, U_{y}, \Lambda_{y}, V_{y}$ are mutually independent.  According to the definitions of $\mathscr{A}, \mathscr{B}, \mathscr{D}$ in (\ref{17050301}) and (\ref{17032301}) the, definition of $\Phi$ in (\ref{16121220}), and (\ref{17040601}),  we see that $\big( W_1, U_{y}\big)$ is independent of $\mathscr{A}, \mathscr{B}, \mathscr{D}$.  Hence, we can condition on $ W_2, \Lambda_{y}, V_{y}$  and thus condition on $\mathscr{A}, \mathscr{B}, \mathscr{D}$ and use the randomness of $ W_1, U_{y}$ in the  sequel.

 Our first step is to use a $k\times q$ Gaussian matrix to approximate $ T_1 U_{y}$.  For simplicity, we denote
 \begin{align*}
 \mathcal{T}:=\text{diag}(\sqrt{r_1}, \ldots, \sqrt{r_k}).
 \end{align*}
 Apparently, $ \mathcal{T}^2=T_1T_1'$.
 More specifically, we claim the following lemma.
 \begin{lem}\label{lem.17032401} With the above notations, there exists a $k\times q$  Gaussian matrix $ G_1$  with i.i.d. $N(0, \frac{1}{q})$ entries, such that for any sufficiently large  $K>0$,
 \begin{align}
 &\mathscr{M}_2= \mathcal{T} G_1\mathscr{B} W_1'+O(\frac{(\log n)^K}{n}),\nonumber\\
 &  \mathscr{M}_3= \mathcal{T} G_1 \mathscr{D} G_1' \mathcal{T}- \frac{1}{q} \ntr \mathscr{D}\cdot  \mathcal{T}\big( G_1  G_1'-I\big) \mathcal{T}+O(\frac{(\log n)^K}{n}), \label{17052595}
 \end{align}
 hold uniformly in $z\in\mathcal{D}$, with overwhelming probability. Here  $O(\frac{(\log n)^K}{n})$ represents a generic $k\times k$ random matrix with all entries of order $O(\frac{(\log n)^K}{n})$.
 \end{lem}

 The proof of Lemma \ref{lem.17032401} will be postponed to Appendix.
In the sequel, we  condition on $\mathscr{A},\mathscr{B}, \mathscr{D}$, and use the randomness of $ W_1$ and $ G_1$ only. First, we introduce the  following well-known  large deviation inequalities for Gaussian  which  will be  frequently used in the  sequel.
\begin{lem} \label{let} Assume that $\mathbf{x}, \mathbf{y}\in \mathbb{R}^n$  are two i.i.d. $N(0, n^{-1} I_n)$ vectors. Let $A\in \mathbb{C}^{n\times n}$ and $\mathbf{b}\in \mathbb{R}^n$ be independent of $\mathbf{x}$ and $\mathbf{y}$.  The following three statements hold with overwhelming probability for any sufficiently large constant  $K>0$:
\begin{align*}
&(i): \big|\mathbf{x}' A\mathbf{x} -\frac{1}{n}\ntr  A\big|\leq \frac{(\log n)^K}{n}\|A\|_{HS}, \quad (ii): |\mathbf{x}' A\mathbf{y}|\leq \frac{(\log n)^K}{n}\|A\|_{HS},\nonumber\\
 &(iii): |\mathbf{x}'\mathbf{b}|\leq \frac{(\log n)^K}{\sqrt{n}}\|\mathbf{b}\|_2.
\end{align*}
\end{lem}

With the aid of  Lemma \ref{let}, we can derive the following.
\begin{lem}\lb{lem5.4} For any given $\delta>0$ in the definition of $\mathcal{D}$ (c.f. (\ref{16121240})), and any sufficiently large constant  $K>0$, the following inequalities hold uniformly on $\mathcal{D}$, with overwhelming probability,
\begin{align}
&\sup_{i,j}\Big|(\scM_1)_{ij}-\delta_{ij}  (1-r_i)\Big(c_2-(1+c_1)z-(1-z)\frac{1}{n}\ntr\big(E\Phi\big)\Big)\Big|\leq \frac{(\log n)^K}{\sqrt{n}},\nonumber\\
& \sup_{i,j}\Big|(\scM_2)_{ij}\Big|\leq \frac{(\log n)^K}{\sqrt{n}}, \qquad \sup_{i,j}\Big|(\scM_3)_{ij}-\delta_{ij}  r_i(1-z)\frac{1}{q}\ntr(\mathcal{H}-\Psi)\Big|\leq \frac{(\log n)^K}{\sqrt{n}}. \label{17032501}
\end{align}
\end{lem}
For simplicity, we denote by $\mathbf{w}_i=(w_{i1},\ldots, w_{in})$ and $\mathbf{b}_i:=(b_{i1},\ldots, b_{iq})$ the $i$-th rows of $ W_1$ and $ G_1$, respectively.

\begin{proof}[Proof of Lemma \ref{lem5.4}]  In light of the definitions in (\ref{16121501}) and (\ref{17050301}), and also (\ref{061911}),  it is elementary to check that
\begin{align}
&\frac{1}{n}\ntr\mathscr{A}=\frac{q}{n}-z -\frac{1}{n}\ntr\big(\Phi((1-z)^2E+z^2 H)\big)\nonumber\\
&=\frac{q}{n}-z-z\frac{p}{n} - (1-z)\frac{1}{n}\ntr\big(E\Phi\big). \label{17032620}
\end{align}
In addition, using the fact that $\|W_2\|\leq C$ with overwhelming probability (c.f. (\ref{0726100})), it is not difficult to see that  for all $z\in \mathcal{D}$,
\begin{align}
\|\mathscr{A}\|_{\text{HS}}^2\leq Cn\big(+\| \Phi\|^2\big). \label{17040357}
\end{align}
with overwhelming probability. In addition, we have
\begin{align}
\|\Phi\|\leq \|( C_{w_2y}-z)^{-1}\| \|S_{w_2w_2}^{-1}\|=O(1) \label{17040401}
\end{align}
with overwhelming probability, where we used the fact that the event $ \Xi_1$ holds with overwhelming probability, and also (\ref{0726100}).
Hence, $\|\mathscr{A}\|_{\text{HS}}\leq C\sqrt{n}$  holds with overwhelming probability on $\mathcal{D}$.

Then, conditioning on the matrix $\mathscr{A}$ and use the randomness of $\mathbf{w}_i$'s only, the first statement in (\ref{17032501}) follows from Lemma \ref{let} (i), (ii) immediately.

For the second estimate in (\ref{17032501}), we condition on   $\mathscr{B}$, and use the randomness of $\mathbf{w}_i$'s and $\mathbf{b}_i$'s. Then the conclusion follows from  Lemma \ref{let} (ii) and the fact that
\begin{align}
\|\mathscr{B}\|_{\text{HS}}\leq C\sqrt{n}  \label{17040375}
\end{align}
holds uniformly on $\mathcal{D}$, with overwhelming probability.  To see (\ref{17040375}), we recall the definition of $\mathscr{B}$ from
(\ref{17032301}).  Similarly to (\ref{17040357}),
\begin{align*}
\|\mathscr{B}\|_{\text{HS}}^2 \leq Cn\big(1+\| \Phi\|^2\big).
\end{align*}
 Using (\ref{17040401}) again, we see that (\ref{17040375}) holds.

For the third estimate, we use the randomness of $\mathbf{b}_i$'s and condition on the $\mathscr{D}$. Then the conclusion again follows from Lemma \ref{let} (i), (ii),  the fact
\begin{align}
\frac{1}{q} \ntr \mathscr{D}=  \frac{1-z}{q}\ntr \Big( S _{yy}-(1-z) S _{y w_2}\Phi S _{w_2y}\Big)= \frac{1-z}{q}\ntr(\mathcal{H}-\Psi)
\label{17032621}
\end{align}
and the fact that
\begin{align}
\|\mathscr{D}\|_{\text{HS}}\leq C\sqrt{n} \label{17040380}
\end{align}
holds with overwhelming probability.  The proof of  (\ref{17040380}) is similar to that of (\ref{17040375}). We thus omit the details.

This completes the proof of Lemma \ref{lem5.4}.
\end{proof}

Recall $s(z)$ in (\ref{17040110}). For brevity, we further  introduce the following function
\begin{align}
\varrho(z):=\frac{1-z}{c_2}s(z)-c_1. \label{17040390}
\end{align}
Hence, what remains is to estimate the normalized traces of $ E\Phi$, $\mathcal{H}$ and $\Psi$.  We collect the results in the following lemma, whose proof will be postponed to Appendix A.
\begin{lem} \label{lem.17032510} For any given $z\in \mathcal{D}$, we have  the following estimates:
\begin{align}
&\frac{1}{q} \ntr \mathcal{H}=  \frac{n-p}{n}+o_{\mathrm{p}}(\frac{1}{\sqrt{n}}),\qquad \frac{1}{n} \ntr  E\Phi(z)=s(z)+o_{\mathrm{p}}(\frac{1}{\sqrt{n}}),\nonumber\\
&\frac{1}{q} \ntr \Psi(z)= \varrho(z)+o_{\mathrm{p}}(\frac{1}{\sqrt{n}}). \label{17040610}
\end{align}
Further, for any sufficiently small constant $\varepsilon>0$,  we have the uniform estimate
\begin{align*}
  \sup_{z\in \mathcal{D}} \Big|\frac{1}{n} \ntr E\Phi(z)-s(z)\Big|=O_{\text{a.s.}}(n^{-\varepsilon}),\qquad \sup_{z\in \mathcal{D}}\Big|\frac{1}{q} \ntr \Psi(z)-\varrho(z)\Big|=O_{\text{a.s.}}(n^{-\varepsilon}).
\end{align*}
\end{lem}

With the aid of Lemmas \ref{lem5.4} and \ref{lem.17032510}, we can now prove Lemma \ref{lem.062802}.

\begin{proof}[Proof of Lemma \ref{lem.062802}] According to the decomposition in (\ref{17031505}), we see from Lemmas \ref{lem5.4} and \ref{lem.17032510} that
\begin{align*}
( M_n)_{ij}(z)&= (\scM_1)_{ij}(z) + (\scM_2)_{ij}(z)+(\scM_2)_{ji}(z)+(\scM_3)_{ij}(z)\nonumber\\
&=\delta_{ij}\Big(\big(c_2-(1+c_1)z-(1-z) s(z)\big) (1-r_i)+(1-z)(1-c_1-\varrho(z))r_i\Big)+O_{\text{a.s.}}(n^{-\varepsilon})\nonumber\\
&=\delta_{ij} m_i(z)+O_{\text{a.s.}}(n^{-\varepsilon})
\end{align*}
uniformly in $z\in \mathcal{D}$ and $i,j$.  Here in the last step we used the definitions in (\ref{17033101}) and (\ref{17040390}). This completes the proof of Lemma \ref{lem.062802}.

\end{proof}

\section{Proofs of Lemma \ref{lem4.3}-\ref{lem4.5}}\lb{sec6}
In this section we prove Lemmas \ref{lem4.3}-\ref{lem4.5}.
First, from (\ref{17031301}),  (\ref{17050801}) and (\ref{17050702}), it is easy to check
\begin{align*}
\sqrt{(\ga_j-d_-)(\ga_j-d_+)}=\frac{\vartheta r_j^2-\varpi}{r_j}.
\end{align*}
This further implies
\begin{align}
s(\ga_j)=-\frac{\varpi}{ \vartheta r_j-\varpi},\qquad s'(\ga_j)= \frac{\varpi\vartheta r_j^2}{\big(\vartheta r_j-\varpi\big)^2\big(\vartheta r_j^2-\varpi\big)}.
\label{s in terms of r}
\end{align}
Moreover, we need the following lemma  whose proof is postponed to Appendix A.
\begin{lem}\lb{lemf_1}
Suppose that the assumptions in Theorem \ref{thm.fluctuation} hold. For any $j\in \{1, \ldots, k_0\}$ such that $r_j$ is well separated from $r_c$ and $1$,   we have the following estimates
\begin{align*}
&\frac{1}{n}\ntr  H\Phi(\ga_j)=   \frac{-c_1+ (1-\ga_j)s(\ga_j)}{\ga_j}+o_{\text{p}}(1),\nonumber\\
&\frac{1}{n}\ntr  (H\Phi(\ga_j) H\Phi(\ga_j)) =  \frac{\ga_j(1-\ga_j)^2s'(\ga_j)+(\ga_j^2-1)s(\ga_j)+c_1}{\ga_j^2}+o_{\text{p}}(1),\nonumber\\
&\frac{1}{n}\ntr (E\Phi (\ga_j) H\Phi(\ga_j))=  (1-\ga_j)s'(\ga_j)-s(\ga_j)+o_{\text{p}}(1),\nonumber\\
&\frac{1}{n}\ntr (E\Phi(\ga_j) E\Phi(\ga_j)) = \ga_js'(\ga_j)+s(\ga_j)+o_{\text{p}}(1).
\end{align*}
\end{lem}

 Then, recall the definitions of $\mathcal{E}$ and $\mathcal{H}$ in (\ref{16121501}), $\Psi$ in (\ref{16122950}). We have the following lemma  whose proof is also postponed to Appendix A.

\begin{lem}\lb{lemf_2}
	Suppose that the assumptions in Theorem \ref{thm.fluctuation} hold. For any $j\in \{1, \ldots, k_0\}$ such that $r_j$ is well separated from $r_c$ and $1$,   we have the following results:
	\begin{align*}
	&\sup_{\alpha,\beta}\Big|\Psi_{\alpha\beta}(\ga_j) -\delta_{\alpha\beta}\varrho(\ga_j)\Big|=o_{\text{p}}(1),
\nonumber\\
	&\sup_{\alpha,\beta}\Big|\big(\Psi(\ga_j)(\mathcal{H}^{-1}+\mathcal{E}^{-1})\Psi(\ga_j)\big)_{\alpha\beta}-\delta_{\alpha\beta} \varrho'(\ga_j)\Big|=o_{\text{p}}(1).
	\end{align*}
\end{lem}

\subsection{Proof of Lemma \ref{lem4.3}}
Recall the decomposition of $ M_n$ in (\ref{17031505}). In the sequel, we consider $\mathscr{M}_1, \mathscr{M}_2$ and $\mathscr{M}_3$ separately.

Recall   $\mathscr{M}_1$ in  (\ref{16123001}), $\mathscr{A}(z)$ in (\ref{17050301}) and  $\Phi(z)$ in (\ref{061911}). Using the identity
\begin{align}
\Phi(z_1)-\Phi(z_2)=(z_1-z_2) \Phi(z_1)W_2W_2'\Phi(z_2), \label{resolvent expansion 1}
\end{align}
it is not difficult to write
\begin{align}
\sqrt{n}(\mathscr{M}_1(\la_j)-\mathscr{M}_1(\ga_j))
=\sqrt{n}(\la_j-\ga_j)\sum_{a=0}^4 W_1\mathcal{A}_a W_1'. \label{16123002}
\end{align}
where
\begin{align}
&\mathcal{A}_0:=-I_n,\qquad \mathcal{A}_1:=- P_{y} W_2'\Phi(\ga_j)W_2W_2'\Phi(\la_j) W_2 P_{y},\nonumber\\
&\mathcal{A}_2=\mathcal{A}_3':= P_{y} W_2' \Phi(\la_j)  W_2+\ga_j P_{y} W_2'\Phi(\ga_j)W_2W_2'\Phi(\la_j) W_2,\nonumber\\
&\mathcal{A}_4:=-(\la_j+\ga_j) W_2' \Phi(\la_j)  W_2-\ga_j^2 W_2'\Phi(\ga_j)W_2W_2'\Phi(\la_j) W_2. \label{12163011}
\end{align}

Since $ W_1 W_1'\to  (I_k-T_1T_1')$ almost surely, we have
\begin{align*}
 W_1\mathcal{A}_0 W_1' =-  (I_k-T_1T_1')+o_{\text{p}}(1).
\end{align*}
The equation (\ref{resolvent expansion 1}) together with Theorem \ref{thm.061901}, (\ref{17040401}) and the fact that the event $ \Xi_1$ in (\ref{17040361}) holds with overwhelming probability implies that
\begin{align}
\|\Phi(\lambda_j)\|=O_{p}(1), \qquad \|\Phi(\lambda_j)-\Phi(\ga_j)\|=o_{p}(1). \label{17040402}
\end{align}

By  (\ref{17040402}), Lemma \ref{let} and Lemma \ref{lemf_1}, we have
\begin{align*}
 W_1\mathcal{A}_1 W_1'=&-\frac{1}{n}\ntr \big( P_{y} W_2'\Phi(\ga_j)W_2W_2'\Phi(\ga_j) W_2 P_{y}\big)\cdot  (I_k-T_1T_1')+o_{\text{p}}(1)\\
=&-s'(\ga_j)  (I_k-T_1T_1')+o_{\text{p}}(1).
\end{align*}
Here we used the definition of $E= W_2 P_{y} W_2'$, and the cyclicity of trace.

Analogously,  we obtain for $a=2,3$
\begin{align*}
 W_1\mathcal{A}_a W_1' =\big(s(\ga_j)+\ga_js'(\ga_j)\big)  (I_k-T_1T_1')+o_{\text{p}}(1).
\end{align*}
and
\begin{align*}
 W_1\mathcal{A}_4 W_1'
&=\big(c_1-s(\ga_j)-\ga_js'(\ga_j)\big)  (I_k-T_1T_1')+o_{\text{p}}(1).
\end{align*}
Combining the above estimate yields
\begin{align}\lb{le4.3pro1}
&\sqrt{n}(\mathscr{M}_1(\la_j)-\mathscr{M}_1(\ga_j))\nonumber\\
&=\sqrt{n}(\la_j-\ga_j)\Big(\big(c_1-1+s(\ga_j)+(\ga_j-1)s'(\ga_j)\big)  (I_k-T_1T_1')+o_{\text{p}}(1)
\Big).
\end{align}
Analogously, for  $\mathscr{M}_2$, using Lemma \ref{let},  we  will get
 \begin{align}\lb{le4.3pro2}
\sqrt{n}(\mathscr{M}_2(\la_j)-\mathscr{M}_2(\ga_j))=\sqrt{n}(\la_j-\ga_j)o_{\text{p}}(1),
\end{align}

At the end, we consider $\mathscr{M}_3$ (c.f. (\ref{16123001})).  By (\ref{17031801}), we can write
\begin{align*}
 \mathscr{M}_3&\equiv \mathscr{M}_3(z):=(1-z) T_1\big(\mathcal{H}-\Psi\big) T_1'.
\end{align*}
Notice that by definition in  (\ref{16122950}), we have
\begin{align}
\Psi=\mathcal{H} \Upsilon\mathcal{E}. \label{17040410}
\end{align}
And we also have  the  expansion
\begin{align}
\Psi(z_1)-\Psi(z_2)= (z_1-z_2)\Psi(z_2)\Big(\mathcal{H}^{-1}+ \mathcal{E}^{-1}\Big)\Psi(z_1)=(z_1-z_2)\mathcal{H}\Upsilon(z_2)YY'\Upsilon(z_1)\mathcal{E}. \label{17031805}
\end{align}
Furthermore, analogously to (\ref{17040401}), we have
\begin{align}
\|\Upsilon\| \leq  \|( C_{y w_2}-z)^{-1}\| \|S_{yy}^{-1}\|=O(1) \label{17040411}
\end{align}
with overwhelming probability, for $z\in \mathcal{D}$. Hence, from (\ref{17040410})-(\ref{17040411}), we also have the analogues of (\ref{17040401}) and (\ref{17040402})
\begin{align*}
\|\Psi(\ga_j)\|=O_{\text{p}}(1), \qquad  \|\Psi(\la_j)\|=O_{\text{p}}(1), \qquad \|\Psi(\la_j)-\Psi(\ga_j)\|=o_{\text{p}}(1).
\end{align*}
  Then approximating $\Phi(\lambda_j)$ by $\Phi(\gamma_j)$, and using Lemma \ref{lemf_2}, we obtain
 \begin{align}
 &\sqrt{n}(\mathscr{M}_3(\la_j)-\mathscr{M}_3(\ga_j)) \nonumber\\
 &=-\sqrt{n}(\la_j-\ga_j) T_1(\mathcal{H}-\Psi(\la_j)+(1-\ga_j)\Psi(
\la_j)(\mathcal{H}^{-1}+\mathcal{E}^{-1})\Psi(\ga_i)) T_1'\nonumber\\
&=-\sqrt{n}(\la_j-\ga_j)\Big(\Big(1-c_1-\varrho(\ga_j)+(1-\ga_j)\varrho'(
\ga_j)\Big) T_1 T_1'+o_{\text{p}}(1)\Big)\nonumber\\
&= -\sqrt{n}(\la_j-\ga_j)\Big(\Big( 1-2\frac{(1-\ga_j)s(\ga_j)}{c_2}+\frac{(1-\ga_j)^2s'(\ga_j)}{c_2}\Big) T_1 T_1'+o_{\text{p}}(1)\Big), \lb{le4.3pro3}
 \end{align}
 where in the second step we used the fact $\mathcal{H}_{ij}=\delta_{ij}(1-c_1)+o_p(1)$, and  in the last step we used the definition (\ref{17040390}).

 Therefore, combining (\ref{s in terms of r}) with the equations \eqref{le4.3pro1}, (\ref{le4.3pro2}) and  \eqref{le4.3pro3}, we can conclude the proof of Lemma \ref{lem4.3} by elementary calculation.

 \subsection{Proof of Lemma \ref{lem4.4}}  In this subsection, we prove Lemma \ref{lem4.4}. Recall the matrices defined in (\ref{17032301}). Let
\begin{align*}
\mathring{\mathscr{D}}:=\mathscr{D}-\frac{1}{q}\ntr \mathscr{D}\cdot I.
\end{align*}
In according to (\ref{17031505}), (\ref{17040420}) and Lemma \ref{lem.17032401}, we can write
 \begin{align}
   M_n- \frac{1}{q} \ntr \mathscr{D}\cdot   \mathcal{T}^2=\left( \begin{array}{cc}
   W_1 &  \mathcal{T} G_1
  \end{array}
  \right) \left( \begin{array}{cc}
\mathscr{A} & \mathscr{B}'\\
\mathscr{B} &  \mathring{\mathscr{D}}
  \end{array}
  \right) \left( \begin{array}{cc}
   W_1' \\
  \big(  \mathcal{T} G_1\big)'
  \end{array}
  \right)+o_{\text{p}}(\frac{1}{\sqrt{n}}). \label{17032410}
 \end{align}
In the sequel, we condition on $\mathscr{A},\mathscr{B}, \mathscr{D}$, and use the randomness of $ W_1$ and $ G_1$ only.
Recall the notation $\mathbf{w}_i=(w_{i1},\ldots, w_{in})$ as  the $i$-th rows of $ W_1$ and  further denote by $\mathbf{d}_i:=(d_{i1},\ldots, d_{iq})$ the $i$-th row of  $  \mathcal{T} G_1$.

According to (\ref{17032410}), we see that
\begin{align}
&\sqrt{n}\Big(\big( M_n\big)_{ij}(\ga_j)-\delta_{ij}  \Big((1-r_i)\frac{1}{n}\ntr \mathscr{A}+r_i\frac{1}{q}\ntr\mathscr{D}\Big)\Big)\nonumber\\
&=\sqrt{n}\Big(\mathbf{w}_i\mathscr{A} \mathbf{w}_j'-\delta_{ij} (1-r_i)\frac{1}{n} \ntr \mathscr{A}\Big)+\sqrt{n}\Big(\mathbf{d}_i\mathscr{B} \mathbf{w}_j'+ \mathbf{w}_i\mathscr{B}' \mathbf{d}_j'\Big)\nonumber\\
&\quad+\sqrt{n}\mathbf{d}_i\mathring{\mathscr{D}} \mathbf{d}_j'+o_{\text{p}}(1)=: Q_{ij}+ o_{\text{p}}(1). \label{17032430}
\end{align}
Let $\mathbb{E}_0$ be the conditional expectation w.r.t.  $\mathscr{A},\mathscr{B}, \mathscr{D}$. Set
\begin{align}
f_{\mathbf{c}}(t):= \exp\Big(\mathrm{i} t \sum_{i\leq j}c_{ij} Q_{ij} \Big), \qquad  \phi_{\mathbf{c}}(t):=\mathbb{E}_0 f_{\mathbf{c}}(t) \label{17032431}
\end{align}
for any collection of fixed numbers $\mathbf{c}:=(c_{ij})_{i\leq j}$.  It suffices to show that  for some deterministic  $s\equiv s(\mathbf{c})$,
\begin{align}
\phi_{\mathbf{c}}'(t)=-s^2t \phi_{\mathbf{c}}(t)+o_{\text{p}}(1). \label{17032420}
\end{align}
 From (\ref{17032420}) one can solve $\phi_{\mathbf{c}}(t)=\exp(-\frac{s^2t^2}{2}+o_p(1))$, which further implies
\begin{align*}
\mathbb{E}f_{\mathbf{c}}(t)=\mathbb{E} \phi_{\mathbf{c}}(t)\to \exp(-\frac{s^2t^2}{2}).
\end{align*}
This is known as Tikhomirov's method for CLT \cite{T}.

To establish (\ref{17032420}), we start with the definitions in (\ref{17032430}) and (\ref{17032431}). Then we have
\begin{align}
\phi_{\mathbf{c}}'(t) &= \mathrm{i}\sqrt{n}\sum_{i\leq j}c_{ij} \sum_{\ell=1}^q \; \Big(\mathbb{E}_0  \Big(w_{i\ell} \big(\mathscr{A} \mathbf{w}_j'\big)_{\ell} f_{\mathbf{c}}(t)\Big) - \delta_{ij} (1-r_i) \frac{1}{n} \ntr \mathscr{A}\; \phi_{\mathbf{c}}(t) \Big)\nonumber\\
&\qquad +\mathrm{i}\sqrt{n}\sum_{i\leq j}c_{ij} \sum_{\ell=1}^q \; \mathbb{E}_0  \Big(\Big(w_{j\ell}\big(\mathbf{d}_i\mathscr{B}\big)_{\ell} + w_{i\ell} \big(\mathscr{B}' \mathbf{d}_j'\big)_{\ell}\Big) f_{\mathbf{c}}(t)\Big)\nonumber\\
&\qquad +\mathrm{i}\sqrt{n}\sum_{i\leq j}c_{ij} \sum_{\ell=1}^q \; \mathbb{E}_0  \Big( d_{i\ell}\big(\mathring{\mathscr{D}} \mathbf{d}_j'\big)_{\ell}  f_{\mathbf{c}}(t) \Big) . \label{17032443}
\end{align}
Using integration by parts formula for Gaussian
\begin{align}
\mathbb{E} g h(g)=\sigma^2 \mathbb{E} h'(g), \qquad g\sim N(0, \sigma^2),  \label{17052701}
\end{align}
 we see that
\begin{align}
&\mathbb{E}_0  \Big(w_{i\ell} \big(\mathscr{A} \mathbf{w}_j'\big)_{\ell} f_{\mathbf{c}}(t)\Big)= \frac{1}{n} \delta_{ij} (1-r_i)\mathscr{A}_{\ell\ell} \phi_{\mathbf{c}}(t)+ \frac{\mathrm{i}t}{n}(1-r_i)\sum_{\alpha\leq \beta} c_{\alpha\beta} \mathbb{E}_0 \Big( \frac{\partial Q_{\alpha\beta}}{\partial w_{i\ell}}\big(\mathscr{A} \mathbf{w}_j'\big)_{\ell} f_{\mathbf{c}}(t)\Big)\nonumber\\
&\mathbb{E}_0  \Big(\big(w_{j\ell}\big(\mathbf{d}_i\mathscr{B}\big)_{\ell} + w_{i\ell} \big(\mathscr{B}' \mathbf{d}_j'\big)_{\ell}\big) f_{\mathbf{c}}(t)\Big)\nonumber\\
&\qquad\qquad=\frac{\mathrm{i}t}{n}\sum_{\alpha \leq \beta} c_{\alpha\beta} \mathbb{E}_0 \Big( \big( (1-r_j)\frac{\partial Q_{\alpha\beta}}{\partial w_{j\ell}}\big(\mathbf{d}_i\mathscr{B}\big)_{\ell} + (1-r_i)\frac{\partial Q_{\alpha\beta}}{\partial w_{i\ell}} \big(\mathscr{B}' \mathbf{d}_j'\big)_{\ell}\big) f_{\mathbf{c}}(t)\Big). \nonumber\\
&\mathbb{E}_0  \Big( d_{i\ell}\big(\mathring{\mathscr{D}} \mathbf{d}_j'\big)_{\ell} f_{\mathbf{c}}(t) \Big)
=\frac{1}{q} \delta_{ij} r_i\mathring{\mathscr{D}}_{\ell\ell} \phi_{\mathbf{c}}(t)+ \frac{\mathrm{i}t}{q}\sum_{\alpha \leq \beta} c_{\alpha\beta} r_i \mathbb{E}_0
\Big(\frac{\partial Q_{\alpha\beta}}{\partial d_{i\ell}}\big(\mathring{\mathscr{D} }\mathbf{d}_j'\big)_{\ell}  f_{\mathbf{c}}(t)\Big). \label{17032441}
\end{align}

Plugging (\ref{17032441})  into (\ref{17032443}) yields
\begin{align}
\phi_{\mathbf{c}}'(t) &= -t\mathbb{E}_0 \Big( Z f_{\mathbf{c}}(t)\Big) , \label{17032462}
\end{align}
where
\begin{align}
Z &:= \frac{1}{\sqrt{n}}\sum_{i\leq j}c_{ij}  \sum_{\alpha\leq \beta} c_{\alpha\beta} \sum_{\ell}  \Big(  (1-r_i)\frac{\partial Q_{\alpha\beta}}{\partial w_{i\ell}}\big(\mathscr{A} \mathbf{w}_j'\big)_{\ell} \nonumber\\
&+ \big( (1-r_j)\frac{\partial Q_{\alpha\beta}}{\partial w_{j\ell}}\big(\mathbf{d}_i\mathscr{B}\big)_{\ell} + (1-r_i)\frac{\partial Q_{\alpha\beta}}{\partial w_{i\ell}} \big(\mathscr{B}' \mathbf{d}_j'\big)_{\ell}\big)+ \frac{n}{q}r_i\frac{\partial Q_{\alpha\beta}}{\partial d_{i\ell}}\big(\mathring{\mathscr{D}} \mathbf{d}_j'\big)_{\ell}\Big). \label{17032450}
\end{align}
Now, note that
\begin{align}
&\frac{\partial Q_{\alpha\beta}}{\partial w_{i\ell}}=\delta_{\alpha i}\sqrt{n}\Big(\big(\mathscr{A} \mathbf{w}_\beta'\big)_\ell+(\mathscr{B}' \mathbf{d}_\beta')_\ell \Big)+ \delta_{\beta i}\sqrt{n}\Big(\big(\mathbf{w}_\alpha\mathscr{A} \big)_\ell+(\mathbf{d}_\alpha\mathscr{B})_\ell \Big), \nonumber\\
& \frac{\partial Q_{\alpha \beta }}{\partial d_{i\ell}}= \delta_{\alpha i} \sqrt{n} \Big(\big(\mathscr{B}\mathbf{w}_\beta'\big)_{\ell}+\big(\mathring{\mathscr{D}}\mathbf{d}_\beta'\big)_{\ell} \Big)+ \delta_{\beta i} \sqrt{n} \Big(\big(\mathbf{w}_\alpha \mathscr{B}'\big)_{\ell}+\big(\mathbf{d}_\alpha \mathring{\mathscr{D}}\big)_{\ell} \Big). \label{17032442}
\end{align}
Substituting (\ref{17032442}) into (\ref{17032450}), and using Lemma \ref{let}, together with the bounds $\|\mathscr{A}\|, \|\mathscr{B}\|, \|\mathscr{D}\|\leq C$ with overwhelming probability on $\mathcal{D}$, it is not difficult to get
\begin{align}
Z= &  \sum_{i\leq j} c_{ij}^2 (1+\delta_{ij}) \Big( (1-r_i)(1-r_j)\frac{1}{n}\ntr \mathscr{A}^2 + (r_i(1-r_j)+r_j(1-r_i)) \frac{1}{q}\ntr\mathscr{B}\mathscr{B}' \nonumber\\
&+\frac{n}{q^2} r_ir_j\ntr  \mathring{\mathscr{D}}^2\Big)+o_{\text{p}}(1)
=: s_n^2+o_{\text{p}}(1). \label{17032461}
\end{align}
 Plugging (\ref{17032461}) into (\ref{17032462}) , we get
\begin{align}
\phi_{\mathbf{c}}'(t) &= -ts_n^2 \phi_{\mathbf{c}}(t)+o_{\text{p}}(1).  \label{17040701}
\end{align}
This implies that $Q_{ij}$'s (c.f. (\ref{17032430})) are asymptotically  independent Gaussians (up to symmetry), and the limiting variance of $Q_{ij}$ is given by the limit of
\begin{align*}
(1-r_i)(1-r_j)\frac{1}{n}\ntr \mathscr{A}^2 + (r_i(1-r_j)+r_j(1-r_i)) \frac{1}{q}\ntr\mathscr{B}\mathscr{B}' +\frac{n}{q^2} r_ir_j\ntr  \mathring{\mathscr{D}}^2.
\end{align*}

Hence, in the sequel, it suffices to estimate the normalized traces of $\mathscr{A}^2$, $\mathscr{B}\mathscr{B}' $, and $\mathring{\mathscr{D}}^2$. The result is formulated in the following lemma, whose proof will be also postponed to Appendix A.

\begin{lem} \label{lem. traces of squared matrices} Suppose that the assumptions in Theorem \ref{thm.fluctuation} hold. For any $j\in \{1, \ldots, k_0\}$ such that $r_j$ is well separated from $r_c$ and $1$,   we have
\begin{align}
&\frac{1}{n} \ntr \mathscr{A}^2=-(1-\ga_j)^2s(\ga_j)+\ga_j(1-\ga_j)^2 s'(\ga_j)+\ga_j^2(1-c_1)+(1-2\ga_j) c_2+o_{\text{p}}(1),\nonumber\\
&\frac{1}{q} \ntr \mathscr{B}\mathscr{B}'=(1-\ga_j)^2\Big(1 -\frac{1}{c_2}s(\ga_j)+\frac{\ga_j(1-\ga_j)}{c_2} s'(\ga_j)\Big)+o_{\text{p}}(1) ,\nonumber\\
&\frac{1}{q} \ntr \mathring{\mathscr{D}}^2=(1-\ga_j)^2\Big(-c_1+c_2+\Big(\frac{(\ga_j-1)^2}{c_2}+\frac{(2c_2-1)}{c_2}\ga_j+\frac{c_1-c_2}{c_2}\Big) s(\ga_j)\nonumber\\
&\qquad\qquad-\frac{\big(1+(c_2-1)\ga_j\big)(1-\ga_j)}{c_2^2} s^2(\ga_j)+\frac{\ga_j(1-\ga_j)^2}{c_2}s'(\ga_j)\Big)+o_{\text{p}}(1).
\label{17052702}
\end{align}
\end{lem}
With the above lemma and (\ref{17040701}), we conclude the proof of Lemma \ref{lem4.4}.

\subsection{Proof of Lemma \ref{lem4.5}}
It is equivalent to show the following lemma.
\begin{lem} \label{mean value approximation} Suppose that the assumptions in Theorem \ref{thm.fluctuation} hold. For any $j\in \{1, \ldots, k_0\}$ such that $r_j$ is well separated from $r_c$ and $1$,   we have
\begin{align}
&\frac{1}{n} \ntr \mathscr{A}(\ga_j)= c_2-(1+c_1) \ga_j - (1-\ga_j)s(\ga_j)+o_{\mathrm{p}}(\frac{1}{\sqrt{n}}),\nonumber\\
&\frac{1}{q} \ntr \mathscr{D}(\ga_j)= (1-\ga_j)(1-c_1-\varrho(\ga_j))+o_{\mathrm{p}}(\frac{1}{\sqrt{n}}). \label{17040260}
\end{align}
\end{lem}

\begin{proof} The results follows from  Lemma \ref{lem.17032510}, (\ref{17032620}) and (\ref{17032621})
immediately.
\end{proof}

With Lemma \ref{mean value approximation}, we can conclude the proof of Lemma \ref{lem4.5} easily from the definitions in (\ref{17033101}) and (\ref{17040390}).

\newpage

\appendix

\section{}\label{s.appendix a}
In this appendix, we prove Lemma \ref{lem.17032401}, Lemma \ref{lem.17032510}, Lemma \ref{lemf_1}, Lemma \ref{lemf_2}, and Lemma \ref{lem. traces of squared matrices}.

\begin{proof}[Proof of Lemma \ref{lem.17032401}]Let $\mathbf{u}_i$' be the $i$-th row of $U_{y}$. We only state the details for the proof of the case $k=2$. For general fixed $k$, the proof is just a straightforward extension. It is well known that a Haar orthogonal  matrix can be obtained from a Gaussian matrix with i.i.d. entries via Gram-Schmidt orthogonalization. Especially,  there exists a pair of i.i.d. Gaussian vectors $\mathbf{g}_1,\mathbf{g}_2\in \mathbb{R}^q$ with i.i.d. $N(0, \frac{1}{q})$ entries, such that
\begin{align}
(\mathbf{u}_1, \mathbf{u}_2)=\Big(\frac{\mathbf{g}_1}{\|\mathbf{g}_1\|}, \frac{\mathbf{g}_2-\langle \mathbf{g}_2,\mathbf{g}_1\rangle \frac{\mathbf{g}_1}{\|\mathbf{g}_1\|^2} }{\| \mathbf{g}_2-\langle \mathbf{g}_2,\mathbf{g}_1\rangle \frac{\mathbf{g}_1}{\|\mathbf{g}_1\|^2}\|} \Big). \label{17040270}
\end{align}
Using  Lemma \ref{let}, it is elementary to see that  $\|\mathbf{g}_i\|^2-1=O(\frac{(\log n)^K}{n})$ and
\begin{align}
\|\mathbf{g}_i\|=1-\frac{1}{2} \big(\|\mathbf{g}_i\|^2-1\big)+O(\frac{(\log n)^K}{n}), \qquad  \mathbf{g}_1 A\mathbf{g}_2'= O(\frac{(\log n)^K}{\sqrt{n}}) \label{17040271}
\end{align}
hold with overwhelming probability, for any $q\times q$ matrix $A$ independent of $\mathbf{g}_i$'s with $\|A\|_{\text{HS}}\leq C\sqrt{n}$.  Hence, for any $\mathbf{d}\in \mathbb{R}^q$ independent of $\mathbf{u}_i$'s, with $\|\mathbf{d}\|\leq C$, we have
\begin{align}
\mathbf{u}_i \mathbf{d}'=\Big(1-\frac{1}{2} \big(\|\mathbf{g}_i\|^2-1\big)+O(\frac{(\log n)^K}{n})\Big)\mathbf{g}_i\mathbf{d}'= \mathbf{g}_i\mathbf{d}'+O(\frac{(\log n)^K}{n}), \label{17051020}
\end{align}
with overwhelming probability.  This proves the first approximation in  Lemma \ref{lem.17032401}  if we regard the columns of  $\mathscr{B} W_1'$ as $\mathbf{d}'$ above.

For the second approximation in  Lemma \ref{lem.17032401} , we again use (\ref{17040270}) and (\ref{17040271}) to see that
\begin{align}
&\mathbf{u}_i \mathscr{D}\mathbf{u}_i'= \Big(1-\frac{1}{2} \big(\|\mathbf{g}_i\|^2-1\big)+O(\frac{(\log n)^K}{n})\Big)^2 \mathbf{g}_i\mathscr{D}\mathbf{g}_i',\nonumber\\
&\mathbf{u}_1 \mathscr{D}\mathbf{u}_2=(1+O(\frac{(\log n)^K}{\sqrt{n}}))\mathbf{g}_1 \mathscr{D}\mathbf{g}_2'- (1+O(\frac{(\log n)^K}{\sqrt{n}}))\mathbf{g}_1\mathbf{g}_2'  \mathbf{g}_1D\mathbf{g}_1'. \label{17050101}
\end{align}
Hence, with overwhelming probability, we have
\begin{align*}
&\mathbf{u}_i \mathscr{D}\mathbf{u}_i'=\mathbf{g}_i\mathscr{D}\mathbf{g}_i'-\big(\|\mathbf{g}_i\|^2-1\big) \frac{1}{q} \ntr \mathscr{D}+ O(\frac{(\log n)^K}{n}),\nonumber\\
&\mathbf{u}_1 \mathscr{D}\mathbf{u}_2= \mathbf{g}_1 \mathscr{D}\mathbf{g}_2'- \mathbf{g}_1\mathbf{g}_2'  \frac{1}{q}\ntr \mathscr{D}+O(\frac{(\log n)^K}{n}).
\end{align*}
This proves the second approximation in Lemma \ref{lem.17032401}  in case of $k=2$.

For more general fixed $k$, it suffices to write $(\mathbf{u}_1, \ldots, \mathbf{u}_k)$ in terms of  $(\mathbf{g}_1, \ldots, \mathbf{g}_k)$,  using Gram-Schmidt orthogonalization. The proof is analogous. We omit the details. Hence, we completed the proof of Lemma \ref{lem.17032401} .
\end{proof}

\begin{proof}[Proof of Lemma \ref{lem.17032510}]  Recall  (\ref{17031930}). It is elementary to check
the first estimate of (\ref{17040610}).
For the second estimate, we use the fact
\begin{align}
\Phi  E= \Big(( E+ H)^{-1}  E-z\Big)^{-1} ( E+ H)^{-1}  E=( C_{w_2y}-z)^{-1}  C_{w_2y},  \label{17040430}
\end{align}
 from which we have
\begin{align}
\frac{1}{n}\ntr  E \Phi=\frac{1}{n} \ntr  ( C_{w_2y}-z)^{-1}  C_{w_2y}=c_1+z\frac{1}{n} \ntr ( C_{w_2y}-z)^{-1}+O(\frac{1}{n}).  \label{17050156}
\end{align}
Recall $\check{s}(z)$ defined in (\ref{17040615}).  Now we claim that  for any small constant $\varepsilon>0$,
\begin{align}
\check{s}_n(z):=\frac{1}{p} \ntr ( C_{w_2y}-z)^{-1}=\check{s}(z)+O_p(n^{-1+\varepsilon}). \label{17040441}
\end{align}
for any  $z\in \mathcal{D}$.
This together with (\ref{17040615}) and (\ref{17050156}) implies that
\begin{align}\lb{lim1}
 \frac{1}{n}\ntr E\Phi= s(z)+O_p(n^{-1+\varepsilon})
\end{align}
for any  $z\in \mathcal{D}$.

Hence, what remains is to show (\ref{17040441}).   To this end, we first focus on a sub domain of $\mathcal{D}$:
\begin{align*}
\mathcal{D}_+ \equiv\mathcal{D}_+(\delta):=\Big\{z\in \mathbb{C}: d_{+}+\delta< \Re z\leq 2, \frac{1}{n}\leq |\Im z|\leq 1\Big\}
\end{align*}
Note that for $z\in \mathcal{D}_+$, the imaginary part of $z$ is away from 0 by at least $\frac{1}{n}$. Then, we have  the trivial bound
\begin{align}
\sup_{z\in \mathcal{D}_+}\big\|( C_{w_2y}-z)^{-1}\big\|\leq n,  \label{17040510}
\end{align}
deterministically.
Further, since the event $ \Xi_1$ (c.f. (\ref{17040361})) holds with overwhelming probability, it is easy to see that
\begin{align}
\sup_{z\in \mathcal{D}}\big\|( C_{w_2y}-z)^{-1}\big\|\leq C,  \label{17040511}
\end{align}
with overwhelming probability.

The deterministic bound and the high probability event (\ref{17040511}) together implies that
\begin{align}
\sup_{z\in \mathcal{D}_+}  \mathbb{E}\big\|( C_{w_2y}-z)^{-1}\big\|^\ell \leq C_\ell \label{17040620}
\end{align}
for some positive constant $C_\ell$, for any given $\ell\in \mathbb{N}$.  Further, we introduce a smooth cutoff function
\begin{align}
\mathcal{Q}\equiv \mathcal{Q} ( E,  H)=\chi \Big(\frac{1}{p}\ntr ( E+ H)^{-1}\Big)=\chi \Big(\frac{1}{p}\ntr S_{w_2w_2}^{-1}\Big), \label{17040640}
\end{align}
where $\chi(x)=1$ if $0\leq x \leq n^\varepsilon$, while $\chi(x)=0$ if $x\geq n^\varepsilon+1$,  and $|\chi'(x)|\leq C$ for all $x\in [n^\varepsilon,n^\varepsilon+1]$.
From (\ref{0726100}), we see that $\frac{1}{p}\ntr S_{w_2w_2}^{-1}=O(1)$ with overwhelming probability. Hence, $\mathcal{Q}=1$ with overwhelming probability. In turn, we see that $\mathcal{Q}=1$ implies deterministically  that
$\|S_{w_2w_2}^{-1}\|= O(n^{1+\varepsilon})$.

Now, to prove (\ref{17040441}) for a given $z\in \mathcal{D}_+$, it suffices to show the following two
\begin{align}
\text{Var} (\check{s}_n(z)\mathcal{Q})=O(\frac{1}{n^2}),\qquad \mathbb{E} (\check{s}_n(z)\mathcal{Q})= \check{s}(z)+o(\frac{1}{\sqrt{n}}) \label{17040631}
\end{align}
since $\mathcal{Q}=1$ with overwhelming probability.
We first show the  variance bound above.
We write
\begin{align*}
  E=\mathfrak{X}\mathfrak{X}^*, \qquad  H=\mathfrak{Y}\mathfrak{Y}^*, \qquad \mathfrak{X}=(x_{ij})_{p-k, q},\qquad \mathfrak{Y}=(y_{ij})_{p-k, n-q},
\end{align*}
and all $x_{ij}$'s and $y_{ij}$'s are i.i.d. $N(0,\frac{1}{n})$. We then regard $\check{s}_n(z)\mathcal{Q}=\frac{1}{p}\ntr  E\Phi\mathcal{Q}$ as a function of the Gaussian variables $x_{ij}$'s and $y_{ij}$'s. Using Poincar\'{e} inequality for Gaussian, we have
\begin{align}
\text{Var} (\check{s}_n(z))\leq \frac{1}{n} \mathbb{E}\Big(\sum_{i, j} \Big|\frac{\partial \big(\frac{1}{p}\ntr  E \Phi(z)\big)\mathcal{Q}}{\partial x_{ij}}\Big|^2+\sum_{\alpha, \beta} \Big|\frac{\partial \big(\frac{1}{p}\ntr  E \Phi(z)\big)\mathcal{Q}}{\partial y_{\alpha\beta}}\Big|^2\Big). \label{17040630}
\end{align}
Observe that
\begin{align*}
\frac{\partial \big((\frac{1}{p}\ntr  E \Phi) \mathcal{Q}\big)}{\partial x_{ij}} &=\frac{1}{p} \Big((\mathfrak{X}\Phi)_{ji}+(\Phi\mathfrak{X})_{ij}+z \big(\mathfrak{X}\Phi E\Phi\big)_{ji}+z\big( \Phi E \Phi \mathfrak{X}\big)_{ij}\Big)\mathcal{Q}\nonumber\\
-&  \frac{1}{p^2}\ntr  E \Phi \chi'\Big(\frac{1}{p}\ntr S_{w_2w_2}^{-1}\Big)\Big(\Big(\mathfrak{X}^*S_{w_2w_2}^{-2}\Big)_{ji}+\Big(S_{w_2w_2}^{-2}\mathfrak{X}\Big)_{ij}\Big)
\end{align*}
Then it is not difficult to see that
\begin{align}
\frac{1}{n} \mathbb{E}\Big(\sum_{i, j} \Big|\frac{\partial \frac{1}{p}\ntr  E \Phi(z)}{\partial x_{ij}}\Big|^2\Big)\leq \frac{C}{n^2}  \mathbb{E}\Big( \big( \sum_{\alpha=1}^3\| E\|^\alpha \|\Phi\|^{\alpha+1}\big)\mathcal{Q}^2\Big)\nonumber\\
+\frac{C}{n^2}  \mathbb{E}\Big( \Big(\|  E\|^2 \|S_{w_2w_2}^{-1}\|^4\Big)\Big((\frac{1}{p}\ntr  E \Phi )\chi'\big(\frac{1}{p}\ntr S_{w_2w_2}^{-1}\big)\Big)^2\Big) \label{17040623}
\end{align}
for some positive constant $C$.  Note that $\mathcal{Q}\neq 0$ only when $\frac{1}{p}\ntr S_{w_2w_2}^{-1}\leq n^\varepsilon+1$, and similarly $\chi'\Big(\frac{1}{p}\ntr S_{w_2w_2}^{-1}\Big)\neq 0$ only if $ \frac{1}{p}\ntr S_{w_2w_2}^{-1}\in [n^{\varepsilon}, n^\varepsilon+1]$. Hence, if the random variables in two expectations are nonzero, we have the bound
$
\| S_{w_2w_2}^{-1}\|=O(n^{1+\varepsilon}).
$
In addition, we also know from (\ref{0726100}) that
$
\|  S_{w_2w_2}^{-1}\|=O(1)
$
holds with overwhelming probability. These facts imply that
\begin{align}
\mathbb{E}  \|  S_{w_2w_2}^{-1}\|^a \mathcal{Q}^b=O(1), \label{17040621}
\end{align}
for any fixed $a, b>0$. The same bound holds if we replace $\mathcal{Q}$ by $\chi'\Big(\frac{1}{p}\ntr S_{w_2w_2}^{-1}\Big)$. Furthermore, it is well known that
\begin{align}
\mathbb{E} \| E\|^\ell=O(1) \label{17040622}
\end{align}
for any given $\ell>0$.  Then combining the first inequality in (\ref{17040401}),  with (\ref{17040620}), (\ref{17040621}) and (\ref{17040622}), we can obtain from (\ref{17040623}) that
\begin{align*}
\frac{1}{n} \mathbb{E}\Big(\sum_{i, j} \Big|\frac{\big(\partial \frac{1}{p}\ntr  E \Phi(z)\big)\mathcal{Q}}{\partial x_{ij}}\Big|^2\Big)\leq \frac{C}{n^2}  \mathbb{E}\Big( \big( \sum_{\alpha=1}^3\| E\|^\alpha \|\Phi\|^{\alpha+1}\big)\mathcal{Q}^2\Big)=O(\frac{1}{n^2})
\end{align*}
uniformly in $z\in\mathcal{D}_+$. For the second term in (\ref{17040630}), we can get the same bound analogously. This completes the proof of the variance estimate in (\ref{17040631}).

The estimate of mean in (\ref{17040631}) was obtained in \cite{BaiH15C} for all
$$z\in \big\{w: \Re w\in [d_--c, d_++c], \Im w\geq c\big\}$$
for any small positive constant $c$. The proof in \cite{BaiH15C} take the advantages of the lower bound $\Im z\geq c$, which guarantees an order $1$ upper bound of $\|( C_{w_2y}-z)^{-1}\|$. Here for $z\in \mathcal{D}_+$, we have (\ref{17040620}) instead. This high order moment bound is sufficient for us to extend the discussion in \cite{BaiH15C}  to the region $\mathcal{D}_+$ easily. Hence we omit the proof of (\ref{17040631}), and refer to  \cite{BaiH15C} for the interested readers.

Furthermore, using Chebyshev's inequality,  (\ref{17040631}) and the  fact that $\mathcal{Q}=1$ with overwhelming probability, we see that for any $z\in \mathcal{D}_+$
\begin{align*}
\mathbb{P} \big(| \check{s}_n(z)-\check{s}(z)|\geq n^{-\varepsilon} \big)=O(n^{^{-2+2\varepsilon}}).
\end{align*}
Now, we define the lattice $\mathcal{S}_+=\mathcal{D}\cap (n^{-\varepsilon}\mathbb{Z}\times  n^{-\varepsilon}\mathrm{i}\mathbb{Z})$, where $\mathbb{Z}=\{\pm 1, \pm 2,\ldots\}$.  Then apparently  $\sharp (\mathcal{S}_+)=O(n^{2\varepsilon})$. This implies that
\begin{align*}
\mathbb{P} \Big(\cup_{z\in \mathcal{S}_+}\big\{| \check{s}_n(z)-\check{s}(z)|\geq n^{-\varepsilon}\big\} \Big)=O(n^{^{-2+4\varepsilon}}).
\end{align*}

In addition, for any $z\in\mathcal{D}$, one can find a $z'\in\mathcal{S}_+$ such that $|z-z'|=O(n^{-\varepsilon})$. Then by continuity of $\check{s}_n(z)$ and $\check{s}(z)$ on $\mathcal{D}$, i.e., $|\check{s}'_n(z)|, |\check{s}'(z)|\leq \eta^{-2}$, we can easily get
\begin{align*}
\mathbb{P} \Big(\cup_{z\in \mathcal{D}}\big\{| \check{s}_n(z)-\check{s}(z)|\geq Cn^{-\varepsilon}\big\} \Big)=O(n^{-2+4\varepsilon}).
\end{align*}
By choosing $\varepsilon$ sufficiently small, we  can then show that the second estimate in (\ref{17040610}) holds uniformly in $\mathcal{D}$ if we replace $o_{\text{p}}(\frac{1}{\sqrt{n}})$ by $O_{\text{a.s.}}(n^{-\varepsilon})$.

Next, we show the third estimate in (\ref{17040610}).  Recall  (\ref{17031930}). We write
\begin{align}
\mathcal{H}=\mathcal{X}\mathcal{X}^*,\qquad  \mathcal{E}=\mathcal{Y}\mathcal{Y}^*,\qquad \mathcal{X}=(\tilde{x}_{ij})_{q, n-p+k},\qquad \mathcal{Y}=(\tilde{y}_{ij})_{q, p-k} \label{17032604}
\end{align}
and all $\tilde{x}_{ij}$'s and $\tilde{y}_{ij}$'s are i.i.d. $N(0,\frac{1}{n})$.  Analogously to the proof of the second estimate in (\ref{17040610}), we can  prove the result for any fixed $z\in \mathcal{D}_+$ at first, then extend the proof to whole $\mathcal{D}$ by continuity. For a fixed $z\in \mathcal{D}_+$, we again  split the proof into two parts
\begin{align}
\text{Var}\Big(\frac{1}{q} \ntr \Psi(z) \tilde{\mathcal{Q}}\Big)= O(\frac{1}{n^2}), \label{17032601}
\end{align}
and
\begin{align}
\mathbb{E}\frac{1}{q} \ntr \Psi(z)\tilde{\mathcal{Q}}=\varrho(z)+o(\frac{1}{\sqrt{n}}), \label{17032602}
\end{align}
where
\begin{align}
 \tilde{\mathcal{Q}} ( \mathcal{E},  \mathcal{H})=\chi \Big(\frac{1}{q}\ntr ( \mathcal{E}+ \mathcal{H})^{-1}\Big)=\chi \Big(\frac{1}{q}\ntr S_{yy}^{-1}\Big). \label{17050844}
\end{align}
Similarly, here we used the fact that $\widetilde{Q}=1$ with overwhelming probability.
The proof of (\ref{17032601}) can be obtained by Poincar\'{e} inequality again. We omit the details. In the sequel, we show how to
derive (\ref{17032602}) by integration by parts formula of Gaussian measures. To ease the presentation, we get rid of the factor
$\tilde{\mathcal{Q}}$ in the discussion, and do the reasoning as if $\|S_{yy}^{-1}\|=O(1)$ deterministically. The rigorous justification of the following derivation requires one to put $\tilde{\mathcal{Q}}$ back into each step.

Using the integration by parts formula for Gaussian variables (c.f. (\ref{17052701})), we arrive at
\begin{align}
&\frac{1}{q}\E \ntr\Psi(z) = \frac{1}{q}  \sum_{i=1}^q\sum_{j=1}^{n-p+k}\E  \tilde{x}_{ij} \big(\mathcal{X}^*\Upsilon\mathcal{E}\big)_{ji} = \frac{1}{qn} \sum_{i=1}^q\sum_{j=1}^{n-p+k}  \E \frac{\partial \big(\mathcal{X}^*\Upsilon\mathcal{E}\big)_{ji}}{\partial x_{ij}}\nonumber\\
&= \frac{1}{qn}\E \sum_{i=1}^q\sum_{j=1}^{n-p+k}  \Big( \big(\Upsilon\mathcal{E}\big)_{ii}+  z \big(\mathcal{X}^*\Upsilon \mathcal{X}\big)_{jj}\big(
\Upsilon\mathcal{E}\big)_{ii}+z\big(\mathcal{X}^*\Upsilon\big)_{ji}\big(\mathcal{X}^*\Upsilon\mathcal{E}\big)_{ji}\Big)\nonumber\\
&=\frac{n-p+k}{qn} \mathbb{E} \ntr  \big(\mathcal{E}\Upsilon\big)+\frac{z}{qn} \mathbb{E} \ntr \big(\mathcal{H} \Upsilon\big) \ntr  \big(\mathcal{E}\Upsilon \big) +\frac{z}{qn}\E \ntr \big(\mathcal{H}\Upsilon\mathcal{E}\Upsilon\big). \label{17040129}
\end{align}
 Note that for any $z\in \mathcal{D}_+$, one has
 \begin{align}
\frac{z}{qn}\E \ntr \big(\mathcal{H}\Upsilon\mathcal{E}\Upsilon\big)=O(\frac{1}{n}). \label{17050820}
 \end{align}
 Similarly to the variance bound in (\ref{17040631}), it is not difficult to show
 \begin{align}
\text{Var}  \big(\frac{1}{p} \ntr \big(\mathcal{H}\Upsilon\big)=O(\frac{1}{n^2}), \qquad  \text{Var}  \big(\frac{1}{p} \ntr \big(\mathcal{E}\Upsilon\big)=O(\frac{1}{n^2}). \label{17050821}
 \end{align}
 To be precise, in order to obtain (\ref{17050820}) and (\ref{17050821}), the tracial quantities  shall be multiplied by $\tilde{\mathcal{Q}}$. But as we mentioned before, we omit it from the presentation.
Combining (\ref{17040129})- (\ref{17050821}), we have
\begin{align*}
&\frac{1}{q}\E \ntr\Psi(z) =\frac{n-p+k}{qn} \mathbb{E} \ntr  \big(\mathcal{E}\Upsilon\big)+\frac{z}{qn} \mathbb{E} \ntr \big(\mathcal{H} \Upsilon\big) \mathbb{E}\ntr  \big(\mathcal{E}\Upsilon \big) +O(\frac{1}{n})\nonumber\\
&=  \frac{n-p+k}{qn} \mathbb{E} \ntr  \big(\mathcal{E}\Upsilon\big)+\frac{1}{qn} \Big((1-z)\E\ntr \big(\mathcal{E} \Upsilon\big) -q\Big) \mathbb{E}\ntr  \big(\mathcal{E}\Upsilon \big) +O(\frac{1}{n})\nonumber\\
&=  \Big(\frac{n-p+k}{qn}-\frac{1}{n}\Big) \mathbb{E} \ntr  \big(\mathcal{E}\Upsilon\big)+\frac{1-z}{qn} \big(\E\ntr \big(\mathcal{E} \Upsilon\big) \big)^2 +O(\frac{1}{n})
\end{align*}
where we used the fact
  \begin{align*}
{z}\mathbb{E} \ntr \big(\mathcal{H} \Upsilon \big)=(1-z)\E \ntr\big(\mathcal{E} \Upsilon \big)-q.
  \end{align*}
Analogously to (\ref{17040631}), we can show
\begin{align*}
\frac{1}{q} \mathbb{E}\ntr \big(\mathcal{E}\Upsilon\big)
= \frac{z}{q} \mathbb{E}\ntr  \big( C_{y w_2}-z I_{q}\big)^{-1} +1=z\tilde{s}(z)+1+o(\frac{1}{\sqrt{n}}),
\end{align*}
where ${\tilde{s}}(z)$ is defined in (\ref{17040351}). Then by (\ref{17040615}), we have
\begin{align}
\frac{1}{q} \mathbb{E}\ntr \big(\mathcal{E}\Upsilon\big)= \frac{1}{c_2}s(z)+o(\frac{1}{\sqrt{n}}) . \label{17040201}
\end{align}
Combining the above results we can conclude that  for any given $z\in \mathcal{D}_+$
\begin{align*}
&\frac{1}{q}\E \ntr\Psi(z)=\frac{(1-c_1-c_2)s(z)+(1-z)s^2(z)}{c_2}+o(\frac{1}{\sqrt{n}})=	\varrho(z)+o(\frac{1}{\sqrt{n}}) .
\end{align*}
where we used (\ref{17040301}). The extension of the result to the whole $\mathcal{D}$ and the proof for the uniform bound  are  the same as the second estimate in (\ref{17040610}). We thus omit the details.

This completes the proof of Lemma \ref{lem.17032510}.

\end{proof}

\begin{proof}[Proof of Lemma \ref{lemf_1}]  Using the fact $((1-z) E-z H)\Phi(z)=I$ and  the second estimate of Lemma \ref{lem.17032510}, we see that
\begin{align}
\frac{1}{n} \ntr \big( H \Phi(\ga_j)\big)=\frac{1-\ga_j}{\ga_j n}  \ntr  E \Phi- \frac{p-k}{\ga_j n}=\frac{-c_1+(1-\ga_j)s(\ga_j)}{\ga_j} +o_{\text{p}}(\frac{1}{\sqrt{n}}). \label{17032702}
\end{align}

Now, observe that
\begin{align}
&\frac{\partial }{\partial z}\Big(\frac{1}{n} \ntr \big( H \Phi(z)\big)\Big)\Big{|}_{z=\ga_j}= \frac{1}{n} \ntr \big( H \Phi(\ga_j)W_2W_2'\big)\Phi(\ga_j)\nonumber\\
&\qquad =-\frac{1}{\ga_j n} \ntr  H \Phi(\ga_j) +\frac{1}{\ga_j n} \ntr  H\Phi(\ga_j) E \Phi(\ga_j)\nonumber\\
&\qquad =-\frac{1}{1-\ga_j}\frac{1}{n} \ntr  H \Phi(\ga_j)+\frac{1}{\ga_j n}  \ntr  H\Phi(\ga_j)  H\Phi(\ga_j). \label{17032701}
\end{align}
Hence, to prove the second and third estimate in  Lemma \ref{lemf_1}, it suffices to consider the LHS of (\ref{17032701}). Although the limit and differentiation cannot be exchanged in general, here we can indeed obtain from (\ref{17032702}) that
\begin{align}
\frac{\partial }{\partial z}\Big(\frac{1}{n} \ntr \big( H \Phi(z)\big)\Big)\Big{|}_{z=\ga_j}=\frac{{\rm{d}} }{{\rm{d}} z}\frac{1-z}{z} s(z)\Big{|}_{z=\ga_j}+O_{\text{p}}(n^{-\frac14}).\label{17032705}
\end{align}
To see this, it suffices to consider the non-limiting ratio
\begin{align}
\frac{\frac{1}{n} \ntr \big( H (\Phi(\ga_j+n^{-\frac{1}{4}})\big)- \Phi(\ga_j))\big)}{n^{-\frac{1}{4}}}=\frac{1}{n} \ntr \big( H \Phi(\ga_j)W_2W_2'\big)\Phi(\ga_j)+O_{\text{p}}(n^{-\frac{1}{4}}), \label{17050840}
\end{align}
which follows by using the expansion (\ref{resolvent expansion 1}) twice. Applying the estimate   in (\ref{17032702}) to the LHS of  (\ref{17050840}) leads to  (\ref{17032705}).

Therefore, combining (\ref{17032702}), (\ref{17032701}) and (\ref{17032705}) yields  the second and the third estimates in Lemma \ref{lemf_1}.
To prove the last estimate in Lemma \ref{lemf_1}, we can consider the derivative of $\frac{1}{n}\ntr  E \Phi$ instead of  $\frac{1}{n}\ntr  H \Phi$ and go through the above procedure. We omit the details.

This concludes the proof of Lemma \ref{lemf_1}.

\end{proof}

\begin{proof}[Proof of Lemma \ref{lemf_2}] From  the identity (\ref{17031801}), we have
\begin{align*}
\Psi=(1-z) S _{yw_2}{\Phi} S _{w_2y}-\mathcal{E}.
\end{align*}
Analogously to the proof of the third estimate in  (\ref{17032501}), we can use the spectral decomposition (\ref{17032720}). Then we write
\begin{align*}
 S _{yw_2}{\Phi} S _{w_2y}=U_{y} \Lambda_{y} V_{y}  W_2'  \Phi  W_2 V_{y}' \Lambda_{y}' U_{y}'.
\end{align*}
Again, similarly to the Gaussian approximation used in (\ref{lem.17032401}), we can approximate the $\alpha$-th row and $\beta$-th column of $U_{y}$ by Gaussian vectors. Then Lemma \ref{let} leads us to the estimates
\begin{align}
\Psi_{\alpha\beta}(\ga_j)= \delta_{\alpha\beta}\frac{1}{q}\ntr \Psi(\ga_j)+O_{\text{p}} (\frac{(\log n)^K}{\sqrt{n}})=\varrho(\ga_j)+O_{\text{p}} (\frac{(\log n)^K}{\sqrt{n}}).  \label{17032725}
\end{align}
where the last step follows from Lemma \ref{lem.17032510}.

Next, for the second estimate in  Lemma \ref{lemf_2},  we first notice that
\begin{align*}
&\Psi(\ga_j)(\mathcal{H}^{-1}+\mathcal{E}^{-1})\Psi(\ga_j)=\frac{\partial }{\partial z} \Psi(z) \Big{|}_{z=\ga_j}=\frac{\partial }{\partial z} \Big((1-z) S _{yw_2}{\Phi} S _{w_2y}\Big) \Big{|}_{z=\ga_j}\nonumber\\
&=- S _{yw_2}{\Phi} S _{w_2y}+ (1-z)   S _{yw_2}\Phi   S _{w_2w_2}\Phi   S _{w_2y}.
\end{align*}
Again, applying the SVD of $Y$, we can derive
\begin{align*}
\Big(\Psi(\ga_j)(\mathcal{H}^{-1}+\mathcal{E}^{-1})\Psi(\ga_j)\Big)_{\alpha\beta}=  \delta_{\alpha\beta}\frac{1}{ q} \frac{\partial }{\partial z} \ntr\Psi(z) \Big{|}_{z=\ga_j} +O_{\text{p}} (n^{-\frac12+\varepsilon})
\end{align*}
Then similarly to (\ref{17032705}), we have
\begin{align*}
\Big(\Psi(\ga_j)(\mathcal{H}^{-1}+\mathcal{E}^{-1})\Psi(\ga_j)\Big)_{\alpha\beta}=\delta_{\alpha\beta} \varrho'(\ga_j)+o_{\text{p}}(1).
\end{align*}
This concludes the proof of Lemma \ref{lemf_2}.
\end{proof}

\begin{proof}[Proof of Lemma \ref{lem. traces of squared matrices}] First, we estimate $\frac{1}{n} \ntr \mathscr{A}^2$. From the definition in (\ref{17050301}), we have
 \begin{align}\lb{eqlem4.4.1}
\frac{1}{n} \ntr (\mathscr{A}(\ga_j))^2&= \frac{1}{n}\ntr \big( P_{y}-\ga_j \big)^2+\frac{1}{n}\ntr \big(( P_{y}-\ga_j )^2 W_2'\Phi(\ga_j) W_2\big)^2\nonumber\\
&\qquad-\frac{2}{n}\ntr \big(( P_{y}-\ga_j  )^3 W_2'\Phi(\ga_j) W_2\big).
 \end{align}
 It is easy to check
 \begin{align*}
\frac{1}{n}\ntr \big( P_{y}-\ga_j  \big)^2= (1-2\ga_j)c_2+\ga_j^2+o_{\text{p}}(1).
 \end{align*}
 For the second term of  \eqref{eqlem4.4.1}, using Lemma \ref{lemf_1} we obtain
  \begin{align*}
 &\frac{1}{n}\ntr \big(( P_{y}-\ga_j  )^2 W_2'\Phi(\ga_j) W_2\big)^2=\frac{1}{n}\ntr \big((1-\ga_j)^2\Phi(\ga_j)E
 +\ga_j^2\Phi(\ga_j)  H\big)^2\\
 &=(1-\ga_j)(1-3\ga_j)s(\ga_j)+\ga_j(1-\ga_j)^2s'(\ga_j)+\ga_j^2 c_1+o_{\text{p}}(1).
 \end{align*}

 Similarly,  we have
 \begin{align*}
&\frac{1}{n}\ntr \big(( P_{y}-\ga_j  I_n)^3 W_2'\Phi(\ga_j) W_2\big)=
\frac{1}{n}\ntr \big((1-\ga_j)^3 \Phi(\ga_j)E
-\ga_j^3\Phi(\ga_j)  H\big)\\
&= (1-\ga_j)(1-2\ga_j)s(\ga_j)+\ga_j^2c_1+o_{\text{p}}(1).
 \end{align*}
Substituting the above results to (\ref{eqlem4.4.1}), we get the first estimate in (\ref{17052702}).

 Next, we investigate $\frac{1}{q} \ntr \mathscr{B}\mathscr{B}'$. From the definition in (\ref{17032301}), it is elementary to check
\begin{align}
\ntr \mathscr{B}\mathscr{B}'
&= (1-\ga_j)^2 \ntr \Big( S _{yy} - S _{y w_2}\Phi  S _{w_2y} +(\ga_j-\ga_j^2) S _{y w_2}\Phi  S _{w_2w_2} \Phi   S _{w_2y}\Big). \label{17031919}
\end{align}
Recall (\ref{17031801}). Then, with overwhelming probability, we have
\begin{align}
\frac{1}{q}\ntr  S _{yw_2}{\Phi}(\ga_j) S _{w_2y}=\frac{1}{1-\ga_j} \frac{1}{q}\ntr \big( \mathcal{E}+ \Psi (\ga_j)\big)=\frac{c_1+\varrho(\ga_j)}{1-\ga_j} +o_{\text{p}}(1), \label{17032010}
\end{align}
by (\ref{17031930}) and  Lemma \ref{lem.17032510}.

Observe that
\begin{align}
 S _{y w_2}\Phi(z)  S _{w_2w_2} \Phi(z)   S _{w_2y} &= \frac{\partial }{\partial z} S _{y w_2}\Phi(z) S _{w_2y} = \frac{\partial }{\partial z}  z S _{yy}\Upsilon(z) S _{yy}, \label{17032001}
\end{align}
where in the last step we used (\ref{070201}).  Using (\ref{061912}) and (\ref{16122950}), we see that
\begin{align}
&\frac{\partial }{\partial z}  z S _{yy}\Upsilon(z) S _{yy}=\frac{\partial }{\partial z}  \Big( \frac{z}{1-z}\mathcal{E}-\mathcal{H}+\frac{1}{1-z}\Psi(z)\Big)\nonumber\\
&\qquad =\frac{1}{(1-z)^2} \mathcal{E}+\frac{1}{(1-z)^2} \Psi(z)+\frac{1}{1-z} \Psi(z)\big(\mathcal{E}^{-1}+\mathcal{H}^{-1} \big) \Psi(z).  \label{17032002}
\end{align}
Furthermore, from Lemma \ref{lemf_2} we have
\begin{align}
&\frac{1}{q} \ntr \mathcal{E}= c_1+o_p(1),\qquad \frac{1}{q} \ntr \Psi(\ga_j)= \varrho(\ga_j)+o_p(1),\nonumber\\
&\frac{1}{q}  \ntr  \Psi(\ga_j)\big(\mathcal{E}^{-1}+\mathcal{H}^{-1} \big) \Psi(\ga_j)= \varrho'(\ga_j)+o_p(1).
\label{17032003}
\end{align}
Combining (\ref{17032001}), (\ref{17032002}) and (\ref{17032003}) yields
\begin{align}
 \frac{1}{q} \ntr   S _{y w_2}\Phi(\ga_j)  S _{w_2w_2} \Phi(\ga_j)   S _{w_2y}=\frac{c_1+ \varrho(\ga_j)}{(1-\ga_j)^2}+\frac{\varrho'(\ga_j)}{1-\ga_j}+o_p(1). \label{17032011}
\end{align}
Therefore, we get the second estimate in (\ref{17052702}) by substituting (\ref{17032010}) and  (\ref{17032011}) into (\ref{17031919}).

At the end, we prove the last estimate in Lemma \ref{lem. traces of squared matrices}.  Note that
\begin{align}
\frac{1}{q}\ntr \mathscr{D}_n^2 &=(1-\ga_j^2)\frac{1}{q}\ntr \Big( S _{yy}-(1-\ga_j) S _{y w_2}\Phi S _{w_2y}\Big)^2
\nonumber\\
&= (1-\ga_j)^2\frac{1}{q}\ntr \big(\mathcal{H}-\Psi\big)^2= (1-\ga_j)^2\Big(\frac{1}{q}\ntr \mathcal{H}^2-\frac{2}{q}\ntr \mathcal{H}\Psi+ \frac{1}{q}\ntr\Psi^2\Big).
\label{17032060}
\end{align}
First, it is not difficult to check
\begin{align}
\frac{1}{q} \ntr \mathcal{H}^2 \stackrel{\text{p}} \longrightarrow  \Big(1+\frac{c_2}{1-c_1}\Big) (1-c_1)^2=(1-c_1+c_2)(1-c_1). \label{17032075}
\end{align}
We further claim that
\begin{align}
&\frac{1}{q}\ntr \mathcal{H}\Psi= \frac{1-c_1-c_1c_2}{c_2} (1-\ga_j) s(\ga_j) +\frac{1}{c_2} (1-\ga_j)^2 s^2(\ga_j)-c_1(1-c_1)+o_{\text{p}}(1),\nonumber\\
&\frac{1}{q}\ntr\Psi^2= \Big(\frac{\ga_j(\ga_j-1)}{c_2}-2c_1(1-\ga_j)+\frac{1-c_1-c_2}{c_2}(1-2\ga_j)\Big)s(\ga_j)\nonumber\\
&\quad+\frac{(1-\ga_j)(2-3\ga_j)}{c_2}s^2(\ga_j)+\frac{\ga_j(1-\ga_j)^2}{c_2}s'(\ga_j) -(1-c_1-c_2)c_1+o_{\text{p}}(1). \label{17040250}
\end{align}
The proof of (\ref{17040250}) will be postponed to the end.  Consequently, we have
\begin{align*}
\frac{1}{q}\ntr \mathscr{D}_n^2 &=(1-\ga_j)^2\bigg(\big(1-c_1+c_2\big)+\Big(\frac{\ga_j(\ga_j-1)}{c_2}-2(1-\ga_j)-\frac{1-c_1-c_2}{c_2}\Big) s(\ga_j)\nonumber\\
&-\frac{\ga_j(1-\ga_j)}{c_2} s^2(\ga_j)+\frac{\ga_j(1-\ga_j)^2}{c_2}s'(\ga_j)\bigg)+o_{\text{p}}(1).
\end{align*}

In addition, we see that
\begin{align*}
\frac{1}{q}\ntr \mathscr{D}_n= (1-\ga_j) \frac{1}{q}\ntr \big(\mathcal{H}-\Psi\big)=(1-\ga_j)\Big(1-\frac{1-\ga_j}{c_2} s(\ga_j)\Big)+o_{\text{p}}(1).
\end{align*}
Then, by definition, we get the last estimate in (\ref{lem. traces of squared matrices}).

Hence, what remains is to prove (\ref{17040250}). Analogously to the proof of Lemma \ref{lem.17032510}, by continuity, we can work with
\begin{align*}
\tilde{\ga}_j=\ga_j+\mathrm{i}\frac{1}{n} \in \mathcal{D}_+
\end{align*}
 instead of $\ga_j$.  Recall $\tilde{Q}$ defined in (\ref{17050844}).  Similarly to  (\ref{17040631}), also (\ref{17032601}) and (\ref{17032602}), we can split the tasks into two steps: first, we show that
\begin{align}
\text{Var} \Big( \big(\frac{1}{q}\ntr \mathcal{H}\Psi(\tilde{\ga}_j) \big)\tilde{\mathcal{Q}}\Big)=O(\frac{1}{n^2}), \qquad  \text{Var} \Big(  \big(\frac{1}{q}\ntr\Psi^2(\tilde{\ga}_j)\big) \tilde{\mathcal{Q}}\Big)=O(\frac{1}{n^2})  \label{17040750}
\end{align}
and then we estimate $\mathbb{E} \big(\big(\frac{1}{q}\ntr \mathcal{H}\Psi(\tilde{\ga}_j) \big)\tilde{\mathcal{Q}}\big)$ and $\mathbb{E} \big(\big(\frac{1}{q}\ntr\Psi^2(\tilde{\ga}_j) \big) \tilde{\mathcal{Q}}\big)$. The proof of  (\ref{17040750}) is analogous to the proof of (\ref{17040631}). Hence, we omit it. In the sequel, we briefly state the estimates of $\mathbb{E} \big(\big(\frac{1}{q}\ntr \mathcal{H}\Psi(\tilde{\ga}_j) \big)\tilde{\mathcal{Q}}\big)$ and $\mathbb{E} \big(\big(\frac{1}{q}\ntr\Psi^2(\tilde{\ga}_j) \big) \tilde{\mathcal{Q}}\big)$,
using integration by parts formula for Gaussian measure.   Again, for brevity, we get rid of the factor
$\tilde{\mathcal{Q}}$ in the discussion, and do the reasoning as if $\|S_{yy}^{-1}\|=O(1)$ deterministically. A rigorous justification of the following derivation requires one to put $\tilde{\mathcal{Q}}$ back into each step.

To ease the presentation, we omit $\tilde{\ga}_j$ from the notations in the following discussions. First, from (\ref{17040410}) we have
\begin{align*}
\frac{1}{q} \ntr  \mathcal{H}\Psi=\frac{1}{q} \ntr  \mathcal{H}^2\Upsilon\mathcal{E},\qquad  \frac{1}{q} \ntr \Psi^2= \frac{1}{q} \ntr  \big(\mathcal{H}\Upsilon \mathcal{E}\big)^2.
\end{align*}
 We use integration by parts (c.f. (\ref{17052701})) again. Recall (\ref{17032604}). We have
\begin{align*}
&\frac{1}{q} \mathbb{E} \ntr  \mathcal{H}^2\Upsilon\mathcal{E}= \frac{1}{q} \sum_{\alpha,\beta} \mathbb{E} x_{\alpha\beta} \big(\mathcal{X}'\mathcal{H}\Upsilon\mathcal{E}\big)_{\beta\alpha}\nonumber\\
&= \frac{1}{q n} \sum_{\alpha,\beta} \mathbb{E} \Big( \big(\mathcal{H}\Upsilon\mathcal{E}\big)_{\alpha\alpha}+ \big(\mathcal{X}'\big)_{\beta\alpha} \big(\mathcal{X}' \Upsilon\mathcal{E}\big)_{\beta\alpha}+\big(\mathcal{X}'\mathcal{X}\big)_{\beta\beta} \big(\Upsilon\mathcal{E}\big)_{\alpha\alpha}
\nonumber\\
&\qquad+\ga_j  \big(\big(\mathcal{X}'\mathcal{H}\Upsilon\big)_{\beta\alpha} (\mathcal{X}'\Upsilon\mathcal{E})_{\beta\alpha}+\big(\mathcal{X}'\mathcal{H}\Upsilon\mathcal{X}\big)_{\beta\beta} \big( \Upsilon\mathcal{E}\big)_{\alpha\alpha}\Big)\nonumber\\
&= \frac{n-p+k}{n}\frac{1}{q} \mathbb{E} \ntr \big(\mathcal{H}\Upsilon\mathcal{E}\big)+\frac{1}{qn} \mathbb{E} \ntr \mathcal{H}\; \ntr \mathcal{E} \Upsilon+\ga_j \frac{1}{qn} \mathbb{E} \ntr \mathcal{H}^2  \Upsilon\; \ntr \mathcal{E}  \Upsilon+O(\frac{1}{n})\nonumber\\
&= \frac{n-p+k}{n}\frac{1}{q} \mathbb{E} \ntr \big(\mathcal{H}\Upsilon\mathcal{E}\big)+\frac{1}{qn} \mathbb{E} \ntr \mathcal{H}\; \mathbb{E}\ntr \mathcal{E} \Upsilon+\ga_j \frac{1}{qn} \mathbb{E} \ntr \mathcal{H}^2  \Upsilon\; \mathbb{E}\ntr \mathcal{E}  \Upsilon+O(\frac{1}{n}),
\end{align*}
where the last step follows from the variance bounds $O(\frac{1}{n^2})$ for all the normalized tracial quantities above. The proofs of these variance bounds are again similar to the proof of the first bound in (\ref{17040631}).

Now, note that three terms on the RHS above can be computed easily. More specifically, we first recall (\ref{17040201}). Then, we also note that
\begin{align}
\frac{1}{q} \mathbb{E} \ntr \big(\mathcal{H}\Upsilon\mathcal{E}\big)=\frac{1}{q} \mathbb{E} \ntr \Psi(\ga_j)=\varrho(\ga_j) +o(1) \label{17040230}
\end{align}
and
\begin{align*}
&\frac{1}{n} \mathbb{E}\ntr \cH^2\Upsilon(\ga_j)=\frac{1}{\ga_j n} \mathbb{E}\ntr \mathcal{H} \Upsilon(\ga_j) (\ga_j \mathcal{H}-(1-\ga_j)\mathcal{E})+\frac{1-\ga_j}{\ga_j n} \mathbb{E}\ntr \mathcal{H} \Upsilon(\ga_j) \mathcal{E}\nonumber\\
&\qquad =-\frac{1}{\ga_j n} \mathbb{E}\ntr \mathcal{H}+  \frac{1-\ga_j}{\ga_j n}\mathbb{E}\ntr\Psi(\ga_j)=\frac{c_2(c_1-1)}{\ga_j}+ \frac{1-\ga_j}{\ga_j } c_2 \varrho(\ga_j)+o(1).
\end{align*}
Hence, we have
\begin{align*}
\frac{1}{q} \mathbb{E} \ntr  \mathcal{H}^2\Upsilon\mathcal{E}&=\big((1-c_1) +(1-\ga_j)s(\ga_j)\big) \varrho(\ga_j) +O(\frac{1}{n})\nonumber\\
&= \frac{1-c_1-c_1c_2}{c_2} (1-\ga_j) s(\ga_j) +\frac{1}{c_2} (1-\ga_j)^2 s^2(\ga_j)-c_1(1-c_1)+o(1).
\end{align*}

Analogously, after a tedious calculation using integration by parts, we will get
\begin{align}
& \frac{1}{q} \mathbb{E}\ntr \Psi^2= \frac{1}{q} \sum_{\alpha,\beta} \mathbb{E}x_{\alpha\beta} \Big(\mathcal{X}'\Upsilon \mathcal{E} \mathcal{H}\Upsilon \mathcal{E}\Big)_{\beta\alpha}\nonumber\\
  &=   \frac{n-p+k}{n} \frac{1}{q } \mathbb{E} \ntr   \mathcal{H}\Upsilon \mathcal{E} \Upsilon \mathcal{E}+ \frac{1}{q n}\mathbb{E}\ntr \mathcal{H} \Upsilon\mathcal{E}\; \mathbb{E}\ntr  \mathcal{E}\Upsilon+z\frac{1}{qn} \mathbb{E} \ntr \mathcal{H}\Upsilon \; \mathbb{E}\ntr \mathcal{H}\Upsilon\mathcal{E} \Upsilon\mathcal{E} \nonumber\\
 &\qquad +z\frac{1}{qn} \mathbb{E} \ntr   \mathcal{H}\Upsilon \mathcal{H} \Upsilon \mathcal{E} \;  \mathbb{E}\ntr \mathcal{E} \Upsilon +O(\frac{1}{n}). \label{17040231}
\end{align}
The tracial quantities on the RHS can now be computed. For example, the terms like $\mathcal{H}\Upsilon \mathcal{E} \Upsilon \mathcal{E}$ can be obtained via taking derivative of $\mathcal{H}\Upsilon \mathcal{E}$. More specifically, we have
\begin{align*}
&\frac{\partial}{\partial z}  \frac{1}{q } \mathbb{E} \ntr   \mathcal{H}\Upsilon (z) \mathcal{E}  \Big{|}_{z=\ga_j}= \frac{1}{q } \mathbb{E} \ntr   \mathcal{H}\Upsilon (\ga_j) \big(\mathcal{E}+\mathcal{H}\big) \Upsilon(\ga_j) \mathcal{E}\nonumber\\
&=\frac{1}{1-\ga_j}  \frac{1}{q } \mathbb{E} \ntr \mathcal{H} \Upsilon(\ga_j) \mathcal{E} +\frac{1}{1-\ga_j}  \frac{1}{q } \mathbb{E} \ntr \mathcal{H} \Upsilon(\ga_j)\mathcal{H} \Upsilon(\ga_j) \mathcal{E} \nonumber\\
&=- \frac{1}{\ga_j}  \frac{1}{q } \mathbb{E} \ntr \mathcal{H} \Upsilon(\ga_j) \mathcal{E} +\frac{1}{\ga_j}  \frac{1}{q } \mathbb{E} \ntr \mathcal{H} \Upsilon(\ga_j)\mathcal{E} \Upsilon(\ga_j) \mathcal{E}.
\end{align*}
From the above and (\ref{17040230}), we get
\begin{align*}
\frac{1}{q } \mathbb{E} \ntr \mathcal{H} \Upsilon(\ga_j)\mathcal{H} \Upsilon(\ga_j) \mathcal{E}&=(1-\ga_j)\varrho'(\ga_j)- \varrho(\ga_j)+o(1)\nonumber\\
&=-\frac{2(1-\ga_j)}{c_2}s(\ga_j)+\frac{(1-\ga_j)^2}{c_2} s'(\ga_j)+c_1+o(1), \nonumber\\
 \frac{1}{q } \mathbb{E} \ntr \mathcal{H} \Upsilon(\ga_j)\mathcal{E} \Upsilon(\ga_j) \mathcal{E}
 &=\ga_j  \varrho'(\ga_j)+\varrho(\ga_j)+o(1)\nonumber\\
&=\frac{1-2\ga_j}{c_2} s(\ga_j)+\frac{\ga_j(1-\ga_j)}{c_2} s'(\ga_j)-c_1+o(1).
\end{align*}
Further, we recall (\ref{17040201}), from which we also have
\begin{align*}
\frac{1}{q} \mathbb{E}\ntr \big(\mathcal{H}\Upsilon\big)=-\frac{1}{\ga_j}+\frac{1-\ga_j}{c_2\ga_j} s(\ga_j) +o(1).
\end{align*}

Plugging all these estimates into (\ref{17040231}), we get
\begin{align}
\frac{1}{q} \mathbb{E}\ntr \Psi^2(\ga_j)
&=
 \frac{2(1-\ga_j)(1-2\ga_j)}{c_2}s^2(\ga_j)+\Big(\frac{1-c_1-c_2}{c_2}\ga_j(1-\ga_j)+2\frac{\ga_j(1-\ga_j)^2}{c_2}s(\ga_j)\Big)s'(\ga_j)\nonumber\\
&\quad +\Big(-2c_1(1-\ga_j)+\frac{1-c_1-c_2}{c_2}(1-2\ga_j)\Big)s(\ga_j)+o(1). \label{17050850}
\end{align}
Taking the derivative for (\ref{17040301}), one can derive
\begin{align}
2(\ga_j-1)s(\ga_j)s'(\ga_j)=-s^2(\ga_j)+s(\ga_j)-(c_1+c_2-\ga_j) s'(\ga_j) \label{17050851}
\end{align}
Using (\ref{17050851}) to (\ref{17050850}), we obtain the second estimate in (\ref{17040250}).

This completes the proof of Lemma  \ref{lem. traces of squared matrices}.

\end{proof}

\section{} \label{A.B}

In this section, we prove Lemma \ref{lem.17042902}.  First, from Lemma   \ref{lem.062802}, $( M_n)_{ij}(z)$ can be approximated by $\delta_{ij}m_{i}(z)$ up to an error $O(n^{-\varepsilon})$ if $z\in \mathcal{D}$.  We now claim that this approximation also holds in the following domain $\widetilde{D}$: we set for any $\tau>0, \varepsilon_0$ the domain
\begin{align*}
\widetilde{\mathcal{D}}\equiv \widetilde{\mathcal{D}}(\delta, \varepsilon_0):= \Big\{z=x+\mathrm{i}\eta: d_+-\delta\leq x\leq  1-\delta, n^{-1+\varepsilon_0}\leq \eta\leq 1\Big\}.
\end{align*}

\begin{lem} \label{lem.17043001} For any small positive constants and $\delta$ and $\varepsilon_0$,  there exists some positive constant $c>0$, such that
\begin{align*}
\sup_{z\in \widetilde{D}} \sup_{i,j=1,\ldots, k}|( M_n)_{ij}(z)- \delta_{ij} m_i(z)| \leq n^{-c}
\end{align*}
in probability.
\end{lem}

For simplicity, in the sequel, we denote
\begin{align*}
\eta_0=n^{-1+\varepsilon_0}.
\end{align*}

To prove Lemma \ref{lem.17043001}, we need a  local law for the Stieltjes transform of MANOVA matrix. Recall $\check{s}(z)$ from (\ref{17040350}).  We also recall the Stieltjes transform $\check{s}_n(z)$  from (\ref{17040441}).
\begin{lem}[Local law of MONOVA at the right edge] \label{lem. local law at the right edge}  For any small $\delta>0$ and $\varepsilon_0>0$,  we have
\begin{align}
\Big| \check{s}_n(z)-\check{s}(z)\Big|\leq \frac{N^{\varepsilon'}}{N\eta} \label{17051030}
\end{align}
uniformly on $\widetilde{\mathcal{D}}$
with overwhelming probability, for any small constant $\varepsilon'>0$.
\end{lem}

\begin{proof} Observe that
\begin{align*}
 ( (E+H)^{-1}E-z)^{-1}&= ((1-z) H^{-1} E-zI)^{-1} (I+H^{-1}E)\nonumber\\
 &= \frac{1}{(1-z)^2} \big(H^{-1}E-\frac{z}{1-z}I\big)^{-1}+\frac{1}{1-z} I.
\end{align*}
Hence, by (\ref{pres1}) and (\ref{17040441}) we have
\begin{align}
\check{s}_n(z)=   \frac{1}{(1-z)^2} \hat{s}_n\big(\frac{z}{1-z}\big)+\frac{1}{1-z}+O(\frac{1}{n}), \label{17051001}
\end{align}
where
\begin{align*}
\hat{s}_n(\omega):= \frac{1}{p} \ntr \big(H^{-1}E-\omega\big)^{-1}
\end{align*}
is the Stieltjes transform of the F matrix $H^{-1}E$.    Observe that $\Im (\frac{z}{1-z}) \sim \Im z$ when $\Re z\in [d_+-\delta, 1-\delta]$. A local law for $\hat{s}_n(\omega)$ at the right edge of $H^{-1}E$ with $\Im \omega\geq \eta_0$ has been proved in \cite{HanP16T} (see Theorem 7.1 therein). Then by (\ref{17051001}), Lemma \ref{lem. local law at the right edge} follows immediately.
\end{proof}

Denoting by $\widetilde{\gamma}_i\equiv \widetilde{\gamma}_{n,i}$ the $i$-th largest $n$-quantile of $f(x)$ in (\ref{17040310}), i.e.,
\begin{align}
\int_{d_-}^{\widetilde{\gamma}_i} f(x) {\rm d} x=\frac{n-i+1}{n}.  \label{170530100}
\end{align}

\begin{cor}[Rigidity of eigenvalues near the right edge]   \label{cor. 17043020} For any small $\delta$ in Lemma \ref{lem. local law at the right edge}, there exists some $c\equiv c(\delta)\in (0,1)$ such that
\begin{align}
\big|\widetilde{\lambda}_i-\widetilde{\gamma}_i\big|\leq n^{\varepsilon'} n^{-\frac{2}{3}}i^{-\frac{1}{3}}, \qquad \forall i=1, \ldots, \lfloor cn\rfloor. \label{17051005}
\end{align}
holds with overwhelming probability, for any small constant $\varepsilon'>0$
\end{cor}

\begin{proof} Denoting by $\hat{\lambda}_1\geq \hat{\lambda}_2\geq \cdots \hat{\lambda}_{p-k} $ the ordered eigenvalues of the F matrix $H^{-1}E$. Observe that  $C_{w_2y}= (E+H)^{-1} E=H^{-1}E(H^{-1}E+I)^{-1}$. Hence, we have
\begin{align*}
\widetilde{\lambda}_i= \frac{\hat{\lambda}_i}{1+\hat{\lambda}_i}.
\end{align*}
Further we set $\hat{\gamma}_i$ to be the solution to
\begin{align*}
\widetilde{\gamma}_i=\frac{\hat{\gamma}_i}{1+\hat{\gamma_i}}.
\end{align*}
In \cite{HanP16T}, it is proved that $|\hat{\lambda}_1-\hat{\gamma}_1|\leq n^{-\frac{2}{3}+\varepsilon'}$ with overwhelming probability. This implies directly that
\begin{align}
\big|\widetilde{\lambda}_1-\widetilde{\gamma}_1\big|\leq  n^{-\frac{2}{3}+\varepsilon'}. \label{17051008}
\end{align}
with overwhelming probability. The rigidity for general $i$ then follows from (\ref{17051008}), Lemma \ref{lem. local law at the right edge}, and the square root decay of the density function in (\ref{17040310}).  The proof is standard, we omit the details. For interested readers, we refer to the proof of Theorem 2.2 in \cite{EYY}.
\end{proof}

With the aid of Lemma \ref{lem. local law at the right edge}, we are now at the stage to prove Lemma  \ref{lem.17043001}.
\begin{proof}[Proof of Lemma \ref{lem.17043001}]  Here we only need to argue that the proof of Lemma   \ref{lem.062802} still works well for $z\in \widetilde{D}$.  It suffices to go through the proofs of Lemmas \ref{lem.17032401}, \ref{lem5.4} and \ref{lem.17032510} again, with the aid of Lemma \ref{lem. local law at the right edge} and Corollary  \ref{cor. 17043020}.

First, analogously to Lemma \ref{lem.17032401}, we will show that
\begin{align}
 \mathscr{M}_2=  \mathcal{T} G_1\mathscr{B} W_1'+O(n^{-\frac{1}{2}+2\varepsilon_0}),\qquad   \mathscr{M}_3=  \mathcal{T} G_1 \mathscr{D} G_1' T_1'+O(n^{-\frac{1}{2}+2\varepsilon_0}), \label{17050140}
 \end{align}
 hold with overwhelming probability on $\widetilde{\mathcal{D}}$. Here the second term on the RHS of the second equation in (\ref{17052595}) has been absorbed into the error term $O(n^{-\frac{1}{2}+2\varepsilon_0})$. Similarly to the proof of Lemma \ref{lem.17032401}, we only show the details for the case of $k=2$.  Recall (\ref{17051020}) and (\ref{17050101}). Denote by $\mathbf{w}_\ell$ the $\ell$-th column of $W_1$.  It suffices to show that with overwhelming probability,
 \begin{align}
\mathbf{g}_1\mathscr{D}(z)\mathbf{g}_1'=O(n^{\varepsilon_0}), \qquad \mathbf{g}_1\mathscr{D}(z)\mathbf{g}_2'=O(n^{\varepsilon_0}), \qquad \mathbf{g}_i'\mathscr{B}(z) \mathbf{w}_j=O(n^{\varepsilon_0})\label{17050103}
\end{align}
hold  for $z\in \widetilde{\mathcal{D}}$. We only show the proof for the first two estimates above. The last one can be checked analogously.
 We use
 Lemma \ref{let} again. For the first two estimates in (\ref{17050103}), it suffices to show that
 \begin{align}
 \frac{1}{n}\ntr \mathscr{D}(z)=O(n^{\varepsilon_0}),\qquad \frac{1}{n^2}\|\mathscr{D}(z)\|_{HS}^2=\frac{1}{n^2}\ntr \mathscr{D}(z) (\mathscr{D}(z))^*=O(n^{-\frac{\varepsilon_0}{2}})  \label{17050111}
 \end{align}
 hold with overwhelming probability on $\widetilde{\mathcal{D}}$.  We start with  the first estimate in (\ref{17050111}).  From the first equation of (\ref{17032621}) and (\ref{17050130}),  we have
 \begin{align*}
 &| \frac{1}{n}\ntr \mathscr{D}(z)|\leq  |1-z| | \frac{1}{n}\ntr  S_{yy}|+|1-z|^2|  \frac{1}{n}\ntr  S_{y w_2} \Phi(z) S_{w_2y} |\nonumber\\
 &\leq C\big(1+\frac{1}{n} \ntr | \widetilde{C}_{w_2y}-z|\big)
 \end{align*}
 with overwhelming probability, where we also used (\ref{0726100}).  Observe that $ \widetilde{C}_{w_2y}$ has the same eigenvalues as $ C_{w_2y}$. Hence, we have on $\widetilde{D}$,
 \begin{align}
 \frac{1}{n} \ntr | \widetilde{C}_{w_2y}-z|=\frac{1}{n}\sum_{i=1}^{p-k} \frac{1}{|\widetilde{\lambda}_i-z|}=O((\log n)^K) \label{17050120}
 \end{align}
 with overwhelming probability, for some sufficiently large  $K>0$. The proof of (\ref{17050120}) from Corollary \ref{cor. 17043020} is standard. We refer to the proof of (3.43) in \cite{BPZ} for a similar argument.

  For the second estimate, we have
 \begin{align*}
 \frac{1}{n^2}\ntr \mathscr{D}(z) (\mathscr{D}(z))^*\leq   |1-z| | \frac{1}{n^2}\ntr  S_{yy} (\mathscr{D}(z))^*|+|1-z|^2|  \frac{1}{n^2}\ntr  S_{y w_2} \Phi(z) S_{w_2y} (\mathscr{D}(z))^*|
 \end{align*}
 By Von Neumann's trace inequality, we can easily get
 \begin{align*}
 \big|\frac{1}{n^2}\ntr  S_{yy} (\mathscr{D}(z))^*\big|\leq \frac{1}{n^2}\ntr  S_{yy} \|(\mathscr{D}(z))^*\|_{\text{op}}\leq \frac{C}{ n\eta_0}=Cn^{-\varepsilon_0}
 \end{align*}
 with overwhelming probability. Analogously, using (\ref{0726100}),  (\ref{17050130}) and (\ref{17050120}),  we have
 \begin{align*}
 \big|\frac{1}{n^2}\ntr  S_{y w_2} \Phi(z) S_{w_2y} (\mathscr{D}(z))^*|\leq  \frac{C}{n^2\eta_0} \ntr | \widetilde{C}_{w_2y}-z|\leq n^{-\frac{\varepsilon_0}{2}}
 \end{align*}
 with overwhelming probability.

 Next, we shall check that an analogue of Lemma \ref{lem5.4}   holds as well on $\widetilde{D}$.  Specifically, we shall show,  that on $\widetilde{D}$, the following estimates hold with overwhelming probability.
\begin{align}
&\sup_{i,j}\Big|(\scM_1)_{ij}-\delta_{ij}\Big(c_2-(1+c_1)z-(1-z)\frac{1}{n}\ntr\big(E\Phi\big)\Big)\Big|\leq {n^{-\frac{\varepsilon_0}{8}}},\nonumber\\
& \sup_{i,j}\Big|(\scM_2)_{ij}\Big|\leq n^{-\frac{\varepsilon_0}{8}}, \qquad \sup_{i,j}\Big|(\scM_3)_{ij}-\delta_{ij} t_i^2(1-z)\frac{1}{q}\ntr(\mathcal{H}-\Psi)\Big|\leq n^{-\frac{\varepsilon_0}{8}}. \label{17050135}
\end{align}
Similarly to the proof of Lemma \ref{lem5.4}, we start with  (\ref{17040420}) and  (\ref{17050140}) and use the large deviation inequality in Lemma \ref{let}.   It suffices to show that
\begin{align}
\frac{1}{n^2} \|\mathscr{A}\|_{HS}^2\leq n^{-\frac{\varepsilon_0}{2}}, \qquad \frac{1}{n^2} \|\mathscr{B}\|_{HS}^2\leq n^{-\frac{\varepsilon_0}{2}}, \qquad \frac{1}{n^2} \|\mathscr{D}\|_{HS}^2\leq n^{-\frac{\varepsilon_0}{2}}\label{17050150}
\end{align}
holds on $\widetilde{D}$ with overwhelming probability. Note that the last estimate of (\ref{17050150}) has already been proved in (\ref{17050111}). The first two can be proved analogously. We thus omit the details.

At the end, we show that an analogue of Lemma \ref{lem.17032510} holds on the domain $\widetilde{D}$ as well.  Apparently, the first estimate in (\ref{17040610}) still holds. In the sequel,  we will show that
\begin{align}
&\sup_{z\in \widetilde{D}}|\frac{1}{n} \ntr  E\Phi(z)-s(z)|\leq n^{-\frac{\varepsilon_0}{2}},\nonumber\\
&\sup_{z\in \widetilde{D}}|\frac{1}{q} \ntr \Psi(z)- \varrho(z)|\leq n^{-\frac{\varepsilon_0}{2}}. \label{17050157}
\end{align}
with overwhelming probability.
Note that the first estimate in (\ref{17050157}) follows from (\ref{17050156}) and Lemma \ref{lem. local law at the right edge}  by choosing $\varepsilon'=\frac{\varepsilon_0}{2}$ therein.

Fo the second estimate in (\ref{17050157}), we first define
\begin{align}
\mathfrak{P}\equiv \mathfrak{P}(z):= \frac{1}{n} \ntr  S _{yy}\Upsilon(z) S _{yy}-  \frac{1}{n }\ntr  \Upsilon(z) S _{yy}. \label{17051055}
\end{align}
Our aim is to show that
\begin{align}
\sup_{z\in \widetilde{D}}|\mathfrak{P}(z)|\leq Cn^{-\frac{1}{2}} \label{17051450}
\end{align}
with overwhelming probability.

We postpone the proof of   (\ref{17051450}) to the end, and show how to  prove the second estimate in (\ref{17050157}) with the aid of (\ref{17051450}).  First, from (\ref{070201}), (\ref{pres1}) and (\ref{17031801}), we can write
\begin{align*}
 \Psi(z)= z(1-z) S_{yy} \Upsilon S_{yy}+(1-z) \mathcal{H}-z\mathcal{E}.
\end{align*}
Hence, from (\ref{17051055}) and (\ref{17051450}), we have
\begin{align}
\frac{1}{q}\ntr \Psi(z)= z(1-z) \frac{1}{q }\ntr  \Upsilon(z) S _{yy}+(1-z) \frac{1}{q}\ntr \mathcal{H}- z\frac{1}{q}\ntr\mathcal{E}+O(n^{-\frac{1}{2}}) \label{17052720}
\end{align}
on $\widetilde{D}$, with overwhelming probability.  It is elementary to see that
\begin{align}
\frac{1}{q}\ntr \mathcal{H}=1-c_1+O(n^{-\frac{1}{2}+\varepsilon}), \qquad \frac{1}{q}\ntr \mathcal{E}=c_1+ O(n^{-\frac{1}{2}+\varepsilon}) \label{17052721}
\end{align}
with overwhelming probability. Furthermore, according to (\ref{pres1}) and (\ref{061911}), we can write
\begin{align}
\frac{1}{q }\ntr  \Upsilon(z) S _{yy}=\frac{1}{q}\ntr \big(C_{yw_2}-z\big)^{-1}=:\tilde{s}_n(z). \label{17052715}
\end{align}
Since $C_{w_2y}$ and $C_{yw_2}$ share the same nonzero eigenvalues, in light of  (\ref{17040441}), we have
\begin{align}
\frac{c_1}{c_2} \check{s}_n(z)=\tilde{s}_n(z)+\frac{c_2-c_1}{c_2}\frac{1}{z}.  \label{17052710}
\end{align}
From (\ref{17040615}), we also have
\begin{align}
\frac{c_1}{c_2} \check{s}(z)=\tilde{s}(z)+\frac{c_2-c_1}{c_2}\frac{1}{z}. \label{17052711}
\end{align}
Hence, combining (\ref{17051030}),  (\ref{17052710}) and (\ref{17052711}) yields
\begin{align}
\sup_{z\in\widetilde{\mathcal{D}}}\Big| \tilde{s}_n(z)-\tilde{s}(z)\Big|\leq \frac{N^{\varepsilon'}}{N\eta_0}. \label{17052713}
\end{align}
Plugging (\ref{17052715}), (\ref{17052713}) and (\ref{17052721}) into (\ref{17052720}), we can get (\ref{17050157}) easily  by choosing $\varepsilon'=\frac{\varepsilon_0}{2}$.

Hence, what remains is to prove  (\ref{17051450}).   It suffices to show the one point estimate for any given $z\in \widetilde{D}$. The uniform bound follows from the continuity and the definition of overwhelming probability. More specifically, similarly to the proof of Lemma \ref{lem.17032510}, we introduce the lattice $\widetilde{\mathcal{S}}:=\widetilde{D}\cap (n^{-3}\mathbb{Z}\times n^{-3}\mathbb{Z})$.  According to the definition of with overwhelming probability, we see that if we can prove $|\mathfrak{P}(z)|\leq n^{-\frac{1}{2}}$ for any $z\in\widetilde{\mathcal{S}}$ with overwhelming probability individually, we also have the uniform bound $ \sup_{z\in \widetilde{\mathcal{S}}}|\mathfrak{P}(z)|\leq n^{-\frac{1}{2}} $ with overwhelming probability.  Moreover, for any $z'\in \widetilde{\mathcal{D}}$, there is one $z\in \widetilde{\mathcal{S}}$ such that $|z-z'|\leq n^{-3}$. Further, we have  the following bound uniformly in all $z\in \widetilde{\mathcal{D}}$
\begin{align}
|\mathfrak{P}'(z)|\leq C\|\Upsilon'(z)\|=O(\frac{1}{\eta_0^2})=O(n^2),  \label{170530105}
\end{align}
conditioning on the event $\|S_{yy}\|, \|S_{yy}^{-1}\|\leq C$, which itself holds with overwhelming probability.  By the continuity in (\ref{170530105}), we can further obtain (\ref{17051450}) from (\ref{170530105}).

Therefore, in the sequel, we only show the proof of  $|\mathfrak{P}(z)|\leq n^{-\frac{1}{2}}$ for one $z\in \widetilde{\mathcal{D}}$.
For brevity, we omit $z$ from the notation. We estimate the high order moments of $\mathfrak{P}$. To this end, we set
\begin{align*}
\mathfrak{H}(m_1,m_2)= \mathfrak{P}^{m_1} \overline{\mathfrak{P}}^{m_2}.
\end{align*}
With this notation, we can write
\begin{align}
\mathbb{E} \mathfrak{H}(m,m)= \mathbb{E} \frac{1}{n} \ntr  S _{yy}\Upsilon S _{yy} \mathfrak{H}(m-1, m)- \mathbb{E} \frac{1}{n} \ntr \Upsilon S _{yy} \mathfrak{H}(m-1, m). \label{17051463}
\end{align}
We apply the Gaussian integral by parts (c.f. (\ref{17052701})) to the first  term above. Then we get
\begin{align}
&\mathbb{E} \Big(\frac{1}{n} \ntr  S _{yy}\Upsilon S _{yy} \mathfrak{H}(m-1, m)\Big)=\frac{1}{n} \sum_{i,j}\mathbb{E} \Big(y_{ij} \big(Y'\Upsilon S_{yy} \big)_{ji} \mathfrak{H}(m-1, m)\Big)\nonumber\\
&=\mathbb{E} \big( \mathfrak{c}_1\mathfrak{H}(m-1, m)\big)+\mathbb{E} \big( \mathfrak{c}_2\mathfrak{H}(m-2, m)\big)+\mathbb{E} \big( \mathfrak{c}_3\mathfrak{H}(m-1, m-1)\big) \label{17051461}
\end{align}
where
\begin{align*}
&\mathfrak{c}_1=\frac{1}{n^2 } \sum_{i,j}\frac{\partial \big(Y'\Upsilon S_{yy} \big)_{ji}}{\partial y_{ij}},\qquad \mathfrak{c}_2=   \frac{1}{n^2 } \sum_{i,j}   \big(Y'\Upsilon S_{yy} \big)_{ji} \frac{\mathfrak{P}}{\partial y_{ij}}, \nonumber\\
& \mathfrak{c}_3=   \frac{1}{n^2 } \sum_{i,j}   \big(Y'\Upsilon S_{yy} \big)_{ji} \frac{\overline{\mathfrak{P}}}{\partial y_{ij}}.
\end{align*}
The estimates of  $\mathfrak{c}_2$ and $\mathfrak{c}_3$ all nearly the same. We thus focus on the former. By elementary calculation, it is not difficult to derive
\begin{align}
\mathfrak{c}_1=&  \frac{1}{n^2 } \sum_{i,j}\Big( ( \Upsilon S _{yy})_{ii}-( Y'  \Upsilon)_{ji} \big(( P_{w_2}-z)  Y'\Upsilon S _{yy}\big)_{ji}\nonumber\\
& -  \big( Y'  \Upsilon  Y( P_{w_2}-z)\big)_{jj} (\Upsilon S _{yy})_{ii}+ ( Y'  \Upsilon)_{ji} y_{ij}+( Y' \Upsilon  Y)_{jj}\Big)\nonumber\\
= & \frac{1}{n} \ntr \Upsilon S_{yy}- \frac{1}{n^2} \ntr \Upsilon Y ( P_{w_2}-z)  Y'\Upsilon S _{yy}\nonumber\\
&- \frac{1}{n^2} \ntr   \Upsilon  Y( P_{w_2}-z)Y' \ntr \Upsilon S _{yy}+ \frac{1}{n^2} \ntr \Upsilon S _{yy}+ \frac{q}{n^2}\ntr \Upsilon S_{yy}\nonumber\\
= &\frac{1}{n} \ntr \Upsilon S_{yy}.  \label{17051460}
\end{align}
where in the last step we used the definition of $\Upsilon$ in (\ref{16121220}).  Similarly, we can derive
\begin{align}
\mathfrak{c}_2
= \frac{2}{n^3} \ntr S_{yy}^2\Upsilon S_{yy}\Upsilon=\frac{2}{n^3} \ntr S_{yy} (C_{yw_2}-z)^{-2}  \label{17051470}
\end{align}

Further, substituting (\ref{17051461}) and (\ref{17051460}) into (\ref{17051463}), we have
\begin{align*}
\mathbb{E} \mathfrak{H}(m,m)= \mathbb{E} \big( \mathfrak{c}_2\mathfrak{H}(m-2, m)\big)+\mathbb{E} \big( \mathfrak{c}_3\mathfrak{H}(m-1, m-1)\big)
\end{align*}
Now, by Young's inequality, we have
\begin{align*}
\mathbb{E} \mathfrak{H}(m,m)\leq \frac{(\log n)^{m}}{m} \big(\mathbb{E} |\mathfrak{c}_2|^{m}+\mathbb{E} |\mathfrak{c}_3|^{m}\big)+\frac{2(m-1)}{m} \frac{1}{(\log n)^{\frac{m}{m-1}}} \mathbb{E} \mathfrak{H}(m,m).
\end{align*}
This implies that
\begin{align}
\mathbb{E} \mathfrak{H}(m,m)\leq C\frac{(\log n)^{m}}{m} \big(\mathbb{E} |\mathfrak{c}_2|^{m}+\mathbb{E} |\mathfrak{c}_3|^{m}\big). \label{17051480}
\end{align}
Using (\ref{17051470}), we see that
\begin{align}
\mathbb{E} |\mathfrak{c}_2|^{m}\leq (\mathbb{E} \|S_{yy}\|^{2m})^{\frac{1}{2}} (\mathbb{E} (\frac{2}{n^3} \ntr |C_{yw_2}-z|^{-2})^{2m})^{\frac12}.  \label{17051477}
\end{align}
Applying (\ref{17052713}) and the fact $\tilde{s}(z)=O(1)$,  we  observe that
\begin{align}
\frac{1}{n^3} \ntr |C_{yw_2}-z|^{-2}=\frac{1}{n^2\eta} \frac{1}{n} \Im \ntr  (C_{yw_2}-z)^{-1}= \frac{c_2}{n^2\eta} \Im\tilde{s}_n(z) = O(n^{-1-\varepsilon_0})  \label{17052441}
\end{align}
on $\widetilde{D}$ with overwhelming probability. Moreover, we also have the deterministic bound
\begin{align*}
\frac{1}{n^3} \ntr |C_{yw_2}-z|^{-2}=O( \frac{1}{n^2\eta_0^2})=O(n^{-2\varepsilon_0}).
\end{align*}
Hence, we have
\begin{align}
(\mathbb{E} (\frac{1}{n^3} \ntr |C_{yw_2}-z|^{-2})^{2m})^{\frac12}\leq n^{-m(1+\frac{\varepsilon_0}{2})} \label{17051475}
\end{align}
Using (\ref{0726100}), we also have
\begin{align}
(\mathbb{E} \|S_{yy}\|^{2m})^{\frac{1}{2}}\leq K^{m} \label{17051476}
\end{align}
for some positive constant $K$.  Plugging (\ref{17051475}) and (\ref{17051476}) into (\ref{17051477}), we obtain
\begin{align*}
\mathbb{E} |\mathfrak{c}_2|^{m}\leq n^{-m(1+\frac{\varepsilon_0}{4})}.
\end{align*}
The same bound holds for $\mathfrak{c}_3$. Hence, we obtain from (\ref{17051480}) that
\begin{align*}
\mathbb{E} |\mathfrak{P}|^{2m}=\mathbb{E} \mathfrak{H}(m,m)\leq n^{-m}.
\end{align*}
Since $m$ can be arbitrary fixed positive integer, we get $|\mathfrak{P}(z)|\leq n^{-\frac{1}{2}}$ with overwhelming probability for any $z\in \widetilde{\mathcal{D}}$.

This completes the proof of Lemma \ref{lem.17043001}.
\end{proof}

With the aid of Lemma \ref{lem.17043001}, Lemma \ref{lem. local law at the right edge}, and Corollary \ref{cor. 17043020}, we prove  Lemma \ref{lem.17042902}.
\begin{proof}[Proof of Lemma \ref{lem.17042902}] For any $z\in \Omega$, we set $z_0=z+\mathrm{i}\eta_0$, where $\eta_0=n^{-1+\varepsilon_0}$.  Our strategy is to compare $M_n(z)$ with $M_n(z_0)$, and use  Lemma \ref{lem.17043001} to conclude Lemma \ref{lem.17042902}. We first show that
\begin{align}
\sup_{z\in \Omega} |( M_n)_{ij}(z_0)-(M_n)_{ij}(z)| \leq n^{-c} \label{17052801}
\end{align}
with overwhelming probability.  To see this, we first observe from (\ref{17050302}) that
\begin{align*}
M_n(z_0)- M_n(z)&=\mathcal{Z}_1 \big(\mathscr{A}(z_0)- \mathscr{A}(z)\big) \mathcal{Z}_1.
\end{align*}
For $i=0, \ldots, 4$, let $\mathcal{A}_i(z_0, z)$ be the matrices obtained via replacing $(\lambda_j,\gamma_j)$ by $(z_0,z)$ in $\mathcal{A}_i$ defined in (\ref{12163011}).  Similarly to (\ref{16123002}), we have
\begin{align*}
M_n(z_0)- M_n(z)&=\mathrm{i}\eta_0\sum_{a=0}^4\mathcal{Z}_1 \mathcal{A}_a(z_0,z) \mathcal{Z}_1.
\end{align*}

Since  $\mathcal{A}_0(z_0,z)=-I_n$, apparently we have
\begin{align*}
\eta_0 \big( \mathcal{Z}_1 \mathcal{A}_0(z_0,z) \mathcal{Z}_1\big)_{ij}=O(n^{-1+\varepsilon_0})
\end{align*}
with overwhelming probability, by simply using the definition of $\mathcal{Z}_1$ in (\ref{17050301}) and Lemma \ref{let}. Hence, it suffices to show the estimate the entries of $\eta_0 \mathcal{Z}_1 \mathcal{A}_a(z_0,z) \mathcal{Z}_1$ for $a=1,\ldots,4$.  In the sequel, we only present the details for $a=1$. The others can be handled similarly.  From the definitions of $\mathcal{Z}_1$ (c.f. (\ref{17050301})) and $\mathcal{A}_1(z_0,z)$, it suffices to estimate the entries of
\begin{align}
\eta_0 W_1 \mathcal{A}_1(z_0,z)  W_1', \qquad   \eta_0 T_1 Y \mathcal{A}_1(z_0,z)  W_1', \qquad  \eta_0 T_1 Y \mathcal{A}_1(z_0,z) Y'T_1'.  \label{17042451}
\end{align}
For the first term, observe that $W_1$ is independent of $\mathcal{A}_1(z_0,z)$, and also the event $\Omega$. Hence, by Lemma \ref{let},
 we have
\begin{align*}
\eta_0 \big(W_1 \mathcal{A}_1(z_0,z)  W_1'\big)_{ij}= \frac{\delta_{ij}(1-r_i)}{n^{2-\varepsilon_0}}   \ntr \mathcal{A}_1(z_0,z)+ O\big(\frac{(\log n)^K}{n^{2-\varepsilon_0}}\| \mathcal{A}_1(z_0,z) \|_{\text{HS}}\big)
\end{align*}
uniformly on $\Omega$, with overwhelming probability.  Note that by the definition of the domain $\Omega$ in (\ref{17052410}), we see that $\sup_{z\in \Omega} \|C_{w_2y}-z\|\leq n^{1-\varepsilon}$.  This together with (\ref{0726100}), (\ref{17050130}) and (\ref{17050120}) implies that
\begin{align*}
\frac{1}{n^{2-\varepsilon_0}}|\ntr \mathcal{A}_1(z_0,z)|&= \frac{1}{n^{2-\varepsilon_0}}|\ntr P_y W_2' \Phi(z) W_2W_2' \Phi(z_0) W_2P_y|\nonumber\\
&\leq  \frac{1}{n^{1-\varepsilon_0+\varepsilon}}\ntr |\widetilde{C}_{w_2y}-z_0|^{-1}=O(\frac{(\log n)^K}{ n^{\varepsilon-\varepsilon_0}})
\end{align*}
holds with overwhelming probability, where in the second step we bounded all the matrices except $(\widetilde{C}_{w_2y}-z_0)^{-1}$ by their operator norms.  Similarly, we have
\begin{align*}
&\frac{(\log n)^{2K}}{n^{4-2\varepsilon_0}}\| \mathcal{A}_1(z_0,z) \|_{\text{HS}}^2\leq C  \frac{(\log n)^{2K}}{n^{4-2\varepsilon_0}} \|\Phi(z)\|^2 \ntr   \Phi(z_0) \Phi(\overline{z_0})\nonumber\\
&\leq   C  \frac{(\log n)^{2K}}{n^{2-2\varepsilon_0+2\varepsilon}} \ntr |\widetilde{C}_{w_2y}-z_0|^{-2}=O(\frac{(\log n)^K}{ n^{2\varepsilon-3\varepsilon_0}})
\end{align*}
with overwhelming probability,  where in the last step we used (\ref{17052441}).  Choosing $\varepsilon>3\varepsilon_0$ (say), we obtain
\begin{align}
\eta_0\big(W_1 \mathcal{A}_1(z_0,z)  W_1'\big)_{ij}= O(n^{-\varepsilon_0}) \label{17052457}
\end{align}
with overwhelming probability.

For the entries of the last two matrices in (\ref{17042451}), we first do the approximation  by choosing $\varepsilon>3\varepsilon_0$ (say)
\begin{align}
&\eta_0T_1 Y \mathcal{A}_1(z_0,z)  W_1' =  \eta_0 \mathcal{T} G_1 \Lambda_yV_y \mathcal{A}_1(z_0,z)  W_1'+O(n^{-\varepsilon_0}),\nonumber\\
& \eta_0 T_1 Y \mathcal{A}_1(z_0,z) Y'T_1' =  \eta_0 \mathcal{T} G_1 \mathcal{A}_1(z_0,z) G_1' \mathcal{T}+ O(n^{-\varepsilon_0}). \label{17042456}
\end{align}
The proof of the above estimates is similar to (\ref{17050140}). We omit the details.   The estimates of the entries of the  first term
on the RHS of two equations in (\ref{17042456}) can be derived in the same way as (\ref{17052457}), using  Lemma \ref{let}.  We thus omit the details. Furthermore, the estimates of the entries of $\eta_0\mathcal{Z}_1 \mathcal{A}_a(z_0,z) \mathcal{Z}_1$ for $a=2,3,4$ are similar.

Hence, we completed the proof of (\ref{17052801}).  Observe that for any sufficiently small $\delta>0$, $\Omega\subset [d_+-\delta, d_++\delta]\subset [d_+-\delta, 1-\delta]$ almost surely, in light of (\ref{071501}). Hence, (\ref{17052801}) together with Lemma \ref{lem.17043001}
 implies that
 \begin{align*}
 \sup_{z\in \Omega} \sup_{i,j} |(M_n)_{ij}-\delta_{ij} m_i(z)|\leq n^{-c}
 \end{align*}
 in probability, by using the continuity of the function $m_i(z)$.  Further, when $\delta$ is sufficiently small, we also have
 \begin{align*}
 \sup_{z\in \Omega}\sup_i|m_i(z)-m_i(\delta_+)|\leq \varepsilon'
 \end{align*}
 in probability,
 by the continuity of $m_i(z)$. This concludes the proof of Lemma \ref{lem.17042902}.
 \end{proof}

\end{document}